\documentclass[12pt]{article}

\title{Torus knotted Reeb dynamics and the Calabi invariant}

\author{Jo Nelson and Morgan Weiler}
\date{}
\usepackage{subcaption}
\usepackage[normalem]{ulem}
\usepackage{amssymb}
\usepackage{latexsym}
\usepackage{amsmath}
\usepackage{amsthm}
\usepackage{eucal}
\usepackage[scr=rsfso]{mathalfa}
\DeclareFontFamily{OT1}{pzc}{}
\DeclareFontShape{OT1}{pzc}{m}{it}{<-> s * [1.10] pzcmi7t}{}
\DeclareMathAlphabet{\mathpzc}{OT1}{pzc}{m}{it}

\usepackage{graphicx, graphbox}
\usepackage[abs]{overpic}

\usepackage[all,cmtip]{xy}

\usepackage[margin=1in]{geometry}

\usepackage[matha,  mathx]{mathabx}

\usepackage{wasysym}

\usepackage{caption}
\usepackage{subcaption}
\usepackage{multirow}

\usepackage{enumerate}
\usepackage[usenames,dvipsnames,svgnames,table]{xcolor}

\usepackage{hyperref}
\hypersetup{colorlinks}

\definecolor{indigo}{RGB}{75,0,150}
\definecolor{brightpurple}{RGB}{102,0,153}
\definecolor{fuchsia}{RGB}{180,51,180}
\definecolor{jolightpurple}{RGB}{188,171,240}
\definecolor{blb}{RGB}{153, 102, 51}
\definecolor{bdb}{RGB}{77, 38, 0}
\definecolor{bg}{RGB}{85, 85, 94}

\hypersetup{colorlinks,
linkcolor=indigo,
filecolor=brightpurple,
urlcolor=brightpurple,
citecolor=fuchsia}

\normalbaselines
\usepackage{tikz}

\usepackage[abs]{overpic}

\newcommand{\mc}[1]{{\mathcal #1}}

\numberwithin{equation}{section}
\numberwithin{figure}{section}

\newtheorem{theorem}{Theorem}[section]
\newtheorem{proposition}[theorem]{Proposition}
\newtheorem{corollary}[theorem]{Corollary}
\newtheorem{lemma}[theorem]{Lemma}
\newtheorem{lemma-definition}[theorem]{Lemma-Definition}

\theoremstyle{definition}
\newtheorem{definition}[theorem]{Definition}
\newtheorem{remark}[theorem]{Remark}

\newtheorem{example}[theorem]{Example}

\newtheorem{notation}[theorem]{Notation}

\renewcommand{\frak}{\mathfrak}

\newcommand{\C}{{\mathbb C}}
\newcommand{\CP}{{\mathbb C}{\mathbb P}}
\newcommand{\Q}{{\mathbb Q}}
\newcommand{\R}{{\mathbb R}}

\newcommand{\Z}{{\mathbb Z}}

\newcommand{\cpq}{\mathbb{CP}^1_{p,q}}
\newcommand{\ds}{{\dot{\Sigma}}}

\newcommand{\ve}{\varepsilon}

\newcommand{\fp}{\mathfrak{p}}

\newcommand{\D}{\mathbf{D}}

\newcommand{\cO}{{\Sigma}}

\newcommand{\op}{\operatorname}

\newcommand{\M}{\mc{M}}

\newcommand{\ind}{\op{ind}}

\newcommand{\Ker}{\op{Ker}}

\newcommand{\CZ}{\op{CZ}}

\newcommand{\bpm}{\begin{pmatrix}}
\newcommand{\epm}{\end{pmatrix}}

\newcommand{\fb}{\mathcal{F}_{b} }
\newcommand{\af}{\substack{\mathcal{A} < L \\ \mathcal{F}_{b}  \leq K}}

\newcommand{\afe}{\substack{\mathcal{A} < L(\varepsilon) \\ \mathcal{F}_{b}  \leq K}}

\newcommand{\vol}{\op{vol}}
\newcommand{\V}{\mathcal{V}}
\newcommand{\A}{\mathcal{A}}
\newcommand{\cur}{\mathcal{C}}

\renewcommand{\P}{\mathcal{P}}
\newcommand{\rot}{\op{rot}}
\newcommand{\Sdg}{\dot\Sigma_g}
\newcommand{\zb}{\mathpzc{b}}
\newcommand{\zp}{\mathpzc{p}}
\newcommand{\zq}{\mathpzc{q}}
\newcommand{\zh}{\mathpzc{h}}

\begin{document}

\maketitle

\begin{abstract}

We establish quantitative existence, action, and linking bounds for all Reeb orbits associated to any contact form on the standard tight three-sphere admitting the standard transverse positive $T(p,q)$ torus knot as an elliptic Reeb orbit with canonically determined rotation numbers. This can be interpreted through an ergodic lens for Reeb flows transverse to a surface of section.  Our results also allow us to deduce an upper bound on the mean action of periodic orbits of naturally associated classes of area preserving diffeomorphisms of the associated Seifert surfaces of genus $(p-1)(q-1)/2$ in terms of the Calabi invariant, without the need for genericity or Hamiltonian hypotheses.   Our proofs utilize knot filtered embedded contact homology, which was first introduced and computed by Hutchings for the standard transverse unknot in the irrational ellipsoids and further developed in our previous work.   We continue our development of nontoric methods for embedded contact homology and establish the knot filtration on the embedded contact homology chain complex of the standard tight three-sphere with respect to positive $T(p,q)$ torus knots, where there are nonvanishing differentials.  We also obtain obstructions to the existence of exact symplectic cobordisms between positive transverse torus knots.
\end{abstract}

\tableofcontents

\section{Introduction}

In this paper we establish some ergodic type results for positive torus knotted Reeb dynamics on the tight three sphere.  These are obtained via our computation of the knot filtration on the embedded contact homology chain complex with respect to the standard transverse positive torus knot $T(p,q)$, realized as a Reeb orbit for any contact form on the tight three sphere, having rotation number $pq + \delta$, where $\delta$ is either zero or a sufficiently small positive irrational number.  From this, we obtain obstructions to symplectic cobordisms between many positive transverse torus knots.

We additionally establish quantitative existence results for Reeb orbits of such contact forms and for periodic orbits of area preserving surface maps homotopic to the Nielsen-Thurston representative of the monodromy of the open book decomposition of $S^3$ along $T(p,q)$ composed with a small positive boundary-parallel twist. Our results on Reeb dynamics are action-linking bounds which do not require a genericity hypothesis. They are motivated by the action-linking estimates proved in the generic case by Bechara Senior-Hryniewicz-Saloma\~{o} \cite{bhs}.  Our dynamical results for surfaces rely on a dictionary between Reeb dynamics and surface dynamics.  In particular, we show that if the Calabi invariant of an exact symplectomorphism of a genus $(p-1)(q-1)/2$ surface with one boundary component of the aforementioned type satisfies a particular upper bound in terms of $p$ and $q$ then it has a periodic orbit whose mean action is bounded above in terms of the Calabi invariant.

{Knot filtered embedded contact homology, a topological spectral invariant, was introduced by Hutchings in \cite{HuMAC}} and computed with respect to the standard transverse unknot in an irrational ellipsoid.  From this Hutchings related the mean action of area preserving disk maps which are rotation near the boundary to their Calabi invariant. It was subsequently computed for the nondegenerate Hopf link in the lens spaces $L(n,n-1)$ using a toric perturbation by Weiler \cite{weiler} and used to {relate the mean action of area preserving annulus diffeomorphisms, subject to a boundary condition, to their Calabi invariant.} 
In both cases, computing the knot filtration required irrational rotation numbers for the Reeb orbits realizing the {knot or} link, although the corresponding surface dynamics results did not.

We previously generalized the construction and invariance of knot filtered embedded contact homology so that it can be defined and computed via Morse-Bott methods with respect to degenerate Reeb orbits with rational rotation numbers \cite{kech}.  We computed the knot filtration for the positive $T(2,q)$ knots, by way of an ECH chain complex with no differential; this cannot be used when $p\neq 2$. The current paper extends our work to provide Morse-Bott computational methods for the positive $T(p,q)$ knot filtration, where there are nonvanishing differentials.

\subsection{Overview of embedded contact homology}\label{ss:overviewech}

Let $Y$ be a closed three-manifold with a contact form $\lambda$.  Let $\xi=\ker(\lambda)$ denote the associated contact structure. The \emph{Reeb vector field} $R$ associated to $\lambda$ is uniquely determined by 
\[
\lambda(R)=1, \ \ \ d\lambda(R, \cdot)=0.
\]
We say that an almost complex structure $J$ on $\R_s\times Y$ is {\em $\lambda$-compatible\/} if $J(\xi)=\xi$, $d\lambda(v,Jv)>0$ for nonzero $v\in\xi$, $J$ is invariant under translation of the $\R$ factor, and $J(\partial_s)=R$.

A (periodic) {\em Reeb orbit\/} is a map $\gamma:\R/T\Z\to Y$ for some $T>0$ such that $\gamma'(t)=R(\gamma(t))$, modulo reparametrization.    A Reeb orbit is said to be \textit{embedded} (equivalently, \textit{simple}) whenever this map is  injective.  Let $\mathcal{P}(\lambda)$ denote the set of embedded Reeb orbits.  After a symplectic trivialization of $\xi|_\gamma$,  the linearized Reeb flow for time $T$ along each Reeb orbit defines a symplectic linear map
\begin{equation}
\label{eqn:lrt}
P_\gamma{(T)} :(\xi_{\gamma(0)},d\lambda) \longrightarrow (\xi_{\gamma(0)},d\lambda).
\end{equation}
We say that a Reeb orbit $\gamma$ is {\em nondegenerate\/} whenever $P_\gamma$ does not have $1$ as an eigenvalue. If all the Reeb orbits in $\mathcal{P}(\lambda)$ are nondegenerate then the contact form $\lambda$ is said to be \textit{nondegenerate}; a generic contact form is nondegenerate.  Fix a nondegenerate contact form for the below discussions.

 A nondegenerate Reeb orbit $\gamma$ is {\em elliptic\/} if $P_\gamma$ has eigenvalues on the unit circle and {\em  hyperbolic\/} if $P_\gamma$ has real eigenvalues.  The {\em Conley-Zehnder index\/} $\CZ_\tau(\gamma)\in\Z$ is a winding number associated to the induced path of symplectic linear matrices depending on the homotopy class $\tau$ of symplectic trivializations.  The parity of the Conley-Zehnder index does not depend on the choice of trivialization $\tau$.  Given an embedded Reeb orbit $\gamma$, the Conley-Zehnder index is even when $\gamma$ is positive hyperbolic and odd otherwise.

Embedded contact homology (ECH) is a three dimensional gauge theory due to Hutchings \cite{Hindex, lecture}.  Let $(Y,\lambda)$ be a closed nondegenerate contact three-manifold $Y$ and let $J$ be a generic $\lambda$-compatible almost complex structure. The \emph{embedded contact homology chain complex} is defined with respect to a fixed homology class $\Gamma \in H_1(Y)$ as the $\Z/2$ vector space\footnote{ECH may be defined with integer coefficients, see \cite[\S9]{obg2}, but we do not require more than $\Z/2$ for this paper.} freely generated by {admissible Reeb currents}.  An \emph{admissible Reeb current} is a finite set of pairs $\alpha = \{ (\alpha_i, m_i) \}$, such that the $\alpha_i$ are distinct embedded Reeb orbits, the $m_i$ are positive integers, the total homology class $\sum_i m_i[\alpha_i] = \Gamma$, and $m_i=1$ whenever $\alpha_i$ is hyperbolic.  We frequently make use of the multiplicative notation, $\alpha = \prod_i \alpha_i^{m_i}$, and we also refer to Reeb currents as Reeb orbit sets.

Given Reeb currents $\alpha$ and $\beta$, a \textit{$J$-holomorphic current} from $\alpha$ to $\beta$ is a finite set of pairs $\cur = \{(C_k,d_k)\}$, where the $C_k$ are distinct irreducible somewhere injective $J$-holomorphic curves in $(\R_s\times Y, d(e^s\lambda))$ and the $d_k$ are positive integers, subject to the asymptotic condition that $C_k$ converges as a current to $\alpha = \sum_i m_i \alpha_i$ as $s\to+\infty$ and to $\beta = \sum_j n_j \beta_j$ as $s\to-\infty$.  Here ``convergence as a current" means that we only keep track of the total algebraic multiplicity of the covers over the ends as well as the underlying curve{, and not the connectedness of the covers}.  We now stipulate that $\alpha$ and $\beta$ are admissible Reeb currents, e.g. ECH generators.

Let $\M_1(\alpha,\beta,J)$ denote the set of $J$-holomorphic currents from $\alpha$ to $\beta$ with ECH index $1$, modulo the usual equivalence of $J$-holomorphic curves up to biholomorphism of the domain.  The ECH index is a topological index, which depends only on $\alpha, \beta$ and $Z \in H_2(Y,\alpha, \beta)$ and is discussed further in \S\ref{s:generalities}. The ECH index bounds the Fredholm index of the underlying somewhere-injective $J$-holomorphic curves from above, and governs the dimension of the moduli space of $J$-holomorphic currents.  Moreover, the ECH index provides a relative $\Z/d$ grading of the chain complex and its homology, where $d$ denotes the divisibility of $c_1(\xi) + 2 \op{PD}(\Gamma) \in H^2(Y;\Z)$ mod torsion.   

In particular, given two admissible Reeb currents $\alpha$ and $\beta$, their index difference $I(\alpha,\beta)$ is defined by choosing an arbitrary $Z\in H_2(Y,\alpha,\beta) $ and setting 
\[
I(\alpha,\beta) = [I(\alpha,\beta,Z)] \in \Z/d,
\]
which is well-defined by the index ambiguity formula \cite[\S 3.3]{Hindex}. {When $\Gamma=0$ and $H_2=0$, as is the case in this paper, there is an absolute grading by setting the index of the empty set to be zero. Moreover, there is always an absolute grading on ECH by homotopy classes of oriented two-plane fields due to Hutchings \cite{Hrevisit}, corresponding to a similar grading on Seiberg-Witten Floer cohomology due to Kronheimer-Mrowka \cite{KMbook} under Taubes' isomorphism (discussed below), and to one on Heegaard Floer homology defined by Huang-Ramos in \cite{hr} under Colin-Ghiggini-Honda's isomorphism \cite{cgh}-\cite{cghiii}. These correspondences were respectively proved by Cristofaro-Gardiner in \cite{CGgradings} and Ramos in \cite{Rgradings}.}

The \emph{ECH differential} is given by 
\[
\partial \alpha = \sum_\beta \#_2 \left( \M_1(\alpha,\beta,J)/\R \right) \beta,
\]
the count in $\Z/2$ of ECH index 1 currents in $\R \times Y$, modulo $\R$ translation in the range.  (The additivity property of the ECH index ensures that the differential decreases the grading by one.)

The definition of the ECH index is the key nontrivial part of ECH \cite{Hindex}, and under the assumption that $J$ is generic, guarantees that the curves of ECH index 1 are embedded, except possibly for multiple covers of \emph{trivial cylinders} $\R \times \gamma$ where $\gamma$ is a Reeb orbit.

 The proof that the differential squares to zero requires some serious expertise with obstruction bundle gluing and is established by Hutchings and Taubes \cite{obg1,obg2}.   Let $ECH_*(Y,\lambda,\Gamma,J)$ denote the homology of the ECH chain complex.  The homology does not depend on the choice of $J$ or on the contact form $\lambda$ for $\xi$, and so defines a well-defined $\Z/2$-module $ECH_*(Y,\xi,\Gamma)$.  Invariance of embedded contact homology is established by Taubes' isomorphism \cite{taubesechswf}-\cite{taubesechswf5} with the `from' version of monopole Seiberg-Witten Floer cohomology $\widehat{HM}^*$; the latter is rigorously detailed in the book by Kronheimer and Mrowka \cite{KMbook}.  
 
 \begin{theorem}[Taubes]
 If $Y$ is connected, then there is a canonical isomorphism of relatively graded $\Z[U]$-modules 
\[
ECH_*(Y,\lambda,\Gamma,J) \simeq \widehat{HM}^{-*}(Y,\mathfrak{s}_\xi + \op{PD}(\Gamma)),
\]
which sends the ECH contact invariant $c(\xi):=[\emptyset] \in ECH(Y,\xi,0)$ to the contact invariant in Seiberg-Witten Floer cohomology.
 \end{theorem}
Here $\mathfrak{s}_\xi$ denotes the canonical spin-c structure determined by the oriented 2-plane field $\xi$.  The contact invariant in Seiberg-Witten Floer cohomology was discovered and implicitly defined in \cite{km-contact} and further explanations are given in \cite[\S 6.3]{kmos} and \cite{mariano}. {See Remark \ref{rmk:U} for context on the $\Z[U]$-module structure in ECH.} Taubes' isomorphism additionally shows that ECH is a topological invariant of $Y$, except that one needs to shift $\Gamma$ when changing the contact structure.   
 
There are two filtrations on ECH respected by the differential, which give rise to distinct spectral invariants.  The filtration by symplectic action enables us to compute ECH via successive approximations, similarly to \cite[\S 3.4, \S 7.1]{preech} and \cite[\S 6-7]{kech}. The \emph{symplectic action} or \emph{length} of a Reeb orbit set $\alpha=\{(\alpha_i,m_i)\}$ is
\begin{equation}\label{eq:action}
\mathcal{A}(\alpha):=\sum_im_i\int_{\alpha_i}\lambda.
\end{equation}
If $J$ is $\lambda$-compatible and there is a $J$-holomorphic current from $\alpha$ to $\beta$, then $\mathcal{A}(\alpha)\geq\mathcal{A}(\beta)${. This follows from} Stokes' theorem, because $d\lambda$ is an area form for $J$-holomorphic curves. The differential $\partial$ counts $J$-holomorphic currents, hence it decreases symplectic action:
\[
\langle\partial\alpha,\beta\rangle\neq0\Rightarrow\mathcal{A}(\alpha)\geq\mathcal{A}(\beta).
\]

Let $ECC_*^L(Y,\lambda,\Gamma,J)$ denote the subcomplex of $ECC_*(Y,\lambda,\Gamma,J)$ generated by Reeb currents of symplectic action less than $L$, and let $\partial^L$ denote the restriction of $\partial$ to $ECC_*^L$ when we want to emphasize its restriction.  The homology $ECH_*^L(Y,\lambda,\Gamma)$ of $ECC_*^L(Y,\lambda,\Gamma,J)$  is independent of $J$ as a result of \cite[Theorem 1.3]{cc2}, {and we refer to it} as \emph{action filtered ECH}.

The second filtration is by linking and defined by Hutchings in \cite{HuMAC}. Let $b$ be an embedded elliptic Reeb orbit (its image is a transverse knot for $\xi$) with rotation number $\rot(b)$. This rotation number is determined by the linearized Reeb flow with respect to a topological Seifert framed trivialization, namely one in the homotopy class in which a constant pushoff of $b$ has linking number zero with $b$. See Remarks \ref{rem:linkingtriv} and \ref{rem:c1sl} for further discussion of this trivialization.

 The \emph{knot filtration} is the function
\[
\mathcal{F}_b(b^m\alpha) := m \rot(b) + \ell(\alpha,b),
\]
where $\ell(\alpha, b)=\sum_im_i\ell(\alpha_i,b)$. For nondegenerate $b$ (i.e., $\rot(b)\in\R\setminus\Q$), the function $\mathcal{F}_b$ will not be integer valued, but it will still take values in a discrete set of nonnegative real numbers if $\rot(b)>0$ and all other Reeb orbits link positively with $b$, or if $\rot(b)<0$ and all other Reeb orbits link negatively with $b$. The former is the setting we will consider in this paper. As proved in \cite{HuMAC}, the function $\mathcal{F}_b$ is not increased by the ECH differential, and thus it filters the ECH chain complex.

One significant feature of knot filtered ECH, proved in \cite[Thm.~6.10]{HuMAC}, is that it is a topological spectral invariant, depending only on $(Y,\xi)$, the Reeb orbit $b$, its rotation number $\rot(b)$, and the (positive) filtration level $K$. This is in contrast to action filtered ECH, which is extremely sensitive to the contact form. Thus we are justified in denoting knot filtered ECH by $ECH_*^{\fb \leq K}(Y,\xi,b,\rot(b))$. In \cite[Thm.~1.5]{kech}, we extended the definition and invariance properties to include rational rotation numbers, which encompass Morse-Bott direct limit style computations. The details are reviewed in \S \ref{ss:kECH}.

Similarly to the notion of the ECH action spectrum, one can define the ECH linking spectrum, which encodes knot filtered embedded contact homology by a sequence of real numbers, as observed by Hutchings.  In particular, when $Y$ is $S^3$ equipped with the standard tight contact structure we define
\[
c_k^{\op{link}}(b,\rot(b)) \in [-\infty, \infty)
\]
to be the infimum over $K$ such that the degree $2k$ generator of $ECH(S^3,\xi_{std})$ is in the image of the inclusion induced map \eqref{eq:kechi1} if $\rot(b)$ is irrational or in the image of the inclusion induced map   \eqref{eq:kechi2} if $\rot(b)$ is rational.  Further details and properties are given in \S \ref{ss:kECH} and Definition \ref{def:cklink}. 

\begin{remark}[Comparison with knot embedded contact homology]\label{rem:ECK}
Knot filtered embedded contact homology is distinct from knot embedded contact homology $ECK$.   We expect that by (understanding what is involved in) taking the rotation number of the transverse knot to $-\infty$,  $ECK$ can be recovered from knot filtered embedded contact homology.  For further discussion, we refer the reader to \cite[Rem.~1.3]{kech}
\end{remark}

\subsection{Positive torus knot filtered ECH of $(S^3,\xi_{std})$}\label{ss:ECHcalc}

We now present our main results on Reeb dynamics. First we introduce the contact forms and knots which will be central to our computations, then we state and discuss our results, and finally we outline the relevant work in \S\ref{s:calabi}-\S\ref{s:spectral}.

The positive torus knot $T(p,q)$ in the three-sphere in $\C^2$ can be algebraically defined by
\[
T(p,q)=\left\{(z_1,z_2)\in S^3 \subset \C^2 \ |\ z_1^p+z_2^q=0\right\}.
\]
By realizing $S^3$ as the unit sphere in $\C^2$, we can associate to it the contact form
\[
\lambda_0 = \frac{i}{2}  \left( z_1 d\bar{z_1} - \bar{z_1} dz_1 + z_2 d\bar{z_2} - \bar{z_2} dz_2 \right),
\]
whose kernel is the two-plane field
\[
\left(\xi_{std}\right)\vert_p = T_pS^3 \cap J_{\C^2}(T_pS^3),
\]
which is the standard tight contact structure on $S^3$; here $J_{\C^2}$ denotes the standard complex structure. As indicated by the notation, up to homotopy this is the standard contact structure on $S^3$.

Our choice of equation defining $T(p,q)$ is motivated by the Milnor fibration from $S^3 \setminus T(p,q)$ to $S^1$, detailed in Example \ref{ex:Tpq}. For the contact form $\lambda_0$, the positive $T(p,q)$ torus knot is not a Reeb orbit, but only a transverse knot for $\xi_{std}$. Our $T(p,q)$ is determined as a transverse knot by its maximal self-linking number $pq-p-q$, a consequence of Etnyre's work in \cite{torus}, and so we refer to it as ``standard." See Remark \ref{rmk:tknot} for further discussion, and the subsequent review of how to realize $T(p,q)$ as an elliptic Reeb orbit.

\begin{remark}
In the paper we restrict to positive torus knots (and unknots), meaning that $p,q \in \Z_{\geq 1}$ and $\gcd(p,q)=1$.  Our results for the unkots provide rotation angles which are whole numbers, which complement the results Hutchings previously established for irrational rotation angles in the irrational tight ellipsoid \cite[Prop.~5.5]{HuMAC}.  For our surface dynamics results, we additionally assume $p,q \in \Z_{\geq 2}$, as the disk was previously considered in \cite{HuMAC}.  We expect the analogous Reeb dynamics results to hold for torus links, with corresponding surface dynamics results applying to surfaces with additional boundary components.  
\end{remark}

The knot filtration for $T(p,q)$ is closely related to the sequence $N(p,q)$ for $a,b\in\R$, which we now define. Write all nonnegative integer linear combinations of $p$ and $q$, with multiplicities, in nondecreasing order. Then $N(p,q)$ refers to this entire sequence, while $N_k(p,q)$ refers to its $k^\text{th}$ element, including multiples, which will begin to occur at their LCM. We always start with $N_0(a,b)=0$. The following theorem contains our main result on Reeb dynamics, from which we deduce our dynamical applications. 

\begin{theorem}\label{thm:kech-intro}
Let $\xi_{std}$ be the standard tight contact structure on $S^3$.  
Let ${b}$ be the standard positive transverse $T(p,q)$ torus knot for $(p,q)$ relatively prime.  If $\rot({b}) = pq,$ then
\[
ECH_{2k}^{\fb \leq K}(S^3,\xi_{std},T(p,q),pq)=\begin{cases}\Z/2&K\geq{ N_k(p,q)  },
\\0&\text{otherwise,}\end{cases}
\]
and in all other gradings $*$,
\[
ECH_*^{\fb \leq K}(S^3,\xi_{std},T(p,q),pq)=0.
\]
If $\rot(b) = pq+\delta$, where $\delta$ is a sufficiently small positive irrational number, then up to grading $*$ and filtration threshold $K$ inversely proportional to $\delta$,
\[
ECH_{2k}^{\fb \leq K}(S^3,\xi_{std},T(p,q),pq+\delta)=\begin{cases}\Z/2&K\geq{ N_k(p,q) + \delta {(\$N_k(p,q) -1)}},
\\0&\text{otherwise,}\end{cases}
\]
where { $\$N_k(p,q)$ is the number of repeats in $\{ N_j(p,q)\}_{j\leq k}$ with value $N_k(p,q)$,} and in all other gradings $*$, up to the threshold inversely proportional to $\delta$,
\[
ECH_*^{\fb \leq K}(S^3,\xi_{std},T(p,q),pq+\delta)=0.
\]
\end{theorem}

\begin{remark}\label{rmk:kd}
We explain the role of the adjective ``sufficiently" in the preceding theorem statement. In order to avoid destroying the structure of the $N(p,q)$ sequence, $\delta$ must be small enough with respect to $k$ and $K$. In particular, we must have $N_k(p,q)+\delta(\$N_k(p,q)-1)  \leq N_\ell(p,q)$ for all $k$, where $N_\ell(p,q)$ is the smallest entry in $N(p,q)$ larger than $N_k(p,q)$ (here, $K$ is $N_k(p,q)$).  Additionally, in the proof of Lemma \ref{lem:d0g0}, we utilize the bound $\delta < \min(1/p,1/q)$. See Remark \ref{rem:delta} for a longer discussion.
\end{remark}

\begin{remark}\label{rem:kechconj-intro}
Theorem \ref{thm:kech-intro} and the known calculations for the unknot in \cite[Prop.~5.5]{HuMAC}, are consistent with the conjecture that for any transverse knot $b \subset S^3$, 
\[
\lim_{k \to \infty} \frac{c_k^{\op{link}}(b,\rot(b))^2}{2k} = \rot(b).
\]
\end{remark}

Before explaining our results further, we mention a couple of corollaries of our above computation of positive torus knot filtered ECH.  The first yields obstructions to the existence of symplectic cobordisms between transverse positive torus knots.  (For further discussion and results for relative symplectic cobordisms to the empty set, see the work of Etnyre and Golla \cite{eghats}.) 
A \textit{symplectic cobordism} from a transverse knot ${b}_+$ to a transverse knot ${b}_-$, both in $\xi_{std}$, with rotation numbers $\theta_1>\theta_2$, is a symplectic submanifold $S\subset\R\times Y$ which is asymptotic to the ${b}_\pm$ as the real coordinate goes to $\pm\infty$. By combining \cite[Rmrk.~7.1]{HuMAC}, \cite[Thm.~1.5]{kech}, Theorem \ref{thm:kech-intro}, with the definition and properties of $c_k^{\op{link}}(b,\rot(b))$ described in Definition \ref{def:cklink}, we obtain the following.

\begin{corollary}\label{cor:scob}
Let $pq \geq p'q'$ and $T(p,q)$ and $T(p',q')$ be positive torus knots in $S^3$. If there is a symplectic cobordism from $T(p,q)$ to $T(p',q')$ in $\R \times S^3$ then $N_k(p,q) \geq N_k(p',q')$ for all $k$.
\end{corollary}
We expect that more obstructions may be obtained by understanding how knot filtered ECH changes with respect to different choices of rotation angles, which in turn correspond to self-linking numbers, for a fixed transverse knot.

We also use Theorem \ref{thm:kech-intro} to prove our main quantitative existence results for Reeb orbits of $(S^3,\xi_{std})$ in \S\ref{s:spectral}.
\begin{theorem}\label{thm:toruslinking} 
Let $\lambda$ be a contact form on $(S^3,\xi_{std})$ with $\vol(\lambda)=V$.  Suppose that the Reeb vector field $R_\lambda$ admits the positive $T(p,q)$ torus knot, denoted by ${b}$, as an elliptic Reeb orbit with symplectic action 1 and rotation number $pq + \Delta$, where $\Delta$ is a positive irrational number.   If $V < \dfrac{pq}{(pq+\Delta)^2}$ then
\[
\inf \left\{ \frac{\A(\gamma)}{\ell(\gamma,b)} \ \bigg \vert \ \gamma \in \mathcal{P}(\lambda) \setminus \{ b\} \right\} \leq \sqrt{\frac{V}{pq}}.
\] 
\end{theorem}
We call Theorem \ref{thm:toruslinking} a ``quantitative existence result" as its conclusion asserts the existence of a Reeb orbit: were no such orbit to exist, the infimum on the left hand side would be infinite.

\begin{remark}\label{rem:stronger}
One would prefer the stronger estimate
\[
\inf \left\{ \frac{\A(\gamma)}{\ell(\gamma,b)} \ \bigg \vert \ \gamma \in \mathcal{P}(\lambda) \setminus \{ b\} \right\} \leq \sqrt{\frac{V}{pq+\Delta}},
\]
under the weaker assumption that $V < \frac{1}{pq + \Delta}$. One would expect this to follow from a proof of the conjecture in Remark \ref{rem:kechconj-intro}. However, the proof of Theorem \ref{thm:kech-intro} requires computing the knot filtration for $k$ arbitrarily large given a fixed $\Delta$, which is not presently understood. See Remarks \ref{rmk:kd} and \ref{rem:delta}. 
To prove Theorem \ref{thm:kech-intro} we must therefore send $\Delta\to0$, providing a weaker lower bound counteracted by the hypothesis $V<pq/(pq+\Delta)^2$, as elucidated in Remark \ref{rmk:Vpqd}. We explain how to obtain the stronger estimate in the hypothetical scenario in which we know we have computed the knot filtration for $k$ as large as needed in Remark \ref{rmk:deltadenom}.
\end{remark}

One may be led to wonder about the existence of contact forms with knotted Reeb orbits.  In work by Etnyre and Ghrist \cite[Theorem 3.6]{eg}, it is shown that there is a tight contact form on $S^3$ whose Reeb vector field posesses periodic orbits of all possible knot and link types simultaneously.   However, it is not true that all transverse knot types are realized in their constructions.  As indicated by Ghrist, Holmes, and Sullivan \cite[Chapter 3]{ghs}, it is likely that some simple knot types may admit only ``complicated presentations," and as a result have arbitrarily negative self-linking number.   There are also natural relations to other dynamical properties.  For example, the work of Franks and Williams \cite{fw} demonstrates that under certain conditions, positive topological entropy forces the existence of infinitely many knot types in $S^3$. Methods for constructing flows (of vector fields more general than Reeb orbits) realizing knots in $S^3$ via braiding periodic orbits appear in the survey \cite{frankssullivan}.

\medskip

We now outline our method of proof for Theorem \ref{thm:kech-intro}. In order to compute the knot filtration, we must understand the ECH chain complex of contact forms with $T(p,q)$ as an elliptic orbit.  For this purpose, we continued our development of new non-toric methods, continuing our work from \cite{preech, kech}.

Since all orbit sets are homologous in $S^3$, a full understanding of the chain complex, including its differential and knot filtration, requires being able to compute the ECH index of any arbitrary pair of generators. This is combinatorially nontrivial. {(For prequantization bundles as in \cite{preech}, the ordinary homology separated most of the ECH generators, which made these calculations considerably less involved.)}

Using a clever perturbation, which allowed us to conclude that the differential vanished for index reasons, the case $p=2$ was handled entirely combinatorially in \cite{kech}. This is in contrast to the case of general $p$ studied in this paper (also encompassing $p=2$) which necessitates an alternate perturbation of the contact form, resulting in a nonvanishing differential.  We now sketch how we establish the result, indicating where various arguments appear in the paper.

\subsubsection*{\S\ref{s:calabi}-\ref{s:topology}: Contact forms and Reeb vector fields}
We introduce our contact forms using the open book decompositions of $S^3$ with positive torus knot binding $T(p,q)$, via the explicit correspondence between abstract open books with periodic monodromy and Seifert fibrations generated by the Reeb vector field of a supported contact form.  These constructions were observed by Colin and Honda \cite{CH} by way of Lisca and Mati\'c \cite{lm} and detailed further in Example \ref{symporb} and Remark \ref{rem:obd}. To provide further explanation, recall that there is a Seifert fibration on $S^3$ given by
\[
Y\left(0; -1; (p, p-m), (q, n)\right),
\]
where $m, n \in \Z$ such that $qm-pn=1$, and whose fibers are positive transverse $T(p,q)$ knots.  The derivative of the $S^1$-action can be realized as the Reeb vector field for a contact form defining the standard tight contact structure, which is positively transverse to the Seifert fibration.  These contact forms are strictly contactomorphic to the contact form arising from the connection 1-form on the associated prequantization orbibundle as a result of Kegel and Lange \cite{kl} and Cristofaro-Gardiner and Mazzucchelli \cite{cgm}. The prequantization orbibundle is denoted by $\frak{p}: S^3 \to \CP_{p,q}^1$, and it has Euler class $-\frac{1}{pq}$.  We are using $\CP^1_{p,q}$ to denote the one dimensional weighted complex projective space that is the quotient of the unit sphere in $\C^2$ by the almost free action
\[
e^{{2\pi}it} \cdot (z_1,z_2) = \left(e^{p{2\pi}it}z_1,e^{q{2\pi}it}z_2 \right)
\]
of $S^1\subset\C$. The contact form
\[
\lambda_{p,q} = \frac{\lambda_0}{p|z_1|^2 + q|z_2|^2}
\]
is invariant under this action, and its differential $d\lambda_{p,q}$ projects to an orbifold symplectic form on $\CP^1_{p,q}$. Its Reeb orbits form the fibers of the prequantization orbibundle as well as the Seifert fibration. The pages of the open book are used to induce the topological Seifert framed push off linking zero type trivialization along the $T(p,q)$ binding, which is used to compute the knot filtration.  The pages also play a supporting role in finding explicit surfaces and trivializations used in our computation of the ECH index.  

\begin{figure}[h]
 \begin{center}
 \begin{overpic}[width=.5\textwidth, unit=1.75mm]{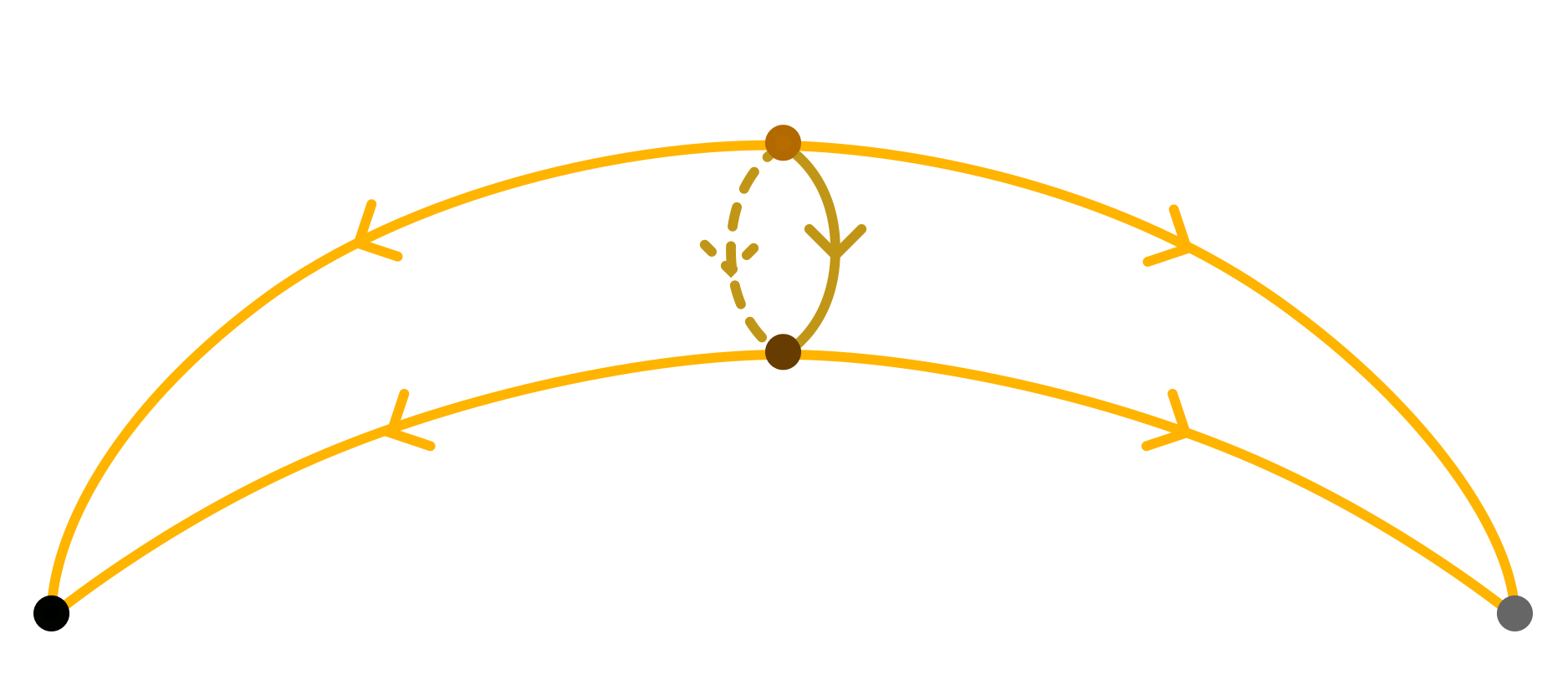}
\put(1,0){$\zp$}
\put(45,0){\textcolor{bg}{$\zq$}}
\put(22.5,18){ \textcolor{blb}{$\zb$}}
\put(22.5,7){\textcolor{bdb}{$\zh$}}
\end{overpic}
\end{center}
\caption{ Here we have depicted the gradient flow of $\mathpzc{H}_{p,q}$ on $\CP^1_{p,q}$ for $p\neq2$ and labeled the fibers of the prequantization orbibundle projecting to the respective critical points. The minima have isotropy $\Z/p$ and  isotropy \textcolor{bg}{$\Z/q$}; the fibers $\zp$ and $\textcolor{bg}{\zq}$ project to the respective minima. The fiber \textcolor{bdb}{$\zh$} projects to the index 1 critical point. The binding fiber \textcolor{blb}{$\zb$} projects to the maximum. }

\label{fig:banana}
\end{figure}

Then we perturb $\lambda_{p,q}$, which is degenerate, using a lift of an orbifold Morse-Smale function from the base orbifold $\CP^1_{p,q}$, to obtain a nondegenerate form. This is similar to the perturbations from \cite{leo, jo2, preech}. We obtain
\begin{equation}\label{eq:intropert}
\lambda_{p,q,\varepsilon}:=  (1+\varepsilon \frak{p}^*\mathpzc{H}_{p,q})\lambda_{p,q}.
\end{equation}

In Figure \ref{fig:banana} we have depicted representative gradient flow lines of the `banana' orbifold Morse function $\mathpzc{H}_{p,q}$ on $\CP^1_{p,q}$. Above the critical points of $\mathpzc{H}_{p,q}$, the Reeb vector field of $\lambda_{p,q,\varepsilon}$ has the following four embedded orbits.   
 \begin{itemize}
 \itemsep-.25em
\item The embedded orbits ${\mathpzc{p}}$ and ${\mathpzc{q}}$ are the respective singular fibers of the Seifert fibration; these project to minima at the orbifold points of respective isotropy $\Z/p$ and $\Z/q$. Both are elliptic.
\item The binding $ \mathpzc{b}$ is an embedded orbit; it is a regular $T(p,q)$ knotted fiber orbit that projects to the nonsingular maximum. It is elliptic.
\item Another regular $T(p,q)$ knotted fiber $\mathpzc{h}$ is an embedded Reeb orbit; it projects to a nonsingular index one critical point. It is positive hyperbolic.
\end{itemize}
All other Reeb orbits of $\lambda_{p,q,\varepsilon}$ have very large action, proportional to $1/\varepsilon$. Filtering the ECH chain complex by action, we fix an action $L$ and choose $\varepsilon$ so that the generators of $ECC_*^{L}(S^3,\lambda_{p,q,\varepsilon(L)})$ are only those admissible Reeb currents the form $\mathpzc{b}^B\mathpzc{h}^H\mathpzc{p}^P\mathpzc{q}^Q$. In this notation, we allow $B,P,Q \in \Z_{\geq 0}$ and $H=0,1$ so that the current is admissible. We obtain a direct system $\{ ECH^L_*(Y,\lambda_{p,q,\varepsilon(L)})\}$ whose direct limit 
\[
\lim_{L\to\infty}ECH^L_*(S^3,\lambda_{p,q,\varepsilon(L)})
\]
is the homology of the chain complex generated by orbit sets whose underlying embedded Reeb orbits all map to the critical points of $\mathpzc{H}_{p,q}$.

 \subsubsection*{\S \ref{s:ECHI}: The differential}
{In \S \ref{s:ECHI} we carry out the action filtered Morse-Bott means of computing embedded contact homology by leveraging our understanding of the cylindrical portion of the differential using a non-generic $S^1$-invariant almost complex structure.  The methods we employed for curve counts in  prequantization bundles as in \cite[\S 4-6]{preech}, \cite{farris}, cannot be directly used. As a result, it is nontrivial to conclude that the differential coefficient is determined by currents consisting of trivial cylinders together with the union of fibers over Morse flow lines in the base contributing to the orbifold Morse differential.  Our proof requires much of the intersection theory developed by Hutchings in \cite{Hrevisit}, as well as a number of additional subtle combinatorial arguments.}

\begin{remark}[Comparison to combinatorial toric methods]
In \S\ref{ss:toric} we explain a toric combinatorial chain complex that is equivalent to our chain complex in the case that $\mathpzc{H}_{p,q}$ is invariant under another $S^1$-action. (This is suggested in Figure \ref{fig:banana}, which is symmetric under vertical rotation). It agrees with the convex version of the chain complex explained in \cite[\S3]{intoconcave}, and also draws inspiration from \cite{T3}, \cite{letter}. 
However, the differential described in \S\ref{ss:toric} has only partially been proven to equal the ECH differential in \cite{T3}. The complication arises from the ``virtual edges" representing the elliptic orbits at the intersections of $S^3$ with the axes in $\C^2$.

 Our approach circumvents this gap and also gives a means to understand the embedded contact homology chain complexes of more general Seifert fiber spaces and open books. These can be much more topologically complicated than the manifolds carrying toric contact forms, which must be diffeomorphic to 3-spheres, $S^2 \times S^1$, or lens spaces. It is also more geometrically intuitive to understand the knot filtration from the perspective of the open book.  Moreover, by \cite{CH}, our methods can be used to study the chain complexes of contact forms adapted to periodic, positive fractional Dehn twist coefficient open books, permitting applications to surface dynamics on the pages of these open books.
\end{remark}

\subsubsection*{\S\ref{s:topology}-\ref{s:ECHI}: ECH index computation}
Since we do not have access to the combinatorial toric chain complex as in \cite{T3, beyond}, for reasons as detailed in \S \ref{ss:toric}, it is rather involved to derive the ECH index. We compute the ECH index in \S\ref{s:topology}-\ref{s:generalities} via three different trivializations. The most significant is the ``orbibundle" trivialization, arising from our prequantization orbibundle description. We use it to compute the Conley-Zehnder index, and hence the monodromy angles, of the fiber orbits, in \S \ref{ss:orbitau}. In \S\ref{ss:trivs} we relate the orbibundle trivialization to the constant and page trivializations, which are more natural when computing the other components of the ECH index, the relative first Chern number and relative intersection pairing, carried out in \S\ref{ss:chern}-\ref{ss:Q}. The formula for the ECH index is given in Theorem \ref{thm:pqI}.

The calculation of the contribution from the Conley-Zehnder indices to the ECH index is extremely combinatorially delicate.  We circumvent this issue by using the known computation of $ECH_*(S^3,\xi_{std})$; however, some study of the ECH index is still needed to compute the differential and desired spectral invariants. (In our previous work \cite{preech,kech}, alternate perturbations were available, which facilitated our understanding of the ECH index, but as $p$ grows, the analogous combinatorial problem becomes increasingly intractable.)

\subsubsection*{\S \ref{s:spectral}: Computing the action and knot filtrations}
 
The final section, \S\ref{s:spectral}, contains our proof of Theorems \ref{thm:kech-intro} and \ref{thm:toruslinking}. We use the idea of the \emph{degree} of a generator,
\[
d(\zb^B\zh^H\zp^P\zq^Q) = \frac{B+H+\frac{1}{p}P+\frac{1}{q}Q}{| e|} = pqB+pqH+qP+pQ,
\]
which computes the weighted algebraic multiplicity of its underlying fibers (there is also a notion of relative degree, given by the difference in the individual degrees of the Reeb currents). This degree equals the \emph{degree} of the curves counted in the ECH differential (see also \cite[\S4]{preech}); we expect this to be true in general for Seifert fibrations with negative Euler class.  Using the description of the chain complex from \S\ref{s:ECHI}, we compute a formula for the degree of an ECH generator in terms of its ECH index.

In particular, so long as $k$ and the action level $L(\varepsilon)$ satisfy the constraints from Lemma \ref{lem:orbitseh}, the group generator of $ECH_{2k}^{L(\varepsilon)}(S^3,\lambda_{p,q,\varepsilon})$ can be represented by an ECH generator of degree $N_k(p,q)$. Using this, in Proposition \ref{prop:ck} we show that
 \[
c_k(S^3,\lambda_{p,q}) = N_k({1/p}, {1/q}),
\]
where the $c_k(S^3,\lambda_{p,q})$ are the \textit{ECH spectral numbers}, defined in \S\ref{ss:ECHspectrum}, which record the actions of certain homologically essential ECH generators and whose limit recovers the contact volume.

Furthermore, degree controls the knot filtration.  Let $\alpha$ be any Reeb current with no $\alpha_i$ equal to $\zb$. Then
\[
\begin{split}
\mathcal{F}_{\zb}(\zb^B\alpha) = & m \rot(\zb) + \ell(\alpha,\zb), \\
=& d(\zb^B \alpha) + B\delta_{L(\varepsilon)},
\end{split}
\]
as we prove in Proposition \ref{prop:Fb}. Here $0<\delta_{L(\varepsilon)}<<1$ is the change in $\rot(\zb)$ under the perturbation (\ref{eq:intropert}) and all other terms are as defined in \S\ref{ss:overviewech}.
We prove Theorem \ref{thm:kech-intro} using these computations along with the description of the chain complex in \S \ref{s:ECHI}, which uses the Morse-Bott direct limit methods from \cite{preech, kech}.  

\begin{remark} We have $d(\zb^B \alpha) =  pq \ \A_{\lambda_{p,q}}(\zb^B \alpha)$. Thus, the estimates from knot filtered ECH (in the limit as $\varepsilon\to0$ and with $\vol(\lambda_{p,q})=1/pq$) realize the extremal values of the ratio between action and linking in these examples, cf. \cite[Prop.~1.3]{bhs}. By taking care with the estimates on action from Lemma \ref{lem:efromL}(ii), it could be possible to use our methods to estimate the systolic interval for perturbations of ellipsoids that go beyond the toric ones.
\end{remark}

\subsection{Applications to surface dynamics}\label{ss:introcalabi}

As in \cite{HuMAC} and \cite{weiler}, in order to study a symplectomorphism $\psi$ we construct a contact form on a three-manifold containing an embedded surface transverse to the Reeb flow, a \emph{global surface of section}, which the Reeb flow returns to itself globally while pointwise performing the symplectomorphism $\psi$, see Definition \ref{def:gss}. Other authors use a similar construction for opposite goals, i.e., they utilize  the existence of surfaces of section to understand Reeb flows, see for example \cite{cdhr} on topological entropy of generic Reeb flows.

We now give an outline of the arguments in our setting, deferring definitions and proofs of various lemmas to \S\ref{s:calabi}. Consider a surface $\Sdg$ of genus $g=(p-1)(q-1)/2$ with one boundary component and an exact symplectic form $\omega$ of area one. Put collar coordinates $(r,\theta)$ with $\sqrt{1-\varepsilon^2}\leq r\leq 1$ and $\theta\in\R/2\pi\Z$ on a neighborhood of $\partial\Sdg$. Near $\partial\Sdg$, any primitive $\beta$ of $\omega$ we consider will have the form
\begin{equation}\label{eqn:betaboundary1}
\beta=\frac{r^2}{2\pi}\,d\theta.
\end{equation}

Let $\psi$ be a symplectomorphism of $(\Sdg,\omega)$ which is freely isotopic to the monodromy $\phi_{p,q}$ of the positive $T(p,q)$ torus knot and a rotation in the boundary coordinates:
\[
\psi(r,\theta)=(r,\theta+\theta_0).
\]
(We suppress $p,q$ decorations from $\phi$ for simplicity, assuming they are clear from context.) We now make the following definition.
\begin{definition}\label{def:f} The \emph{action function} of $(\psi,\beta,\theta_0)$ is the unique function $f=f_{\psi,\beta,\theta_0}$ for which
\[
df=\psi^*\beta-\beta\text{ and }f|_{\partial\Sdg}=\theta_0.
\]
The fact that $f$ exists and is unique is a consequence of Lemma \ref{lem:psiexact}.  We will drop the subscripts from the action function when they are clear from context.
\end{definition}

\begin{example} The action function in the case where $\psi$ is the periodic representative of the monodromy when $(p,q)=(2,3)$ is given as follows. In $(x,y)\in\R^2/\Z^2$ coordinates on $\dot\Sigma_1$ we may write
\[
\psi(x,y)=(x+y,-x).
\]
It is a straightforward computation to show that $\psi^*\beta=\beta$, where $\beta$ is the quotient of $\frac{1}{2}(x\,dy-y\,dx)$ under the $\Z^2$-action. Thus $f\equiv\frac{1}{6}$ and in this case $\theta_0$ equals the fractional Dehn twist coefficient of $\psi$; see \S\ref{ss:OBD}.

\end{example}

\begin{definition} The \emph{Calabi invariant} of $\psi$ is the number
\[
\mathcal{V}(\psi)=\V_\beta(\psi):=\int_{\Sdg}f\omega.
\]
\end{definition}

In general, the Calabi invariant depends on $\beta$ (as in Pirnapasov-Prasad \cite{pp}). However, the variance in $\beta$ is controlled by the homotopy class of $\psi$, and so in the cases under consideration, all $\V_\beta(\psi)$ are equal; see Lemma \ref{lem:Vindep}.

\begin{remark}\label{rmk:f}
\begin{enumerate}[ (i)]
\itemsep-.25em
\item Usually, the Calabi invariant is defined for Hamiltonian symplectomorphisms, as is the case in \cite{pp}. Our definition of $f$, following \cite{HuMAC}, drops the requirement that $\psi$ be Hamiltonian (or even isotopic to the identity, although, as explained in the proof of Lemma \ref{lem:psiexact}, that requirement depends on the free isotopy class of $\psi$). See Joly \cite[\S1]{joly} for further discussion.

\item Because $\psi$ is a rotation near the boundary,
\[
\psi^*\left(\frac{r^2}{2}\,d\theta\right)=\frac{r^2}{2}\,d\theta,
\]
thus not only does $f|_{\partial\Sdg}=\theta_0$, but $f\equiv\theta_0$ in a collar neighborhood of $\partial\Sdg$.
\end{enumerate}
\end{remark}

Finally, we recall several definitions and facts from \cite{HuMAC, weiler}. An $\ell$-tuple $\gamma=(\gamma_1,\dots,\gamma_\ell)$ of points in $\Sdg$ is a \textit{periodic orbit} of $\psi$ if $\gamma_{i+1\mod\ell}=\psi(\gamma_i)$. A periodic orbit is said to be \textit{simple} if $\gamma_i\neq\gamma_j$ for $i\neq j$ and its \textit{total action} is given by
\[
\A(\gamma):=\sum_{i=1}^\ell f(\gamma_i).
\]
If $\ell(\gamma)=|(\gamma_1,\dots,\gamma_\ell)|$ then the \textit{mean action} of $\gamma$ is the ratio $\A(\gamma)/\ell(\gamma)$. As in the case of the Calabi invariant, the total action and mean action do not depend on $\beta$; see Lemma \ref{lem:Aindep}.

Let $\P(\psi)$ denote the simple periodic orbits of $\psi$.
In \S\ref{ss:mac} we will prove the following theorem. The proof will use a suspension construction, described in \S\ref{ss:3from2}, to relate the ECH computations in \S\ref{s:topology}-\ref{s:spectral} to the dynamics of $\psi$.

\begin{theorem}\label{thm:Calabi} Let $\psi$ be an area-preserving diffeomorphism of $(\Sdg,\omega)$, wherein
\begin{itemize}
\itemsep-.25em
\item $g=(p-1)(q-1)/2$,
\item $\psi$ is freely isotopic to the monodromy of the standard genus $g$ open book decomposition of $S^3$ with $T(p,q)$ as the binding, and further is isotopic relative to the boundary of $\Sdg$ to the Nielsen-Turston representative of this monodromy twisted positively near the boundary by an amount $-\frac{1}{pq}<d\leq0$,
\item  $\theta_0=\frac{1}{pq}+d$ and $f=f_{\psi,\beta,\theta_0}>0$ for any primitive $\beta$ of $\omega$ satsifying (\ref{eqn:betaboundary1}), 
\item $\V(\psi)<pq\cdot\theta_0^2$.
\end{itemize}
Then we have
\[
\inf\left\{\frac{\A(\gamma)}{\ell(\gamma)} \ \bigg \vert \ \gamma\in\P(\psi)\right\}\leq\sqrt{\frac{\V(\psi)}{pq}}.
\]
\end{theorem}

We define the terms ``open book decomposition," ``monodromy," and ``Nielsen-Thurston representative," as well as discussing the particulars of a positive twist near the boundary, in \S\ref{s:calabi}. Note that the lower bound on $d$ is necessary to obtain the hypothesis $f>0$.

\begin{remark}
Several of the hypotheses of Theorem \ref{thm:Calabi} have no analogue in \cite{HuMAC} or \cite{weiler, weiler2}. The necessity of these new hypotheses is explained before the statement of Proposition \ref{prop:constructcontact} in \S\ref{ss:3from2} and in Remark \ref{rmk:Vpqd}.
\end{remark}

We would expect the right hand side of the conclusion of Theorem \ref{thm:Calabi} to be $\V(\psi)$, which by our assumption $\V(\psi)<pq\cdot\theta_0^2$ is smaller than the current value. The stronger estimate is expected by work in the generic setting, as first initiated by Irie \cite{irie, irie2} and further developed by Bechara Senior, Hryniewicz, and Salom\~ao \cite{bhs}. However, the stronger upper bound would require a stronger result than what Theorem \ref{thm:toruslinking} provides, as detailed in Remark \ref{rem:stronger}.

\subsubsection*{Related results}

The study of existence, multiplicity, and quantitative features of periodic points of surface diffeomorphisms dates back at least a century to the Poincar\'{e}-Birkhoff theorem. In the 90s, Franks' work in \cite{franks92, franks} provided the framework for modern work on the subject: any area-preserving homeomorphism of the open annulus with at least one periodic point must have infinitely many. Franks and Handel then showed in 2003 in \cite{frankshandel} that Hamiltonian diffeomorphisms of surfaces of genus at least one have periodic points of arbitrarily high minimal period. Most recently, work of Le Calvez in \cite{lc1} classifies orientation-preserving homeomorphisms of surfaces of all genera without wandering points (a class which includes symplectomorphisms) based on their number of periodic points and the quantitative characteristics of those points.

Our work continues the line of reasoning suggested by Asaoka and Irie in their study of Hamiltonian diffeomorphisms of surfaces via ECH spectral invariants in \cite{ai}. Since their work, many authors have used ECH or its relative for mapping tori, periodic Floer homology (PFH), to study surface diffeomorphisms, and we discuss the work most closely related to ours here. For example, as we have already discussed, Hutchings initiated the study of knot filtered ECH in order to prove an estimate on the infimum (and the supreum) of the mean action of periodic orbits of disk symplectomorphisms in terms of their Calabi invariant in \cite{HuMAC}. In other directions, see the introduction to work of Prasad \cite{prasad} for a nice discussion of the related Conley conjecture, which was also addressed via Floer-theoretic methods.

Work by Pirnapasov and Prasad \cite[Thm.~1.12]{pp} is similar to but applies to more general surfaces and isotopy classes of $\psi$ than Theorem \ref{thm:Calabi} under the assumption that $\psi$ is a $C^\infty$ generic Hamiltonian, which cannot be assumed to be compactly supported in the interior because perturbations are required all the way up to the boundary. (See Remark~\ref{rmk:f}~(i) for a discussion of the {difference between our case and the Hamiltonian setting}.) Moreover, the upper bound on the infimum of mean action in \cite[Thm.~1.12]{pp} is stronger than ours, because the Calabi invariant itself appears rather than $\sqrt{\V(\psi)/pq}$. Finally, they also obtain a lower bound on the supremum of mean action, which we cannot obtain using arguments similar to those in \cite{HuMAC} or \cite{weiler, weiler2} because taking an inverse of $\psi$ changes its isotopy class. Isotopy class is of no consequence in \cite{HuMAC}, while in \cite{weiler, weiler2} it affects the lens space appearing in the analogue of Proposition \ref{prop:constructcontact}.

Pirnapasov and Prasad \cite{pp} use the spectral invariants coming from periodic Floer homology. (The analogy between ECH and PFH is most obvious in the case of a Reeb flow transverse to the pages of an open book decomposition, which is the setting we explore in this paper.) The PFH spectral invariants, which have been recently developed by \cite{cghs, cgpz, eh}, are analogous to the ECH spectrum. However, ECH comes with a knot filtration, which up to the rotation parameter gives rise to a topological invariant of the three-manifold together with the binding of an open book.  We are therefore able to compute it using a  contact form of our choice, in contrast to the ECH spectrum. In PFH, the quantity analogous to linking number is called ``degree," but it plays a slightly different role than the knot filtration does and is more similar to the ECH index; see  \cite[Thm.~8.1]{eh}.

Significant features of \cite{pp} having no analogue in this paper are their generic density equidistribution results for periodic orbits \cite[\S 1.1]{pp}. While our Theorem \ref{thm:Calabi} can be interpreted as the fact that the action function takes an average value on some orbit set that is at most its global average, the results referenced from \cite{pp} identify, for any continuous function on a symplectic surface-with-boundary, a sequence of orbit sets on which the average values of $f$ approximate the global average of $f$.

We now turn to the work of authors who have identified additional features of their periodic points; many of these results did not require the use of Floer theoretic methods. In \cite{abror}, Pirnapasov removes the hypothesis of \cite[Thm.~1.2]{HuMAC}, which required that the disk map be a rotation near the boundary. Le Calvez (and collaborators) have used generating families and finite dimensional methods to study surface dynamics.  As mentioned above, Le Calvez classifies in \cite{lc1} the orientation-preserving homeomorphisms of surfaces without wandering points. In \cite{lc2} he shows that a $C^1$ pseudorotation of a disk equal to a rotation on the boundary must have Calabi invariant equal to its boundary rotation number, which under the additional assumptions of smoothness and rotation near the boundary implies Hutchings' original result. 

Guih\'{e}neuf, Le Calvez, and Passeggi found infinitely many periodic orbits for measure-preserving homeomorphisms isotopic to the identity of surfaces of genus at least two, along with an estimate on their rotation vectors in \cite{glcp}. {Their hypotheses require the measure to have total support and the rotation vector to be a real multiple of an integral one. Prasad proved a similar result in \cite{prasad}, without the conclusion about rotation vectors, using PFH; for a discussion of the relationship between these methods, see the remark following \cite[Cor.~1.1]{glcp} and the introduction to \cite{prasad}.} Finally, in \cite{lcs}, Le Calvez and Sambarino showed that for genus at least one, generic $C^r$ Hamiltonian diffeomorphisms admit infinitely many non contractible periodic orbits.

\subsubsection*{Acknowledgements} 
 We thank Michael Hutchings for thoughtful discussions and correspondence. We also thank Umberto Hryniewicz, Tom Mrowka, Rohil Prasad, Abror Pirnapasov, and Jeremy Van Horn-Morris for helpful conversations.  Finally, we thank the referee for their careful reading of our paper and their suggestions. Jo Nelson is partially supported by NSF grants DMS-2104411 and CAREER DMS-2142694.  During her stay at the Institute for Advanced Study, she was supported as a {Stacy and James Sarvis Founders’ Circle Member}. Morgan Weiler was partially supported by an NSF MSPRF grant DMS-2103245.

\section{From contact geometry to surface dynamics}\label{s:calabi}

Here we collect the material relevant to establishing our applications to the study of surface dynamics. First, in \S\ref{ss:OBD} we review open book decompositions and their relationships to Reeb vector fields. Next, \S\ref{ss:calabi} contains the proof of Lemma \ref{lem:psiexact}, which allowed us to define the action function in \S\ref{ss:introcalabi}.  We also compute the effect of a boundary twist on the Calabi invariant, which is key to the proof of Theorem \ref{thm:Calabi} in \S\ref{ss:mac}. The dictionary between our surface symplectomorphisms $\psi$ and contact geometry is proved in Proposition \ref{prop:constructcontact} in \S\ref{ss:3from2}. Finally, in \S\ref{ss:mac}, we prove Theorem \ref{thm:Calabi}, using the ECH results of \S\ref{s:spectral} as a black box.

\subsection{Open book decompositions and contact forms}\label{ss:OBD}

We review the relationship between open book decompositions and contact forms. Ultimately, our goal is to describe the restrictions on the return map of the Reeb vector field of a contact form supported by a given open book. We provide specific citations as well as summarize material from \cite[\S2.1]{kech}, \cite{CH}, \cite{etnlec}, and \cite[Ch.9]{osbook}.

\subsubsection{Basic definitions}

We now collect the basic notions associated to open books and compatible contact structures.

\begin{definition}
{Given a closed oriented three-manifold $Y$, the pair $(B,\pi)$ is an \emph{open book decomposition} of $Y$ if the \emph{binding} $B$ is an oriented link in $Y$ and $\pi:Y\setminus B\to S^1$ is a fibration over the circle so that each $\pi^{-1}(\theta)$ is a Seifert surface for $B$. The \emph{pages} of $(B,\pi)$ are the closures of the $\pi^{-1}(\theta)$.} 
Open book decompositions of $Y$ and $Y'$ are equivalent if there is a diffeomorphism between $Y$ and $Y'$ taking bindings to bindings and pages to pages.
\end{definition}

\begin{example} Define
\[
H^+=\{(z_1,z_2)\in S^3\ | \ z_1z_2=0\}\text{ and }H^-=\{(z_1,z_2)\in S^3\ | \ z_1\overline{z_2}=0\}.
\]
Both have the same underlying set, but they carry different orientations: $H^+$ as the boundary of the pages of the open book decomposition with projection map
\[
\pi_+:S^3\setminus H^+\to S^1\text{ given by }(z_1,z_2)\mapsto\frac{z_1z_2}{|z_1z_2|},
\]
and $H^-$ as the boundary of the pages of the open book decomposition with projection map
\[
\pi_-:S^3\setminus H^+\to S^1\text{ given by }(z_1,z_2)\mapsto\frac{z_1\overline{z_2}}{|z_1\overline{z_2}|}.
\]
In polar coordinates $(r_1,\theta_1,r_2,\theta_2)$ on $\C^2$, these are the maps $\pi_\pm(r_1,\theta_1,r_2,\theta_2)\mapsto\theta_1\pm\theta_2.$
\end{example}

It is often helpful to start by `ignoring' the total space $Y$ in the following sense. An \emph{abstract open book} is a pair $(\Sigma,\phi)$ where $\Sigma$ is an oriented compact surface with nonempty boundary and $\phi$ is a diffeomorphism of $\Sigma$ which is the identity near $\partial\Sigma$; the map $\phi$ is called the \emph{monodromy}. We have the following:
\begin{itemize}
\item {Knowing the data of an abstract open book allows us to construct a manifold $Y_\phi:=\Sigma\times[0,1]/\sim_\phi$, where $(z,1)\sim_\phi(\phi(z),0)$ for all $z\in\Sigma$ and $(z,t)\sim_\phi(z,t')$ for all $z\in\partial\Sigma$.  (This is essentially the mapping torus of $\phi$ with an additional boundary quotient.) The manifold $Y_\phi$ carries an open book decomposition with binding $B_\phi:=\partial\Sigma\times[0,1]/\sim_\phi$ and projection map, called $\pi_\phi$, the projection onto the $[0,1]$-coordinate.}
\item {Given abstract open books $(\Sigma,\phi)$ and $(\Sigma',\phi')$, if there is a diffeomorphism between $\Sigma$ and $\Sigma'$ under which $\phi$ and $\phi'$ are conjugate, then we say the two abstract open books are equivalent.} Equivalent abstract open books determine diffeomorphic open book decompositions: see \cite[Lem.2.4, Rmk.2.6]{etnlec}.
\item The monodromy $\phi$ of $(B,\pi)$ is the return map of the flow of a vector field which is positively transverse to the pages and meridional near $B$, see Definition \ref{def:return}. The diffeomorphism type of the pages determines $\Sigma$, therefore open book decompositions conversely determine abstract open books (again up to the appropriate equivalences).
\end{itemize}

\begin{example} The pages of $(H^+,\pi_+)$ are annuli and the monodromy is a right-handed Dehn twist around the core circle of the annulus, as shown in \cite[Ex.~4.4.8(2)]{geigesbook}. That the pages of $(H^-,\pi_-)$ are annuli and that its monodromy is a left-handed Dehn twist can be shown analogously.
\end{example}

The \emph{mapping class group} of a surface with boundary is the set of isotopy classes of diffeomorphisms of the surface which are the identity along the boundary. The same group is obtained if one starts with isotopy or homotopy classes of homeomorphisms. The notion of equivalence for abstract open books immediately motivates the role of the mapping class group in the study of open books: if two open books have isotopic monodromies, then they are equivalent.

Given an abstract open book $(\Sigma,\phi)$, it can be \emph{positively/negatively stabilized} by adding a 1-handle and composing $\phi$ with a right/left-handed Dehn twist about a closed curve intersecting the cocore of the new 1-handle exactly once. At the level of three-manifolds, stabilization amounts to connect summing with $S^3$.

\begin{example}\label{ex:Tpq} The well-known Milnor fibration
\[
\pi:S^3\setminus T(p,q) \to S^1, \ (z_1,z_2)\mapsto\frac{z_1^p+z_2^q}{|z_1^p+z_2^q|}
\]
is the projection map of an open book decomposition of $S^3$ with the positive torus knot
\[
T(p,q)=\left\{(z_1,z_2)\in S^3\ |\ z_1^p+z_2^q=0\right\}
\]
as its binding. We will denote this open book decomposition of $S^3$ by $(T(p,q),\pi_{p,q})$. Their corresponding abstract open books can be constructed via iteratively positively stabilizing $(H_+,\pi_+)$, as explained in \cite[\S 9]{osbook},\cite[\S 2]{ao}. (In these references, the stabilization procedure is performed ambiently, so that the resulting binding is drawn as $T(p,q)$.) The pages are genus $\frac{(p-1)(q-1)}{2}$, as can be computed using the Euler characteristic.

Torus links correspond to rational open books of lens spaces, see \cite{cabling} but as previously mentioned, we restrict our attention to torus knots, and continue to assume that $p$ and $q$ are relatively prime.
\end{example}

Open books are closely related to contact structures.
\begin{definition}
An open book decomposition $(B,\pi)$ of a 3-manifold $Y$ and a cooriented contact structure $\xi$ on $Y$ are \emph{compatible} if there is a contact structure isotopic to $\xi$ which is the kernel of a contact form $\lambda$ for which $B$ is a Reeb orbit, $d\lambda$ is an area form on every page, and the orientation of $B$ as a Reeb orbit agrees with the boundary orientation of the pages. We also say $(B,\pi)$ \emph{supports} $\xi$. A contact form $\lambda$ realizing the compatibility between $(B,\pi)$ and the isotopy class of $\ker(\lambda)$ is said to be \emph{supported} by $(B,\pi)$.\footnote{Here we use the most general terminology for this relationship from \cite{cdr}.}
\end{definition}

The contact connected sum $\xi \# \xi_{std}$ is compatible with the positive stabilization of an open book decomposition compatible with $\xi$, and is isotopic to $\xi$.
\begin{theorem}[Giroux Correspondence \cite{bhh, giroux}]\label{thm:giroux} Given a closed 3-manifold $Y$, there is a one to one correspondence between oriented\footnote{Here ``oriented" means that the contact structure is the kernel of a contact form; more general types of contact structures exist, but we do not consider them here.} contact structures on $Y$, up to isotopy through oriented plane fields, and open book decompositions of $Y$, up to positive stabilization.
\end{theorem}

\subsubsection{Periodic monodromies with positive fractional Dehn twist coefficients}

We now explain the class of contact forms to which the $\lambda_{p,q}$ belong.
If a mapping class is not reducible, it is either periodic or pseudo-Anosov, according to whether it contains a unique representative which is itself a periodic or pseudo-Anosov homeomorphism. When the surface has boundary, this homeomorphism might only ``represent" the class in the sense that it is freely isotopic (i.e. the isotopy may move the boundary) to other representatives of the class. When the mapping class is periodic, this homeomorphism is a diffeomorphism. Such a representative is called the \emph{Nielsen-Thurston representative} of the class. Identifying a single reference for the proof is complicated by its history (see the ``Historical remarks" at the end of \cite[\S13.3]{farbmarg}), however, the monograph \cite{flp} provides a complete proof, with the classification appearing in \cite[Thm.~11.7]{flp}.

The \emph{fractional Dehn twist coefficient} of a mapping class measures the difference between a true representative $\phi$ of the monodromy mapping class, which is the identity on the boundary, and its Nielsen-Thurston representative, which we will here call $\psi$ and which is likely not the identity on the boundary.  It is defined as follows. At each boundary component $C$ of $\Sigma$ attach an annulus collar $C\times[0,1]$ by gluing $C\times\{1\}$ and $C$. Let $H:\Sigma_x\times[0,1]_t\to\Sigma$ be the free isotopy from $\phi(x)=H(x,0)$ to $\psi(x)=H(x,1)$ and define $\beta:C\times[0,1]\to C\times[0,1]$ by $\beta(x,t)=(H(x,t),t)$. Notice that now the original map $\phi$ is isotopic  to $\psi\cup\beta$ relative to the boundary of the glued surface.

Choose an oriented identification of $C \simeq \R/ \Z$ and lift $\beta$ to $\widetilde{\beta}:\R\times[0,1]\to\R\times[0,1]$. Define 
\[
f(x):= \widetilde{\beta}(x,1)|_\R -\widetilde{\beta}(x,0)|_\R + x,
\]
where $|_\R$ denotes projection onto the $\R$ component. We call $\beta$ a \emph{fractional Dehn twist} by $c \in \Q$, where the \emph{fractional Dehn twist coefficient} $c$ is the rotation number of the circle map which $f$ projects to, e.g.
\[
c:= \lim_{n\to\infty} \frac{f^n(x)-x}{n} \mbox{ for any } x.
\]
When $\psi$ is periodic, $c = f(x) -x$ for any $x$, while {when $\psi$ is pseudo-Anosov, it is also possible to relate $c$ to a geometric quantity determined by the stable and unstable foliations.}  Furthermore, when $\partial \Sigma$ is not connected, number the boundary components and define a fractional Dehn twist coefficient $c_i$ for the $i^\text{th}$ component.

When $\phi$ is periodic, the total space $Y_\phi$ of its abstract open book is Seifert fibered, with the $S^1$ action induced by the mapping torus $[0,1]$ direction. When furthermore all the $c_i$ are positive, it is possible to find a compatible contact structure $\xi$ which is $S^1$-invariant and transverse to the $S^1$-fibers of the Seifert fibration. 

\begin{theorem}[{\cite[Thm.~4.1]{CH}}]\label{thm:CH} Suppose $(\Sigma,\phi)$ is an abstract open book with periodic monodromy $\phi$. Let $c_i$ be the fractional Dehn twist coefficient of the $i^\text{th}$ boundary component and assume all $c_i>0$. Then there is a Seifert fibration on $Y_\phi$ and a contact form $\lambda_\phi$ on $Y_\phi$ adapted to the open book decomposition $(B_\phi,\pi_\phi)$ whose Reeb vector field generates the Seifert fibration.
\end{theorem}
{In addition, \cite[Lem.~4.3]{CH} shows, via \cite{lm}, that given a Seifert fibered space $Y$ (with fibering fixed), if $\xi, \xi'$ are both $S^1$-invariant transverse contact structures, then they are isotopic.}

Explicit contact forms $\lambda_{p,q}$ and $\lambda_{p,q,\varepsilon}$ adapted to the open book $(T(p,q),\pi)$ are described in more detail in Example \ref{symporb} and Remark \ref{rem:obd}.

\subsubsection{Reeb return maps of supported contact forms}

Open book decompositions with supported contact forms are an example of a more general dynamical phenomenon, which we now define. (We draw from \cite[\S2.2]{cdhr}.)
\begin{definition}\label{def:gss}
Let $Y$ be a closed, connected, orientable smooth three-manifold with a vector field $X$. Denote the flow of $X$ for time $t$ by $\psi^t$. A \emph{global surface of section} for the flow of $X$ is an embedding $\iota:S\to Y$ of a compact surface $S$ such that:
\begin{enumerate}[(i)]
\item If $\partial S\neq\emptyset$ then $\iota(\partial S)$ is a link made up of periodic orbits of $X$.
\item The embedding $\iota:S\setminus\partial S\to Y\setminus\iota(\partial S)$ is transverse to $X$.
\item For every $p\in Y$ there are $t_-<0<t_+$ for which $\psi^{t_\pm}(p)\in\iota(S)$.
\end{enumerate}
\end{definition}

The pages of an open book decomposition are global surfaces of section for the Reeb flow of a supported contact form. When the monodromy $\phi$ is periodic, Colin-Honda's proof of Theorem \ref{thm:CH} implies a strong relationship between the fractional Dehn twist coefficients of $\phi$ and the \emph{rotation number} of the binding Reeb orbits, a concept we quickly review now: see \S\ref{ss:CZ} for a more detailed treatment.

Let $\gamma$ be a $T$-periodic elliptic Reeb orbit of $(Y,\lambda)$ and $\tau$ be the symplectic trivialization of $\xi|_\gamma$ which sends a constant pushoff of $\gamma$ into a Seifert surface to $\gamma\times\{(x,y)\}\in\gamma\times\R^2$, where $(x,y)$ is constant in the parameter of $\gamma$. Let $\psi^t$ denote the Reeb flow. Under the trivialization $\tau$, the path $d\psi^t:\xi|_{\gamma(0)}\to\xi|_{\gamma(t)}$ is homotopic relative to the conjugacy classes of its endpoints in $Sp(2;\R)=SL(2;\Z)$ to a path of rotation matrices. The \emph{rotation number} of $\gamma$ is the angle $\theta$ for which $d\psi^T$ is conjugate to rotation by $\theta$.

The proof of the following corollary is contained within the proof of Theorem \ref{thm:CH}.
\begin{corollary}\label{cor:FTDCrot} In the setting of Theorem \ref{thm:CH}, the rotation number of the $i^\text{th}$ binding component as a Reeb orbit is $1/c_i$.
\end{corollary}

Corollary \ref{cor:FTDCrot} implies that in some cases, there is a strong relationship between the periodic representative of the monodromy and the Reeb flow. We provide a few more definitions in order to explain this relationship.

\begin{definition}\label{def:return} Let $S$ be a global surface of section for the flow of $X$ on $Y$.
\begin{itemize}
\item The \emph{return time} of the flow of $X$ is a map $f:S\to\R_{>0}$ given by sending a point $x\in S$ to the minimum $t$ so that $\psi^t(\iota(x))\in\iota(S)$. It extends continuously to $\partial S$ in our setting (see \cite[Def.~2.5]{cdhr} for the precise requirements).
\item The \emph{return map} of the flow of $X$ is the diffeomorphism
\[
\psi:S\to S \mbox{ given by } x\mapsto\iota^{-1}\circ\psi^{f(x)}\circ\iota(x).
\]
\end{itemize}
\end{definition}

\begin{remark}
In the context of a contact form supported by an open book decomposition, the return map must be freely isotopic to the monodromy, but they must differ near the boundary. In some cases, including the setting of Theorem \ref{thm:CH}, the return map even equals the Nielsen-Thurston representative of the monodromy (providing a degenerate Morse-Bott contact form).
\end{remark}

\begin{example}\label{cor:CHdetails} The open book decomposition $(T(p,q),\pi_{p,q})$ of Example \ref{ex:Tpq} falls under the hypotheses of Theorem \ref{thm:CH}. Thus there is a contact form $\lambda_{p,q}$ on $S^3$ for which:
\begin{enumerate}[ (i)]
\itemsep-.25em
\item The (right handed) $pq$-periodic representative $\psi_{p,q}$ of the monodromy $\phi_{p,q}$ is the return map of the Reeb vector field $R$ of $\lambda_{p,q}$.
\item $R_{p,q}$ generates the $S^1$-action on $S^3$ with fundamental domain given by an orbifold 2-sphere with two exceptional points, one with isotropy group $\Z/p\Z$ and the other with isotropy group $\Z/q\Z$.
\item The binding  $T(p,q)$ is a regular fiber of the Seifert fibration, and as a Reeb orbit has rotation number $pq$.
\end{enumerate}
\end{example}

We directly construct $\lambda_{p,q}$ and establish these properties in \S \ref{s:topology}.

\subsection{The Calabi invariant}\label{ss:calabi}

We now provide further details on the action function and Calabi invariant, supplementing \S\ref{ss:introcalabi}. Recall that $\Sdg$ is a surface of genus $g=(p-1)(q-1)/2$ with one boundary component and an exact symplectic form $\omega$. Furthermore, we have put collar coordinates $(r,\theta)$ with $\sqrt{1-\varepsilon^2}\leq r\leq 1, \theta\in\R/2\pi\Z$ on a neighborhood of $\partial\Sdg$, and any primitive $\beta$ of $\omega$ near $\partial\Sdg$ has the form
\begin{equation}\label{eqn:betaboundary}
\beta=\frac{r^2}{2\pi}\,d\theta.
\end{equation}

Further recall that $\psi$ is a symplectomorphism of $(\Sdg,\omega)$, which is freely isotopic to the monodromy $\phi_{p,q}$ of $(T(p,q),\pi_{p,q})$ from Example \ref{ex:Tpq} and a rotation
\[
\psi(r,\theta)=(r,\theta+\theta_0)
\]
in the boundary coordinates. (We often suppress $p,q$ decorations for simplicity.)

The following lemma allows us to construct the action function $f$, i.e., the primitive of $\psi^*\beta-\beta$ with $f|_{\partial\Sdg}=\theta_0$.
\begin{lemma}\label{lem:psiexact} There is a primitive $\beta$ of $\omega$ for which $\psi$ is exact, that is,
\begin{equation}\label{eqn:psiexact}
[\psi^*\beta-\beta]=0\in H^1(\Sdg;\R).
\end{equation}
Moreover, any two such primitives $\beta,\beta'$ differ by an exact one form which is the derivative of a function which is constant near the boundary of $\Sdg$.
\end{lemma}

\begin{proof} Let $\beta_0$ be any primitive of $\omega$. First note that $\psi^*\beta_0-\beta_0$ is immediately closed by the fact that $\psi$ is a symplectomorphism, and so defines a class in $H^1(\Sdg;\R)$.

\medskip

\textbf{Claim:} There is a closed 1-form $\alpha_0\in\Omega^1(\Sdg;\R)$ for which
\[
0=[\psi^*(\beta_0+\alpha_0)-(\beta_0+\alpha_0)]=[\psi^*\beta_0-\beta_0]+(\psi-id)^*[\alpha_0].
\]
Assuming the claim, we set $\beta=\beta_0+\alpha_0$, which is also a primitive for $\omega$ because $\alpha_0$ is closed.

\medskip

The claim is true so long as the pullback action of $\psi-id$ on $H^1(\Sdg;\R)$ is surjective. This action is dual to the pushforward action on $H_1(\Sdg;\R)$. Thus to prove the claim it remains to show that $\psi_*$ does not have one as an eigenvalue. Assuming by contradiction that $\psi_*$ does have one as an eigenvalue, then $\phi_*$ must also because capping off a single boundary component has no effect on the action on $H_1$ by \cite[\S6.3]{farbmarg}. In \cite[\S2.1]{eo}, Etnyre-Ozbagci compute $H_1$ of the total space of an open book decomposition in terms of $\phi$, and if $\phi$ has one as an eigenvalue then the total space must have nontrivial $H_1$. But this contradicts the fact that $H_1(S^3;\Z)=0$.

The proof that any two primitives of $\omega$ satisfying (\ref{eqn:psiexact}) can only differ by an exact one form is similar: if $\beta-\beta'=[\alpha]\in H^1(\Sdg;\R)$, then
\[
0=[\psi^*\beta'-\beta']=[\psi^*(\beta+\alpha)-(\beta+\alpha)]=[\psi^*\beta-\beta]+(\psi-id)^*[\alpha]=(\psi-id)^*[\alpha],
\]
and we know by the argument above that $\psi^*$ does not have one as an eigenvalue. Finally, if $\alpha=dh$ for $h:\Sdg\to\R$, the fact that both $\beta$ and $\beta+\alpha$ satisfy (\ref{eqn:betaboundary}) forces $h$ to be constant near the boundary of $\Sdg$.
\end{proof}

\begin{remark} The action function depends on $\beta$. If $h:\Sdg\to\R$ is a smooth function which is constant near the boundary, then
\[
\psi^*(\beta+dh)-(\beta+dh)=df+\psi^*(dh)-dh=d(f+h\circ\psi-h),
\]
showing that
\begin{equation}\label{eqn:fdg}
f_{\psi,\beta+dh}=f+h\circ\psi-h.
\end{equation}
Note that we do not need to consider any other variations in $\beta$ (e.g., by a closed but not exact 1-form) by the uniqueness part of Lemma \ref{lem:psiexact}.
\end{remark}

The Calabi invariant is defined by 
\[
\V(\psi):=\int_{\Sdg}f\omega
\]
and is independent of $\beta$ but does depend on $\theta_0$. While the action function depends on $\beta$, the Calabi invariant does not.
\begin{lemma}\label{lem:Vindep} If $\beta, \beta'$ both satisfy (\ref{eqn:betaboundary}) and (\ref{eqn:psiexact}) then $\V_\beta(\psi)=\V_{\beta'}(\psi)$.
\end{lemma}
\begin{proof} The proof is the same as that of \cite[Lem.~1.4]{weiler} because $\beta$ and $\beta'$ must differ by an exact one form by Lemma \ref{lem:psiexact}. For $\psi$ whose action on $H^1$ does have one as an eigenvalue, the proof from \cite{weiler} would not translate to our setting, but we are lucky that our $\psi_*$ does not have one as an eigenvalue (see the proof of the claim in the previous lemma).\footnote{In the cases of the annulus studied in \cite{weiler} or the disk in \cite{HuMAC}, this difficulty does not arise, because the total space of an open book on the disk or annulus is always a lens space with $H^1=0$.} Note that (\ref{eqn:fdg}) is key to the proof of \cite[Lem.~1.4]{weiler}, which is why we make special note of it in the remark above.
\end{proof}

Similarly, the total action does not depend on $\beta$.
\begin{lemma}\label{lem:Aindep} The total action $\A(\gamma)$ is independent of the choice of $\beta$ satisfying (\ref{eqn:psiexact}).
\end{lemma}
\begin{proof} The proof is the same as that of \cite[Lem.~1.1]{HuMAC}. However, instead of using the fact that the path $\psi_*\eta_{\ell-1}-\eta_1$ is homologous to a path along the boundary,\footnote{Note that in \cite[Lem.~1.1]{HuMAC}, the path referred to as $\phi_*\eta_d$ ought to be $\phi_*\eta_{d-1}$.} we instead use Lemma \ref{lem:psiexact} to argue that any other $\beta'$ must be of the form $\beta+dh$ where $h:\Sdg\to\R$ is constant near the boundary. Therefore the difference between the total action of a simple orbit $\gamma$ computed with $f_{\psi,\beta}$ and with $f_{\psi,\beta+dh}$ is precisely
\[
\int_{\psi_*\eta_{\ell-1}-\eta_1}dh,
\]
which equals the difference in the values of $h$ at the two boundary endpoints of the path $\psi_*\eta_{\ell-1}-\eta_1$. Since $h$ is constant near $\partial\Sdg$, this difference is zero.
\end{proof}

The following example, explaining the dependence of $\V(\psi)$ on $\psi$ near $\partial\Sdg$, will be key to the proof of Theorem \ref{thm:Calabi} in \S\ref{ss:mac}. Specifically, we analyze the effect on the action function $f$, the total action $\A$, and the Calabi invariant $\V$ of composing $\psi$ with an integrable\footnote{Here ``integrable" means of the form $(r,\theta)\mapsto(r,\theta+g(r))$ for some smooth function $g$ of $r$.} right-handed twist in the boundary annulus of $\Sdg$, where $\beta$ is standard. It is very similar to the computation appearing in Step 2 of the proof of \cite[Thm.~1.2]{HuMAC} assuming \cite[Prop.~2.2]{HuMAC}, but with the emphasis placed on the change in boundary rotation and explanations added for a few details.

\begin{example}\label{ex:Vtwist} Let $\zeta>0$ be small and let $\varepsilon:\left[1-\zeta,1\right]\to[\varepsilon(1),0]$ be a smooth monotonically decreasing function which is identically zero near $\{r=1-\zeta\}$ and identically $\varepsilon(1)<0$ near $\{r=1\}$. Define a right handed twist in the boundary annulus by
\[
T_\varepsilon(r,\theta)=(r,\theta+2\pi\varepsilon(r)),
\]
meaning that
\[
T_\varepsilon^*\beta-\beta=\begin{pmatrix}1&2\pi\varepsilon'(r)\\0&1\end{pmatrix}\begin{pmatrix}0\\r^2/2\pi\end{pmatrix}-\begin{pmatrix}0\\r^2/2\pi\end{pmatrix}=r^2\varepsilon'(r)\,dr=df_\varepsilon.
\]
Now let $\hat\psi=T_\varepsilon\circ\psi$ and $\hat f=f_{\hat\psi,\beta,\theta_0+\varepsilon(1)}$. The fact that $\psi^*=id$ on the collar containing the support of $T_\varepsilon$ means that on that collar
\[
d\hat f=\hat\psi^*\beta-\beta=\psi^*T_\varepsilon^*\beta-\beta=\psi^*(df_\varepsilon)=df_\varepsilon,
\]
as $df_\varepsilon$ is invariant under rotation.

Recall from Definition \ref{def:f} that the value of the action function at the boundary equals the amount by which the map rotates along the boundary. Thus
\[
\hat f(1,\theta)=\theta_0+\varepsilon(1).
\]
By the fundamental theorem of line integrals,
\[
\hat f(R,\theta)=\hat f(1,\theta)+\int_1^Rr^2\varepsilon'(r)\,dr.
\]
Combining the previous two equations results in
\begin{equation}\label{eqn:fd}
\hat f(R,\theta)=\theta_0+\varepsilon(1)-\int_R^1r^2\varepsilon'(r)\,dr.
\end{equation}
Integrating by parts, we obtain
\[
\hat f(R,\theta)=\theta_0+R^2\varepsilon(R)+\int_R^12r\varepsilon(r)\,dr.
\]
Using the estimates $0\geq\varepsilon(r)\geq\varepsilon(1)$, we calculate that on the collar support of $T_\varepsilon$,
\begin{equation}\label{eqn:Cdiff}
\theta_0+R^2\varepsilon(R)\geq\hat f\geq\theta_0+R^2\varepsilon(R)+\varepsilon(1)(1-R^2).
\end{equation}
In particular, the lower bound in (\ref{eqn:Cdiff}) and the fact that $\varepsilon$ is nonincreasing implies that on the support of $T_\varepsilon$,
\begin{equation}\label{eqn:Cflb}
\hat f\geq\theta_0+\varepsilon(1)=f+\varepsilon(1).
\end{equation}

Away from the collar we still have $d\hat f=df$. Therefore for $x\in\Sdg$ away from the collar, by the fundamental theorem of line integrals applied to any curve $\gamma$ from a point $(1-\zeta,\theta)$ abutting the collar support of $T_\varepsilon$ to $x$, we have:
\[
\hat f(x)-\hat f(1-\zeta,\theta)=\int_\gamma df=f(x)-f(1-\zeta,\theta).
\]
Using our computation in (\ref{eqn:fd}) to rewrite $\hat f(1-\zeta,\theta)$ and using the fact that $f(1-\zeta,\theta)=\theta_0$ because $(1-\zeta,\theta)$ is in the collar on which $f\equiv\theta_0$, explained in Remark~\ref{rmk:f}(ii), we may rearrange the equation above to obtain
\begin{align*}
\hat f(x)&=f(x)+\left(\theta_0+\varepsilon(1)-\int_{1-\zeta}^1r^2\varepsilon'(r)\,dr\right)-\theta_0
\\&=f(x)+\varepsilon(1)-\int_{1-\zeta}^1r^2\varepsilon'(r)\,dr.
\end{align*}
Integrating by parts and using our bounds on $\varepsilon(r)$ again, we obtain
\begin{equation}\label{eqn:hatdiff}
f\geq \hat f\geq f+\varepsilon(1)(2\zeta-\zeta^2)
\end{equation}
away from the support of $T_\varepsilon$.

We may now compute $\V(\hat\psi)$. From (\ref{eqn:hatdiff}) and (\ref{eqn:Cdiff}), using the fact that $\varepsilon(R)<0$, we obtain the upper bound $\V(\hat\psi)\leq\V(\psi)$. For the lower bound, let $C$ denote the collar support of $T_\varepsilon$. Using the lower bounds in (\ref{eqn:Cflb}) and (\ref{eqn:hatdiff}), as well as the fact that $\int_{\Sdg}\omega=1$, where $\omega=d\beta$ and on the collar $C$ we have $\beta=\frac{r^2}{2\pi}\,d\theta$,
\begin{align}
\V(\hat\psi)&\geq\int_0^{2\pi}\int_{1-\zeta}^1(f+\varepsilon(1))\cdot\frac{r}{\pi}\,dr\wedge d\theta+\int_{\Sdg\setminus C}\left(f+\varepsilon(1)(2\zeta-\zeta^2)\right)\omega \nonumber
\\&=\V(\psi)+\varepsilon(1)(2\zeta-\zeta^2).\label{eqn:Vdiff}
\end{align}
\end{example}

\subsection{Contact manifolds from Reeb return maps}\label{ss:3from2}

The following `suspension' construction provides a dictionary between the dynamics of an exact symplectomorphism $\psi$, which is freely isotopic to some $T(p,q)$ monodromy $\phi_{p,q}$ on $\Sdg$, and the Reeb dynamics of a contact form $\lambda_\psi$ on $(S^3,\xi_{std})$ adapted to the open book decomposition with binding $T(p,q)$ and Reeb return map $\psi$. The utility is that, by the Giroux correspondence (Theorem \ref{thm:giroux}), knowing only that $\lambda_\psi$ is adapted to $(T(p,q),\pi_{p,q})$ is enough to know its contactomorphism type. Along with a careful comparison between $\psi$ and $\phi_{p,q}$ near $\partial\Sdg$ (\`a la Corollary \ref{cor:FTDCrot}), this will allow us in \S\ref{ss:mac} to use our computation of knot-filtered ECH in Theorem \ref{thm:kech-intro} for the simpler contact form $\lambda_{p,q}$ of Example \ref{cor:CHdetails} and obtain consequences for Reeb orbits of $\lambda_\psi$.

In the course of the proof, we relate $\psi$ to the model map $\psi_{p,q}$, the Nielsen-Thurston representative of $\phi_{p,q}$ as in Example \ref{cor:CHdetails} (i),  via \emph{partial Dehn twists}, which we now define. We will use $T_d$ to denote the {partial Dehn twist} on the boundary collar annulus of $\Sdg$ in coordinates $(r,\theta)\in\left[\sqrt{1-\zeta^2},1\right]\times[0,2\pi]$ given by
\[
T_d(r,\theta)=(r,\theta+D_d(r)),
\]
where $D_d:\left[\sqrt{1-\zeta^2},1\right]\to[0,x]$ is a smooth monotone function which is identically zero near $\sqrt{1-\zeta^2}$ and identically $d$ near 1. In Example \ref{ex:Vtwist} we considered $T_\varepsilon$, which agrees with the case of $d=\varepsilon(1)<0$ so that the partial twist is a right-handed twist with some stretching of the collar radius from $\sqrt{1-\zeta^2}$ to $1-\zeta$ that is not significant. Here we will consider both $d<0$ and $d>0$ so that the partial twist might be left-handed. In both cases, we assume $D_d$ is monotone.

\begin{proposition}\label{prop:constructcontact}
Let $\psi$ be a symplectomorphism of $\Sdg$ which is freely isotopic to $\phi_{p,q}$ and which is isotopic relative to the boundary to $T_{2\pi d}\circ\psi_{p,q}$ for some $-\frac{1}{pq}<d<0$. Let $\theta_0=\frac{1}{pq}+d$ and assume that $f=f_{\psi,\beta,\theta_0}>0$. Then there is a tight contact form $\lambda_\psi$ on $S^3$ for which:
\begin{enumerate}[\em (i)]
\itemsep-.25em
    \item\label{1} The open book decomposition $(T(p,q),\pi_{p,q})$ of $S^3$ and associated abstract open book $(\Sdg,\phi_{p,q})$ 
    is adapted to $\lambda_\psi$. The return time of the flow of its Reeb vector field $R_\psi$ is $f$ and its return map is $\psi$.
    \item\label{2} The binding orbit $b$ has action $1$ and is elliptic with rotation number in the push off linking zero (equivalently, page; see \S\ref{ss:trivs}) trivialization
    \[
    \frac{1}{\theta_0}=\frac{1}{\frac{1}{pq}+d}.
    \]
    \item\label{3} There is a bijection $\P(\psi)\cup\{b\}\to\P(\lambda_\psi)$, which, when restricted to $\P(\psi)$, sends total action to symplectic action and period to intersection number with a single page (which equals the linking number with the binding orbit $b$).
    \item\label{4} $\operatorname{vol}(S^3,\lambda_\psi)=\mathcal{V}(\psi)$.
\end{enumerate}
\end{proposition}

\begin{remark}\label{rmk:theta0} The preceding proposition will be applied in the case
\begin{equation}\label{eqn:delta}
d=\frac{1}{pq+\Delta}-\frac{1}{pq}<0,
\end{equation}
where $\Delta>0$. Equation (\ref{eqn:delta}) arises because as $\Delta\to0$, the rotation number in conclusion (\ref{2}) must equal the rotation number used in our computation of the knot filtration, as in Lemma \ref{lem:pagetrivCZp}, that is,
\[
\frac{1}{\frac{1}{pq}+d}=pq+\delta_{\zb,L},
\]
which is equivalent to (\ref{eqn:delta}) with $\Delta=\delta_{\zb,L}$.

In the case of disk maps in \cite{HuMAC} and annular maps in \cite{weiler, weiler2}, a broader class of maps could be considered because the irrational ellipsoid could be used to compute unknot and Hopf link filtered ECH.  Changing the parameters of the irrational ellipsoid allows one to realize any positive irrational rotation angle. 
For torus knots in $S^3$ there is not a similarly direct means of understanding large changes to the rotation angle.

Furthermore, if $d$ is too small (i.e., less than $-1/pq$, so that it cannot be written as in (\ref{eqn:delta}) for some $\Delta$), the hypothesis on the action function, that $f>0$, will not be satisfied, as shown by the computation in Example \ref{ex:Vtwist}. This did not appear in \cite{HuMAC} and \cite{weiler, weiler2} because there was no restriction on the free homotopy class of the monodromy obtained in the analogue to conclusion (\ref{1}). 
\end{remark}

\begin{proof}
We will generalize the method of the proof for \cite[Prop.~2.1]{HuMAC} to the higher genus setting.

\hfill

\noindent\textbf{Step 1: Mapping torus.} Let $z$ denote the coordinate on $\Sigma$. Let
\[
M_\psi=\frac{[0,1]_t\times\Sigma}{(1,z)\sim(0,\psi(z))}
\]
denote the mapping torus of $\psi$. Let $\eta$ be a smooth function of $t$ for which
\begin{itemize}
    \item $\eta(0)=0$ and $\eta'(0)=0$,
    \item $\eta(1)=0$ and $\eta'(1)=1$, and
    \item $-\frac{\min(f)}{\max(f)}<\eta'(t)\leq1$.
\end{itemize}
Let $\lambda_0$ denote the following one-form on $[0,1]\times\Sigma$:
\[
\lambda_0(t,z)=(1-\eta'(t))f\,dt+\eta'(t)f\circ\psi\,dt+\beta+(t-\eta(t))\,df+\eta(t)\psi^*df.
\]
As in the proof of \cite[Step~1,~Prop.~3.1]{weiler} we obtain
\begin{enumerate}[(a)]\setlength\itemsep{0em}
    \item $\lambda_0(t,z)=f(z)\,dt+\beta$ for $z$ near $\partial\Sigma$,
    \item $\lambda_0\in\Omega^1(M_\psi)$,
    \item $\lambda_0$ is a contact form,
    \item $R_0$ is a positive multiple of $\partial_t$,
    \item it takes time $f_{\psi,\beta,\theta_0}(z)$ to flow under $R_0$ from $(0,z)$ to $(1,z)$, and
    \item $\vol(M_\psi,\lambda_0)=\omega(\Sigma)\V(\psi)=\V(\psi)$.
\end{enumerate}

Note that it is in this step where we require the hypothesis $f>0$.

\hfill

\noindent\textbf{Step 2: Closing the binding.} Consider the oriented coordinates $(\rho,\mu,\tau)$ on the solid torus $\mathbb{T}=\D^2(\zeta)\times(\R/2\pi\Z)$, where $\tau\in\R/2\pi\Z$, $\rho\in[0,\zeta]$, and $\mu\in\R/2\pi\Z$. Let $g:\operatorname{int}(M_\psi)\to \mathbb{T}$ be given by
\[
g(t,r,\theta)=\left(\sqrt{1-r^2},2\pi t,\theta+2\pi t\theta_0\right).
\]
Let $Y_\psi$ denote the union of $\operatorname{int}(M_\psi)$ with $\mathbb{T}$ via $g$; we will prove in the next step that $Y_\psi$ is diffeomorphic to $S^3$.

\hfill

\noindent\textbf{Step 3: Open book decomposition.} Denote by $B$ the set $\{\rho=0\}$. Let $P:Y_\psi-B\to S^1$ be given by $(t,z)\mapsto t$. The preimages $P^{-1}(t)$ are diffeomorphic to $\operatorname{int}(\Sdg)$. We claim that $P$ is a projection map for an open book decomposition with page $\Sdg$ and monodromy $\phi_{p,q}$, which will prove that $Y_\psi$ and $S^3$ are diffeomorphic.

The meridional direction near $B$ (corresponding to $\partial\Sdg$) is given by $\partial_\mu$, which extends to the mapping torus part of $Y_\psi$ near $\partial\Sdg$ as $\frac{1}{2\pi}\partial_t-\theta_0\partial_\theta$. In $\operatorname{int}(M_\psi)$, the direction $\partial_t$ is transverse to the fibers of $P$. Choose a smooth monotone function
\[
D_{-\theta_0}:[\sqrt{1-\zeta^2},1]\to[-\theta_0,0] \mbox{ with } D_{\theta_0}\equiv0 \mbox{ near } \sqrt{1-\zeta^2} \mbox{ and } D_{\theta_0}(1)=-\theta_0.
\]
The smooth vector field
\[
V:=\frac{1}{2\pi}\partial_t+D_{\theta_0}(r)\partial_\theta
\]
interpolates between these two directions as $r$ varies from $\sqrt{1-\zeta^2}$ to $1$, and remains transverse to the fibers of $P$ throughout.

By the fact that $\operatorname{int}(M_\psi)$ is the mapping torus of $\psi$, the diffeomorphism induced on $P^{-1}(0)$ by the flow of $\partial_t$ on $Y-B$ is $\psi$.

Near $B$,
\[
g^{-1}(\rho,\mu,\tau)=\left(\frac{\mu}{2\pi},\sqrt{1-\rho^2},\tau-\theta_0\mu\right),
\]
therefore the flow of $\partial_\mu=V$ for time $2\pi$ will send $(t,r,\theta)$ to 
\[
\left(t+1,r,\theta-2\pi\theta_0\right)\sim\left(t,r,\theta\right),
\]
which is $T_{-2\pi\theta_0}\circ\psi$. By our hypothesis that $\psi$ is isotopic relative to the boundary to $T_{2\pi d}\circ\psi_{p,q}$ and that $\theta_0=\frac{1}{pq}+d$, the open book decomposition $(B,P)$ of $Y_\psi$ has abstract open book $(\Sdg,\phi_{p,q})$. By Example \ref{cor:CHdetails}, we may conclude that $Y_\psi$ and $S^3$ are diffeomorphic, and that $(B,P)=(T(p,q),\pi_{p,q})$.

\hfill

\noindent\textbf{Step 4: Contact form.} The gluing in Step 3 induces a contact form $\lambda_\psi$ on $S^3$ as follows. Define $\lambda_\psi=\lambda_0$ from Step 1 on $\operatorname{int}(M_\psi)$. We will define an extension to the binding $B=T(p,q)$. In the $\mathbb{T}$ coordinates,
\[
\lambda_0=\theta_0\,dt+\frac{r^2}{2\pi}\,d\theta=\frac{\theta_0}{2\pi}\,d\mu+\frac{\left(\sqrt{1-\rho^2}\right)^2}{2\pi}(-\theta_0\,d\mu+d\tau)=\frac{\theta_0\rho^2}{2\pi}\,d\mu+\frac{1-\rho^2}{2\pi}\,d\tau.
\]
Converting to Cartesian coordinates $x=\rho\cos\mu, y=\rho\sin\mu$ gives us
\begin{align*}
\lambda_0&=\frac{\theta_0(x^2+y^2)}{2\pi}\left(-\frac{y}{x^2+y^2}\,dx+\frac{x}{x^2+y^2}\,dy\right)+\frac{1-x^2-y^2}{2\pi}\,d\tau
\\&=\frac{\theta_0}{2\pi}(-y\,dx+x\,dy)+\frac{1-x^2-y^2}{2\pi}\,d\tau,
\end{align*}
which is perfectly well defined when $\rho=0$. Therefore we may extend $\lambda_0$ over the binding $B=T(p,q)=\{\rho=0\}$ by setting it equal to $\frac{1-\rho^2}{2\pi}\,d\tau$ there, and we call the extension $\lambda_\psi$.

We know that $\lambda_0$ is contact on $\operatorname{int}(M_\psi)$ by conclusion (c) of Step 1. To check that the extension $\lambda_\psi$ is contact when $\rho=0$, we compute
\[
d\lambda_\psi=\frac{\theta_0\rho}{\pi}\,d\rho\wedge d\mu-\frac{\rho}{\pi}\,d\rho\wedge d\tau,
\]
therefore
\[
\lambda_\psi\wedge d\lambda_\psi=\left(\frac{\theta_0\rho^3}{2\pi^2}+\frac{\theta_0\rho(1-\rho^2)}{2\pi^2}\right)\,d\rho\wedge d\mu\wedge d\tau=\frac{\theta_0\rho}{2\pi^2}\,d\rho\wedge d\mu\wedge d\tau.
\]
Because $\theta_0=\frac{1}{pq}+d>0$, the orientation induced by $\lambda_\psi\wedge d\lambda_\psi$ agrees with the orientation $\rho\,d\rho\wedge d\mu\wedge d\tau$ of $\mathbb{T}$, which agrees with the orientation of $\operatorname{int}(M_\psi)$ as a mapping torus because $g$ is orientation-preserving. Therefore the one-form $\lambda_\psi$ is a contact form on $S^3$, proving the first assertion.

The Reeb vector field of $\lambda_\psi$ near $\rho=0$ is
\[
R_\psi=\frac{2\pi}{\theta_0}\partial_\mu+2\pi\partial_\tau,
\]
which in Cartesian coordinates is
\[
R_\psi=\frac{2\pi}{\theta_0}(-y\partial_x+x\partial_y)+2\pi\partial_\tau
\]
and is therefore defined over the circle $\rho=0$ (even though $\mu$ is not defined there). Thus the circle $\rho=0$ is an orbit of $R_\psi$, oriented in the $\partial_\tau=\partial_\theta$ direction. Call this orbit $b$.

Using conclusions (d, e, f) of Step 1 and the fact that $M_\psi$ is the mapping torus of $\psi$, we have now proved assertions (\ref{1}), (\ref{3}), and (\ref{4}). (We obtain conclusion (iii) because the Reeb orbits are integral curves of $\partial_t$, and intersection number with a page equals linking number with the binding because pages are Seifert surfaces for the binding.) We may also compute at this point that $\A(b)=1$, proving part of assertion (\ref{2}).

\hfill

\noindent\textbf{Step 5: Binding orbits.} It remains to prove the rest of assertion (\ref{2}), namely that $b$ is elliptic and has rotation number $1/\theta_0$.

First note that the contact structure $\xi_\psi|_b$ is the bundle of disks $\tau=const$, oriented by $d\lambda_\psi$. This is the standard orientation with oriented coordinates $\rho,\mu$ on each disk. Therefore, in the trivialization of $\xi_\psi|_b$ as the subbundle of the tangent bundle 
\[
T\mathbb{T}=T\mathbf{D}^2\times T(R/2\pi\Z)=\xi_\psi|_b\times Tb
\]
 consisting of tangent planes to the disks, the rotation number of $b$ is the coefficient of $\partial_\mu$ in $R_\psi$ divided by the coefficient of $\partial_\tau$, which is $1/\theta_0$. Because the zero page consists of the points $\mu=0$ near the image of $b$, this trivialization of $\xi_\psi|_b$ is precisely the trivialization for which a constant pushoff of $b$ does not intersect the zero page. That is, it is the page trivialization. See \S\ref{ss:trivs} and in particular the paragraph before Remark \ref{rmk:trivs} for a more extensive discussion of the various trivializations in play.

Finally, the orbit $b$ is elliptic as its return map may be directly computed from either of the formulas for $R_\psi$ in Step 4. 

We have proved the rest of assertion (\ref{2}) and therefore the proposition.
\end{proof}

\subsection{Reduction of mean action bounds to contact geometry}\label{ss:mac}

We now prove our surface dynamics result, Theorem \ref{thm:Calabi}, using results on ECH from later in the paper as a black box. For convenience, we restate the theorem here.

\begin{theorem} Let $\psi$ be an area-preserving diffeomorphism of $(\Sdg,\omega)$, wherein:
\begin{itemize}
\itemsep-.25em
\item $g=(p-1)(q-1)/2$,
\item $\psi$ is isotopic relative to the boundary to $T_{2\pi d}\circ\psi_{p,q}$ for some $-\frac{1}{pq}<d\leq0$,\footnote{Because $\psi_{p,q}$ is the periodic Nielsen-Thurston representative of the monodromy $\phi_{p,q}$, we therefore have that $\psi$ is freely isotopic to $\phi_{p,q}$.}
\item $\theta_0=\frac{1}{pq}+d$ and $f=f_{\psi,\beta,\theta_0}>0$ for any primitive $\beta$ of $\omega$ satsifying (\ref{eqn:betaboundary}), and
\item $\V(\psi)<pq\cdot\theta_0^2$.
\end{itemize}
Then we have
\[
\inf\left\{\frac{\A(\gamma)}{\ell(\gamma)} \ \bigg \vert \ \gamma\in\P(\psi)\right\}\leq\sqrt{\frac{\V(\psi)}{pq}}.
\]
\end{theorem}

\begin{proof}[Proof of Theorem \ref{thm:Calabi}]
\noindent\textbf{Step 1:} Assume $\theta_0\in\R\setminus\Q$ (in particular, $d<0$). Then $\psi$ satisfies the hypotheses of Proposition \ref{prop:constructcontact}, providing us with a contact form $\lambda_\psi$. Now, because of the assumption $\V(\psi)<pq\cdot\theta_0^2$, we may apply Theorem \ref{thm:toruslinking6} to $\lambda_\psi$. We obtain
\[
\inf\left\{\frac{\A(\gamma)}{\ell(\gamma)} \ \bigg \vert \ \gamma\in\P(\psi)\right\}\leq\sqrt{\frac{\V(\psi)}{pq}},
\]
freely using the conclusions of Proposition \ref{prop:constructcontact}.

\hfill

\noindent\textbf{Step 2:} We now drop the assumption that $\theta_0\in\R\setminus\Q$. Apply Example \ref{ex:Vtwist} with $\zeta$ small enough that $T_\varepsilon$ is supported in the region where $\psi$ is a rotation,
\[
\sqrt{\frac{\V(\psi)}{pq}}-\theta_0<\varepsilon(1)<0,
\]
and $\theta_0+\varepsilon(1)\in\R\setminus\Q$. Let $\hat\psi$ denote the map $T_\varepsilon\circ\psi$, and use a hat to denote all its associated functions and quantities. Note that we may freely make $\zeta$ as close to zero and $\theta_0+\varepsilon(1)$ as close to $\sqrt{\V(\psi)/pq}$ as we like.

By Example \ref{ex:Vtwist}, we have:
\begin{enumerate}[{(i)}]
\item On the collar domain of $T_\varepsilon$, $\hat f\geq\theta_0+\varepsilon(1)$ by (\ref{eqn:Cflb}).
\item $\V(\psi)\geq\V(\hat\psi)\geq\V(\psi)+\varepsilon(1)(2\zeta-\zeta^2)$ by (\ref{eqn:Vdiff}) and the discussion beforehand.
\item $\hat\theta_0=\theta_0+\varepsilon(1)$, so by the bounds on $\varepsilon(1)$, we have $\V(\hat\psi)<pq\cdot\hat\theta_0^2$.
\item If $\gamma$ is a periodic orbit of $\psi$ (necessarily away from the collar on which $\psi$ is a rotation, because $\theta_0\in\R\setminus\Q$), then it is a periodic orbit of $\hat\psi$, so by adding the inequalities from (\ref{eqn:hatdiff}) applied to the value of the action function $f$ at each of the $\ell(\gamma)$ points in $\gamma$,
\[
\A(\gamma)\geq\hat\A(\gamma)\geq\A(\gamma)+\ell(\gamma)\varepsilon(1)(2\zeta-\zeta^2).
\]
\end{enumerate}

Because $\hat\psi$ satisfies the assumptions of Step 1, we use (ii)-(iii) to obtain
\begin{equation}\label{eqn:worsebound}
\inf\left\{\frac{\hat\A(\gamma)}{\ell(\gamma)} \ \bigg \vert \ \gamma\in\P(\hat\psi)\right\}\leq\sqrt{\frac{\V(\hat\psi)}{pq}}\leq\sqrt{\frac{\V(\psi)}{pq}}.
\end{equation}
By (i) and the fact that $\theta_0+\varepsilon(1)>\sqrt{\V(\psi)/pq}$, the infimum on the left hand side of (\ref{eqn:worsebound}) must be realized by some $\gamma$ away from the collar support of $T_\varepsilon$, thus by (iv),
\[
\inf\left\{\frac{\A(\gamma)}{\ell(\gamma)} \ \bigg \vert \ \gamma\in\P(\psi)\right\}\leq\sqrt{\frac{\V(\psi)}{pq}}-\varepsilon(1)(2\zeta-\zeta^2).
\]
Sending $\zeta\to0$ provides us with the desired conclusion.

\end{proof}

\begin{remark}\label{rmk:2.4}
One would prefer the stronger upper bound of $\V(\psi)$ instead of $\sqrt{\V(\psi)/pq}$ under the weaker assumption that $\V(\psi)<\theta_0$. To do so, we would require a similar strengthening of the Reeb dynamics result in Theorem \ref{thm:toruslinking6}, but it is not yet understood how to compute the full knot filtration on ECH when the rotation number is not $pq$. Note that requiring the stronger hypothesis $\V(\psi)<pq\cdot\theta_0^2$ means that we cannot decrease $\theta$ by a large amount, all the way to $\V(\psi)$, in order to improve the upper bound from Reeb dynamics via the twisting argument in \cite{HuMAC}.

\end{remark}

\section{Reeb currents of $T(p,q)$ fibrations}\label{s:topology}

As detailed in Example \ref{cor:CHdetails}, the open book decomposition with binding the right-handed $T(p,q)$ torus knot supports a contact form $\lambda_{p,q}$ whose Reeb vector field integrates to the Seifert fibration of $S^3$ with regular fiber $T(p,q)$. The contact form $\lambda_{p,q}$ is strictly contactomorphic to the connection 1-form on the prequantization orbibundle over complex one dimensional weighted projective space $\CP_{p,q}^1$ with Euler class $-\frac{1}{pq}$, as we explain in \S\ref{ss:OBDs}. Using the latter description, in \S\ref{ss:morse} we perturb $\lambda_{p,q}$ to be nondegenerate up to a given action level, using (the lift of) an appropriate Morse-Smale function on the base orbifold $\CP^1_{p,q}$. We then elucidate the associated perturbed Reeb dynamics and Conley-Zehnder indices (up to a large action threshold) in \S\ref{ss:Reeb} and \S\ref{ss:CZ}, respectively; the fibers which project to critical points comprise the Reeb currents generating the action filtered ECH chain complexes we study in \S\ref{s:generalities}-\ref{s:spectral}.

\subsection{Positive $T(p,q)$ fibrations of $S^3$}\label{ss:OBDs}

Let $\fp: Y \to \cO$ be an oriented three dimensional Seifert fibration with oriented base  of genus $g$ and normalized Seifert invariants $Y(g; b; (a_1,b_1),...,(a_r,b_r))$, per the description in \cite{orlik, lrbook} and using the conventions in \cite{lm}. The \emph{Euler class} of $Y$ is defined by the rational number $e(Y):= -b - \sum_{i=1}^r \frac{b_i}{a_i},$ and it does not depend on the choice of normalized or unnormalized Seifert invariants.  The unaccompanied $b$ is used for a fiber of type $(1,b)$, which does not project to a singular point.

\begin{remark}
Lisca-Mati\'c classified which Seifert fibered 3-manifolds admit contact structures transverse to their fibers, which moreover are shown to be universally tight, see \cite[Thm.~1.3, Cor.~2.2, Prop.~3.1]{lm}.  An oriented Seifert fibered 3-manifold admits a positive, $S^1$-invariant transverse contact structure if and only if $e(Y) < 0$.  So that we are in agreement with Giroux's conventions for his classification of tight contact structures transversal to the fibration of a circle bundle over a surface in terms of the Euler class, we utilize a negative Euler class, as in \cite{lm, lrbook}.   A Seifert fibration of a closed orientable 3-manifold $Y$ over an oriented orbifold $\cO$ can be realized by a Reeb flow if and only if the Euler class of the fibration is nontrivial by \cite[Thm 1.4]{kl}.\footnote{Our Euler class is referred to as the real Euler class by Kegel-Lange, who also use the opposite sign convention from us.}
\end{remark}

We recall the description of the classical Seifert fibering of $S^3$ along the positive torus knot $T(p,q)$, as in \cite{ch-sfs, moser}, which provides a partition of $S^3$ into orbits over the orbifold 2-sphere wherein the $z_1$ and $z_2$-axes correspond to singular fibers and each principal fiber is a positive torus knot.
 
   \begin{proposition}\label{prop:SeifertInvts}
The $S^1$-action $e^{2\pi it}\cdot(z_1,z_2)=\left(e^{2\pi pit}z_1,e^{2\pi qit}z_2\right)$
generates the Seifert fibration of $S^3$ given by ${Y\left(0; -1; (p, p-m), (q, n)\right),}$ where $m, n \in \Z$ such that $qm-pn=1$, and has Euler class ${e(Y)= -\frac{1}{pq}.}$
\end{proposition}

Our subsequent computation of the Conley-Zehnder indices will make use of the following properties of weighted complex one dimensional projective space.  Further details can be found in \cite[\S 2]{kech}.

\begin{example}\label{symporb}
Weighted complex one dimensional projective space $\CP^1_{p,q}$ is defined\footnote{Weighted projective space $\CP^1_{p,q}$ is in fact an algebraic variety admitting two distinct orbifold structures.  The other orbifold structure is realized by $\CP^1/ (\Z_{p} \times \Z_{q})$,  as described on \cite[\S 3.a.2]{mann}.  To be absolutely precise,  one should decorate the algebraic variety $\CP_{p,q}^1$ to indicate that it is equipped with a specific orbifold structure. We forgo this as we will only consider one orbifold structure.} by the quotient of the unit sphere $S^3 \subset \C^2$ by the almost free action of $S^1 \subset \C$ of the form
\[
e^{{2\pi}it} \cdot (z_1,z_2) = \left(e^{p{2\pi}it}z_1,e^{q{2\pi}it}z_2 \right).
\]
The following 1-form is invariant under the above $S^1$-action
\begin{equation}
\lambda_{p,q} = \frac{\frac{1}{2\pi}\lambda_0}{p|z_1|^2 + q|z_2|^2}, \mbox{ where } \lambda_0 = \frac{i}{2} \sum_{j=1}^2 \left( z_j d\bar{z}_j - \bar{z}_jd z_j \right).
\end{equation}
As a result, $\omega_{p,q} := d\lambda_{p,q}$ descends to an orbifold symplectic form on $\CP^1_{p,q}$, and its cohomology class satisfies ${[\omega_{p,q}] = \frac{1}{pq}} \mbox{ in }H^2(\CP^1_{p,q}, \Q)\cong \Q$,  see \cite[Lem.~4.2]{hong-cz}.   (The orbifold chart that provides this desired orbifold structure on $\CP^1_{p,q}$ is described in \cite[Prop.~3.1]{mann}.)\end{example}

\begin{definition}\label{def:orbX}
We define the \emph{orbifold Euler characteristic} by
\[
\chi^{orb}(\cO) = \sum_S (-1)^{\dim(S)}\frac{\chi(S)}{|\Gamma(S)|}.
\]
where the sum is taken over all strata $S$ of the stratification of $\Sigma,$ $\chi(S)$ is the ordinary Euler characteristic, and $\Gamma(S)$ is the isotropy.  (This agrees with Boyer-Galicki  \cite[\S 4.3-4.4]{bgbook}.)
\end{definition}

\begin{example}\label{ex:orbX}
  We can give a Riemann surface $(\Sigma; z_1,...,z_k)$ of genus $g$ with $k$ marked points the structure of an orbifold by defining local uniformizing systems $(\tilde{U}_j, \Gamma_{a_j}, \varphi_j)$ centered at the point $z_j$, where $\Gamma_{a_j}$ is the cyclic group of order $a_j$ and $\varphi_j:\tilde{U}_j \to U_j = \tilde{U}_j/\Gamma_{a_j}$ is the branched covering map $\varphi_j(z)=z^{a_j}$. The orbifold Chern number agrees with the orbifold Euler characteristic:
\[
c_1^{orb}(\Sigma; z_1,...,z_k) = \chi^{orb}(\Sigma; z_1,...,z_k) = 2- 2g -k + \sum_{j=1}^k \frac{1}{a_j}.
\]
Thus for weighted projective space, 
\begin{equation}\label{eq:orbc}
{c_1^{orb}(\CP_{p,q}^1)=\frac{p+q}{pq}.}
\end{equation}
\end{example}

A \emph{principal $S^1$-orbibundle} or \emph{prequantization orbibundle} whose total space $Y$ is a manifold is the same as an almost free $S^1$-action on $Y$.  The \emph{connection 1-form} $A$ induces a contact form, as $iA(R) =1$ and $d(iA)(R, \cdot)=0$, where $R$ is the derivative of the $S^1$-action.  As is the case for prequantization bundles over smooth surfaces, the \emph{curvature form} satisfies $dA=i\mathfrak{p}^*\omega$, is $S^1$-invariant, and vanishes in the the $S^1$-direction of $R$.  Here $\omega$ is a symplectic form on the base orbifold, and ${-\frac{1}{2\pi}}[\omega]$ represents the  Euler class $e \in H^2_{\mbox{\tiny \em orb}}(\cO;\R)$ of the orbibundle.  Note that the cohomology class of $\omega$ can be canonically identified with an element in $H^2_{\mbox{\tiny \em orb}}(\cO;\R) $ via the equivariant de Rham theorem, cf. \cite[\S 3.2]{kl}.   Amalgamating the results of Kegel-Lange \cite{kl} with Cristofaro-Gardiner -- Mazuchelli \cite{cgm} provides the following classification result.  See \cite[\S 2]{kech} for additional details.

\begin{theorem}{\em \cite[Thm.~1.1, Thm.~1.4 Cor.~1.6]{kl}, \cite{cgm} }
The classification of prequantization orbibundles up to strict contactomorphism coincides with the classification of Seifert fibrations $Y \to \cO$ of orientable, closed 3-manifolds with nonvanishing Euler class.\end{theorem}

\subsection{Morse-Smale functions for $T(p,q)$ fibrations}\label{ss:morse}
The Reeb vector field associated to $\lambda_{p,q}$ realizing the associated Seifert fiber space of Proposition \ref{prop:SeifertInvts} as a principal $S^1$-orbibundle over $\CP^1_{p,q}$, as in Example \ref{symporb}, is given by
\[
R_{p,q} := 
2\pi i \left(p \left (z_1 \frac{\partial}{\partial z_1} - \bar{z}_1 \frac{\partial}{\partial \bar{z}_1}\right) + q \left (z_2 \frac{\partial}{\partial z_2} - \bar{z}_2 \frac{\partial}{\partial \bar{z}_2}\right)\right).
\]
A \emph{principal orbit} of a prequantization orbibundle is any orbit having the longest action (equivalently period) among all the periodic orbits, assuming one exists.   The principal orbits of the Reeb vector field $R_{p,q}$ are precisely those projecting to the nonsingular points of the base $\CP^1_{p,q}$ and have action 1.  We note that the $T(p,q)$ binding of the associated open book is a principal orbit.

There are two non-principal orbits of interest, which respectively project to each of the singular points of $\CP^1_{p,q}$. These orbits are said to be \emph{exceptional} and their actions are given by $1/|\Gamma_{x}|$ where $|\Gamma_{x}|$ is the {order} of the cyclic isotropy group at $x$. (A nonsingular point has $|\Gamma_{x}|=1$.) Denote 
\begin{equation}\label{eq:p}
\mathpzc{p}(t) = \left( e^{2\pi i t}, 0\right), \ \ \ t \in [0,1];
\end{equation}
we have that $\mathpzc{p}(t)$ projects to the orbifold point  whose isotropy group is $ \Z/p\Z$.  Similarly, denote
\begin{equation}\label{eq:q}
\mathpzc{q}(t) = \left(0, e^{2 \pi i t} \right), \ \ \ t \in [0,1];
\end{equation}
we have that $\mathpzc{q}(t)$ projects to the orbifold point whose isotropy group is $ \Z/q\Z$.

We use the following Morse-Smale functions $\mathpzc{H}_{p,q}$,  which are $C^2$ close to one, on the orbifolds $\CP^1_{p,q}$ to perturb $\lambda_{p,q}$.  

\begin{remark}
The Morse-Smale  function $\mathpzc{H}_{p,q}$ that we construct admits exactly 4 critical points.  The binding $\mathpzc{b}$ projects to the maximum, the principal Reeb orbit $\mathpzc{h}$ projects to the saddle,  $\mathpzc{p}$ projects to a {minimum} at the singular point with isotropy  $\Z/p\Z$, and $\mathpzc{q}$ projects to a minimum at the singular point with isotropy  $\Z/q\Z$. For $T(2,q)$, there is  different Morse-Smale function ${H}_{2,q}$ that we utilized to perturb $\lambda_{2,q}$ in \cite[\S 2]{kech}.  The benefit of using ${H}_{2,q}$ was that the ECH differential vanished for index reasons, but it does not work for $p\neq 2$.  \end{remark}

The construction of these orbifold Morse functions $\mathpzc{H}_{p,q}$ on the orbifolds $\CP^1_{p,q}$ is as follows, see  Figure \ref{fig:morsepq} for an illustration. 

\begin{proposition}[Morse functions $\mathpzc{H}_{p,q}$]\label{prop:morsep}
There exists a Morse function $\mathpzc{H}_{p,q}$ on $\CP^1_{p,q}$, such that $|\mathpzc{H}_{p,q}|$ is $C^2$ close to one, with exactly four critical points, such that the binding projects to the nonsingular index 2 critical point, there is a nonsingular index 1 critical point, and there are index 0 critical points at each of the two isotropy points. There are stereographic coordinates defined in a small neighborhood of $x\in X:=\op{Crit}(H_{p,q})$ in which $H_{p,q}$ takes the form  
\begin{enumerate}[\em (i)]
 \itemsep-.35em
    \item $r_0(u,v):=(u^2+v^2)/2-1$,\,\,\, if $x\in X_0$
    \item $r_1(u,v):=(v^2-u^2)/2$,\,\,\,\,\,\,\,\,\,\,\,\,\, if $x\in X_1$
        \item $r_2(u,v):=1-(u^2+v^2)/2$,\,\,\,\,\,if $x\in X_2$
\end{enumerate}
{Moreover, $\mathpzc{H}_{p,q}$ is invariant under the return map $\psi_{p,q}$   of the Reeb vector field $R_{p,q}$.} 
\end{proposition}

\begin{figure}[h]
 \begin{center}
 \begin{overpic}[width=.5\textwidth, unit=1.75mm]{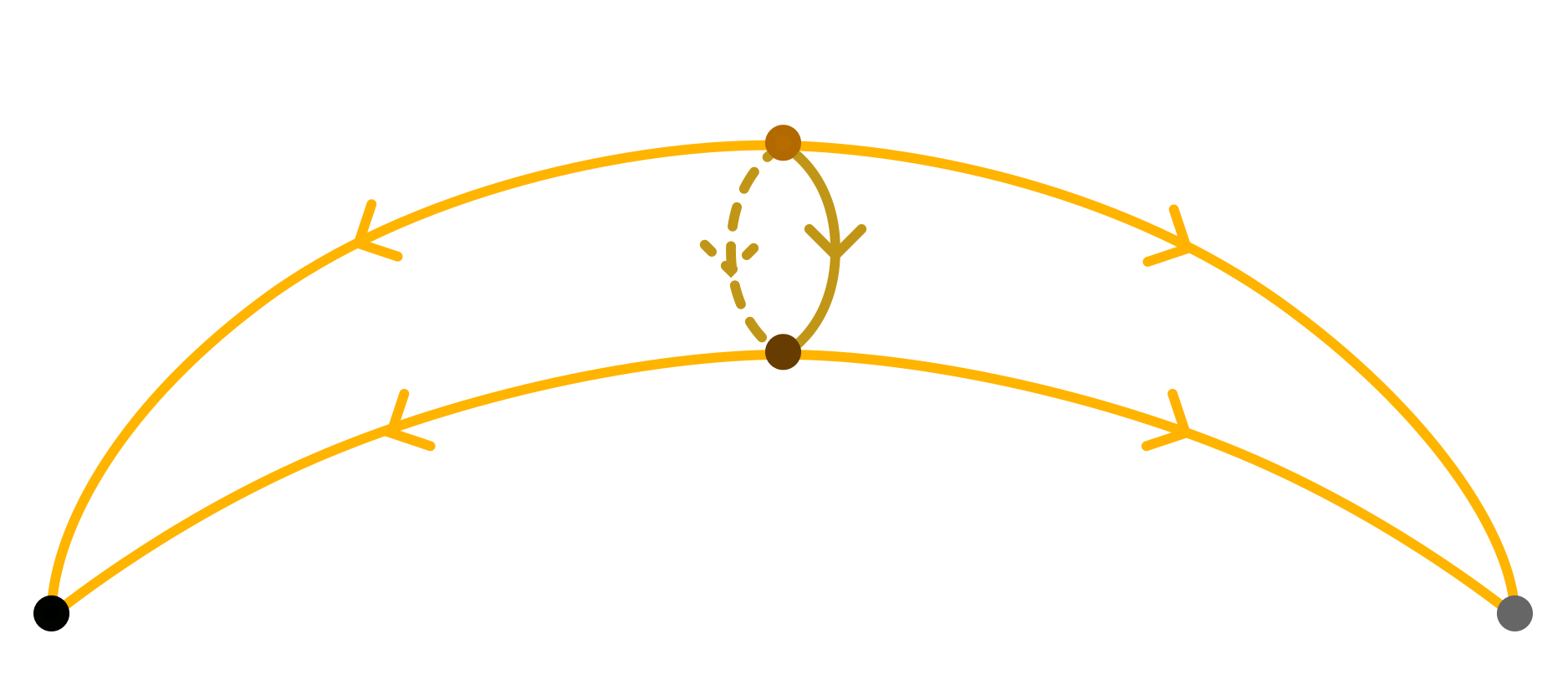}
\put(1,0){$\zp$}
\put(45,0){\textcolor{bg}{$\zq$}}
\put(22.5,18){ \textcolor{blb}{$\zb$}}
\put(22.5,7){\textcolor{bdb}{$\zh$}}
\end{overpic}
\end{center}
\caption{ Here we have depicted the gradient flow of $\mathpzc{H}_{\zp,\zq}$ on $\CP^1_{p,q}$ and labelled the fibers of the prequantization orbibundle projecting to the critical points. The minima have isotropy $\Z/p$ and  isotropy \textcolor{bg}{$\Z/q$}; the fibers $\zp$ and $\textcolor{bg}{\zq}$ project to the respective minima. The fiber \textcolor{bdb}{$\zh$} projects to the index 1 critical point. The binding fiber \textcolor{blb}{$\zb$} projects to the maximum.  The arrangement of these gradient flow lines is key to the proof of Proposition \ref{prop:pqM}.}
\label{fig:morsepq}
\end{figure}

\begin{proof}
This follows from a simplification of the construction given in the proof of \cite[Lem. 2.16]{leo}.  In particular, we only need their local construction of a rotationally invariant Morse function in terms of geodesic discs centered at the desired critical points which correspond to the orbifold points.  (Their statement is one of existence for globally $G$-invariant Morse functions on $S^2$, where $G$ is a  subgroup of $\op{SO}(3)$.)
\end{proof}

The following lemma guarantees that the Morse functions in Proposition \ref{prop:morsep} are Smale with respect to $\omega_{p,q}(\cdot, j\cdot)$ restricted to $\CP^1_{p,q}$.

\begin{lemma}\label{lemma:smale}
If $H$ is a Morse function on a 2-dimensional {orbifold} $\Sigma$ with isolated quotient singularities such that $H(p_1)=H(p_2)$ for all $p_1, p_2\in\mbox{\em Crit}(H)$ with Morse index 1, then $H$ is Smale, given any metric on $S$.
\end{lemma}
\begin{proof}
Given metric $g$ on $\Sigma$, $H$ fails to be Smale with respect to $g$ if and only if there are two distinct critical points of $H$ of Morse index 1 that are connected by a gradient flow line of $H$. Because all such critical points have the same $H$ value, no such flow line exists.
\end{proof}

\subsection{Perturbed Reeb dynamics for $T(p,q)$ fibrations}\label{ss:Reeb}
As in \cite{jo2, leo, preech, kech}, we use the Morse functions $\mathpzc{H}_{p,q}$ to perturb the degenerate contact form $\lambda_{p,q}$:
\begin{equation}\label{eqn:lambdapertR}
\begin{split}
\lambda_{p,q,\varepsilon}:= & (1+\varepsilon \frak{p}^*\mathpzc{H}_{p,q})\lambda_{p,q}. \end{split}
\end{equation}

By adapting a standard computation, \cite[Prop.~4.10]{jo2}, to the $S^1$-orbibundle framework, as elucidated in \cite[Rem.~2.1]{LT}, we obtain the following description of the perturbed Reeb vector field.

\begin{lemma}
The Reeb vector field of $\lambda_{p,q,\varepsilon}$ is given by
\begin{equation}
\label{perturbedreeb1}
R_{p,q,\varepsilon}=\frac{R_{p,q}}{1+\varepsilon \fp^*\mathpzc{H}_{p,q}} + \frac{\varepsilon \widetilde{X}_{\mathpzc{H}_{p,q}}}{{(1+\varepsilon \fp^*\mathpzc{{H}_{p,q}})}^{2}}, 
\end{equation}
where $\widetilde{X}_{\mathpzc{H}_{p,q}}$ denotes the horizontal lift of the Hamiltonian vector field\footnote{We use the convention $\omega(X_{{H}}, \cdot) = d{H}.$} $X_{\mathpzc{H}_{p,q}}$ on $\CP^1_{p,q}$.  
\end{lemma}

We additionally have:

\begin{lemma}\label{lem:efromL}
Given Morse functions $\mathpzc{H}_{p,q}$ as constructed in Proposition \ref{prop:morsep}, which are ${C^2}$ is close to 1:  
\begin{enumerate}[\em (i)]
\itemsep-.25em
\item For each $L>0$, there exists $\varepsilon(L)>0$ such that for all $\varepsilon<\varepsilon(L)$, all Reeb orbits of $R_{p,q,\varepsilon}$ with $\mathcal{A}(\gamma) < L$ are nondegenerate and project to critical points of $\mathpzc{H}_{p,q}$, where $\mathcal{A}$ is computed using $\lambda_{p,q,\varepsilon}$.
\item The action of a Reeb orbit $\gamma_x^{k|\Gamma_x|}$ of $R_{p,q,\varepsilon}$ over a critical point $x$ of $\mathpzc{H}_{p,q}$ is proportional to the length of the fiber, namely
\[
\mathcal{A}\left(\gamma_x^{k|\Gamma_x|}\right) = \int_{\gamma_x^{k|\Gamma_x|}} \lambda_{p,q,\varepsilon} = k (1+\varepsilon \fp^*\mathpzc{H}_{p,q}(x)).
\]
\end{enumerate}
\end{lemma}

We choose $\varepsilon(L)$ so that the embedded orbits realizing the generators of $ECC^L_*(S^3,\lambda_{p,q,\varepsilon(L)};J)$ consist solely of fibers above critical points of $\mathpzc{H}_{p,q}$ and that $\varepsilon(L)\sim\frac{1}{L}$, as in  \cite[Lem. 3.1]{preech}.  Moreover, a sufficiently large choice of $L$ guarantees that the generators of $ECC^L_*(S^3,\lambda_{p,q,\varepsilon(L)});J)$ are comprised solely of Reeb orbits that are fibers above critical points.  As in \cite[\S 3.4]{preech}, we can form the following direct system from these action filtered complexes.

\begin{proposition}\label{prop:directlimitcomputesfiberhomology} With $(S^3,\lambda_{p,q})$ and $\varepsilon(L)$ as discussed above, there is a direct system formed by $\{ ECH^L_*(Y,\lambda_{p,q,\varepsilon(L)})\}$. The direct limit $\lim_{L\to\infty}ECH^L_*(S^3,\lambda_{p,q,\varepsilon(L)})$ is the homology of the chain complex generated by Reeb currents $\{(\alpha_i,m_i)\}$ where the $\alpha_i$ are fibers above critical points of $\mathpzc{H}_{p,q}$.
\end{proposition}

By taking a direct limit, we can compute ECH via Proposition \ref{prop:directlimitcomputesfiberhomology}; this involves passing to filtered Seiberg-Witten Floer cohomology as explained in \cite[\S 7]{preech} and further refined for the purposes of computing knot filtration in \cite[\S 7]{kech}.  We obtain the following classification of fiber Reeb orbits via the Conley-Zehnder computations that will be carried out in \S \ref{ss:CZ}. 

\begin{lemma}\label{lem:orbitseh} Up to large action $L(\varepsilon)$, as determined by Proposition \ref{lem:efromL}, the generators of $ECC_*^{L(\varepsilon)}(S^3,\lambda_{p,q,\varepsilon}, J)$ are of the form $\mathpzc{b}^B\mathpzc{h}^H\mathpzc{p}^P\mathpzc{q}^Q$, where $B,P,Q \in \Z_{\geq 0}$ and $H=0,1$. Moreover,
\begin{itemize}
\itemsep-.25em
\item $\mathpzc{b}$ is elliptic and projects to the nonsingular index 2 critical point of $\mathpzc{H}_{p,q}$;
\item $\mathpzc{h}$ is positive hyperbolic and projects to to the nonsingular index 1  critical point of $\mathpzc{H}_{p,q}$;
\item $\mathpzc{p}$ is elliptic  and projects to the index 0 critical point of $\mathpzc{H}_{p,q}$ with isotropy $\Z/p\Z$;
\item $\mathpzc{q}$ is elliptic  and projects to the index 0 critical point of $\mathpzc{H}_{p,q}$ with isotropy $\Z/q\Z$.
\end{itemize}
\end{lemma}

\begin{remark}\label{rem:obd}
The contact forms $\lambda_{p,q, \varepsilon}$ and $\lambda_{p,q}$ are supported by the open book decomposition $(T(p,q),\pi)$ because $\widetilde X_{\mathpzc{H}_{p,q}}$ is tangent to the pages.  From (\ref{perturbedreeb1}), it follows that $R_{p,q,\varepsilon}$ is positively transverse to the pages while remaining tangent to the binding because $\widetilde X_{\mathpzc{H}_{p,q}}$ vanishes along the binding.
\end{remark}

We conclude by computing the following linking numbers, which will be relevant to our later computation of the knot filtration on ECH. 
\begin{corollary}\label{cor:linkingnumber} 
For $T(p,q)$, we have:
\begin{enumerate}[\em (i)]
\itemsep-.35em
\item $\ell(\zp,\zq)=1$,
\item $\ell(\zb,\zp)= \ell(\zh,\zp)=q $,
\item $\ell(\zb,\zq) = \ell(\zh,\zq) =p$,
\item $\ell(\zb,\zh)=pq$.
\end{enumerate}
\end{corollary}
\begin{proof}
Conclusion (i) follows from the fact that the Seifert fibration of $S^3$ in Proposition \ref{prop:SeifertInvts} respects the tori where $|z_1|$ and $|z_2|$ are constant, thus respects the splitting of $S^3$ into two solid tori whose cores are $\zp$ and $\zq$, which link once.

Conclusions (ii, iii)  follow from the fact that $\fp(\zp)$ and $\fp(\zq)$ are periodic points of the monodromy of the open book $(T(p,q),\pi)$ of periods $p$ and $q$, respectively. (The periods can be computed from Remark \ref{cor:CHdetails}(i, ii)). Their linking numbers with the binding equal their intersection numbers with the page, which are precisely these periods. For the first equalities in (ii) and (iii), note that $\mathpzc{h}$ is isotopic to $\mathpzc{b}$ in the complement of $\mathpzc{p}$ and $\mathpzc{q}$. 

Finally, conclusion (iv) follows from the fact that the linking number of disjoint $(p,q)$-torus knots on the same torus is $pq$.

\end{proof}

\subsection{Conley-Zehnder indicies}\label{ss:CZ}

We now compute the Conley-Zehnder indicies of the Reeb orbits with respect to various trivializations.  For the orbits $\mathpzc{b}$ and $\mathpzc{h}$ we utilize the constant and orbibundle trivializations, while for $\mathpzc{q}^q$ and $\mathpzc{p}^p$ we utilize the orbibundle trivialization, which also provides the formulae for other iterates of $\mathpzc{q}$ and $\mathpzc{p}$.  We make use of the change of trivialization formulae established in  \cite[\S 3]{kech} to switch to the page trivialization, which is used to determine the rotation angle relevant to our computation of knot filtered ECH.  Additional details for a different orbifold Morse function appeared in  \cite[\S 4]{kech}. 

The linearized Reeb flow along a Reeb orbit $\gamma: \R / T\Z \to Y$, as defined with respect to a choice of a trivialization $\tau \in \mathcal{T}(\gamma)$,  defines a symplectic map $P_{\gamma(t)}:\xi_{\gamma(0)} \to \xi_{\gamma(t)}$.   If $\gamma$ is nondegenerate then the path of symplectic matrices $\{ P_{\gamma(t)} \ | \ 0 \leq t \leq T \}
$ has a well-defined \emph{Conley-Zehnder} index $CZ_\tau(\gamma) := CZ(P_{\gamma(t)}) \in \Z.$

In three dimensions, the Conley-Zehnder index has the following properties, depending on the eigenvalues of the linearized return map.
\begin{itemize}
\itemsep-.25em
\item If the eigenvalues of the linearized return map are real, then $\gamma$ is said to be \emph{hyperbolic}.  In this case, there is an integer $n\in \Z$ such that the linearized Reeb flow along $\gamma$ rotates the eigenspaces of the linearized return map by angle $n\pi$ with respect to $\tau$, thus
\[
CZ_\tau(\gamma^k) = kn.
\]
Moreover, the integer $n$ is always even when the eigenvalues are positive while $n$ is always odd when the eigenvalues are negative.  We say that $\gamma$ is positive hyperbolic in the former case and negative hyperbolic in the latter.  The integer $n$ is the \emph{monodromy angle} of $\gamma$.
\item If the eigenvalues of the linearized return map lie on the unit circle, then $\gamma$ is said to be \emph{elliptic} and $\tau$ is homotopic to a trivialization wherein the linearized Reeb flow rotates by angle $2\pi\theta_t$ for each $t \in [0,T]$, where $\theta:[0,T] \to \R$ is continuous and $\theta_0=0$. The nondegeneracy assumption forces $\theta_T$ to be irrational and
\begin{equation}\label{CZ:elliptic}
CZ_\tau(\gamma^k) = 2 \lfloor k\theta_T \rfloor + 1.
\end{equation}
We call $\theta_T$ the \emph{monodromy angle} of $\gamma$.   We will use the terminology `rotation number' to designate the $\theta_T$ obtained from the homotopy class of trivializations satisfying the ``push off linking zero" property, which is equivalent to the ``page trivialization,"  and used to compute the knot filtration.
\end{itemize}

The Conley-Zehnder index depends only on the Reeb orbit $\gamma$ and the homotopy class of $\tau \in \mathcal{T}(\gamma)$.  Given another trivialization $\tau' \in \mathcal{T}(\gamma)$ we have the following change of trivialization formula
\begin{equation}\label{CZ:changetriv}
CZ_\tau(\gamma^k) - CZ_{\tau'}(\gamma^k) = 2k(\tau' - \tau);
 \end{equation}
see the discussion around \eqref{eqn:trivchange}, for a more precise description of the difference between two trivializations in the right hand side.

\begin{remark}\label{rem:linkingtriv}
In \S \ref{s:spectral}, we compute the knot filtration using the page trivialization, defined in \S \ref{ss:trivs}.  Since $\zb$ is the binding of the $T(p,q)$ open book, a pushoff of $\gamma$ into the page has linking number zero with $\gamma$.  Since $H_1(Y)=0$,  the page trivialization is well-defined as the trivialization in which this pushoff is constant. Let $\op{rot}(\gamma):=\theta_T \in \R$ denote the monodromy angle of $\gamma$ in the page trivialization; we refer to it as the \emph{rotation number} of $\gamma$.
\end{remark}

\begin{remark}\label{rem:c1sl}
Typically when computing the Conley-Zehnder index of a nondegenerate Reeb orbit, one uses a different homotopy class of trivialization $\tau$, which extends to a trivialization of $\xi$ over a surface bounded by $\gamma$, rather than one yielding a zero linking number with respect to the pushoff of the Reeb orbit.  These two trivializations differ by the self-linking number of the transverse knot $\gamma$.   The \emph{self-linking number} of any transverse knot (oriented and positively transverse to $\xi$) is defined as follows.

Let $\tau$ denote the homotopy class of a symplectic trivialization of $\xi|_\gamma$ for which a pushoff of $\gamma$ has linking number zero with $\gamma$.  If $\Sigma$ is a Seifert surface for $\gamma$ then 
\[
\op{sl}(\gamma):=-c_1(\xi|_\Sigma, \tau).
\]
As computed in Lemma~\ref{lem:relfirstCherncalcp}~(iii) (using the notation $c_\Sigma$ for $c_1(\xi|_\Sigma,\tau)$ explained in Remarks \ref{rmk:not} and \ref{rmk:trivs}), we have that $pq-p-q$ is the self linking number.  With respect to the push off linking zero trivialization, we get that the rotation angle of the binding is $pq$, whereas using the trivialization that extends over a disk, we show below that the monodromy angle is $p+q$.
\end{remark}

We first recall the following formula for the Conley-Zehnder indices of iterates of Reeb orbits associated to $\lambda_{p,q,\varepsilon}$ which project to critical points $x$ of ${\mathpzc{H}_{p,q}}$.  Here let $\gamma_p^k$ denote the $k$-fold iterate of an orbit which projects to $p \in \mbox{Crit}({\mathpzc{H}_{p,q}})$.  (Additional discussion of technicalities and proofs can be found in \cite[\S 4]{jo2}, \cite{vknotes}.)

\begin{lemma}\label{lempre}
Fix $L>0$ and let ${\mathpzc{H}_{p,q}}$ be a Morse-Smale function on $\CP^1_{2,q}$ which is $C^2$ close to 1.  Then there exists $\varepsilon >0$ such that all periodic orbits $\gamma$ of $R_{p,q,\varepsilon}$ with action $\mathcal{A}(\gamma) <L$ are nondegenerate and project to critical points of ${\mathpzc{H}_{p,q}}$.  The Conley-Zehnder index of such a Reeb orbit over $x \in \mbox{\em Crit}({\mathpzc{H}_{p,q}})$ is given by
\[
\begin{array}{lcl}
CZ_\tau(\gamma_x^k) &=& {RS}_\tau(\gamma^k) -1 + \mbox{\em index}_x{\mathpzc{H}_{p,q}},\\
\end{array}
\]
where $RS$ stands for the Robbin-Salamon index, which can be associated to a degenerate Reeb orbit \cite{RS}.
\end{lemma}

\subsubsection{Constant trivialization}\label{ss:consttau}
The constant trivialization $\tau_0$, as for example considered in \cite[\S 3]{preech} and briefly reviewed below, yields the following expression of the Conley-Zehnder index of the fibers which project to nonsingular points below.  Let $x\in \CP^1_{p,q} $ be any point with with trivial isotropy.  Then for any point $y \in \fp^{-1}(x)$, a fixed trivialization of $T_x\CP^1_{p,q}$ yields a trivialization of $\xi_y$ because $\xi_y \cong T_x\CP^1_{p,q}$. This trivialization is invariant under the linearized Reeb flow and is regarded as a  \emph{constant trivialization} over the orbit $\gamma_x$ because the linearized Reeb flow, with respect to this trivialization, is the identity map.

Using this constant trivialization, we have the following results; see also \cite[Lem. 3.3]{vknotes}, \cite[Lem. 4.8]{jo2}.

\begin{lemma}\label{consttrivlem}
Let $x\in \CP^1_{p,q} $ be any point with with trivial isotropy, and let $\gamma_x = \fp^{-1}(x) $ be the $S^1$ fiber realizing a Reeb orbit of $\lambda_{p,q}$ which projects to $x$.   Then for the constant trivialization $\tau_0$ we obtain $RS_{0}(\gamma_x)=0$ and $RS_{0}(\gamma_x^k) =0$, where $RS$ denotes the Robbin-Salamon index. \end{lemma}

\begin{corollary}\label{cor:CZ0bh}
Fix $L>0$ and $\mathpzc{H}_{p,q}$ as in Proposition \ref{prop:morsep}.  Then there exists an $\varepsilon >0$ such that all $k$-fold iterates of $ \mathpzc{b}$ with action $\A(\mathpzc{b}^k) <L$ are nondegenerate and $CZ_{0}( \mathpzc{b}^k) = 1$ and $CZ_{0}( \mathpzc{h}^k) = 0$.
\end{corollary}

\subsubsection{Orbibundle trivialization}\label{ss:orbitau}

To understand Conley-Zehnder indices of Reeb orbits which project to critical points we make use of a global orbibundle trivialization relating the Robbin-Salamon indices of the fiber to the orbifold Chern class of the base $\CP^1_{p,q}$.   When combined with Lemma \ref{lempre} this yields the following results by way of \cite[Thm.~3.1]{hong-cz}, and a direct adaptation of \cite[\S 4.2.2]{kech} using Example \ref{symporb} and \eqref{eq:orbc}.

\begin{lemma}\label{lem:orbtrivCZp}
Fix $L>0$ and $\mathpzc{H}_{p,q}$ a Morse-Smale function as in Proposition \ref{prop:morsep} on $\CP^1_{p,q}$ which is $C^2$ close to 1.  Then there exists $\varepsilon>0$ such that all periodic orbits $\gamma$ of $R_{p,q,\varepsilon}$ with action $\mathcal{A}(\gamma)<L$ are nondegenerate and project to critical points of $\mathpzc{H}_{p,q}$.  The Conley-Zehnder index of such a Reeb orbit over $x \in \op{Crit}(\mathpzc{H}_{p,q})$ is given by
\[
CZ_{orb}(\gamma_x^{k|\Gamma_x|}) = 2(p+q)k + \mbox{\em index}_x \mathpzc{H}_{p,q} -1 
\]
In particular, we have:
\[
\begin{array}{lcl}
CZ_{orb}(\mathpzc{b}^k) & = & 2(p+q)k + 1, \\
CZ_{orb}(\mathpzc{h}^k) & = & 2(p+q)k , \\
CZ_{orb}({\mathpzc{p}}^{pk}) &= & 2(p+q)k -1, \\
CZ_{orb}(\mathpzc{q}^{qk}) & = & 2(p+q)k -1. \\
\end{array}
\]
Thus,
\begin{itemize}
\itemsep-.35em
\item $\mathpzc{b}$ is elliptic of monodromy angle ${p+q} + \delta_{\mathpzc{b},L}$, where $0<\delta_{\mathpzc{b},L} \ll 1 $ is irrational;
\item $\mathpzc{h}$ is positive hyperbolic with  monodromy angle  $2(p+q)$; 
\item $\mathpzc{p}$ is elliptic of monodromy angle $\frac{p+q}{p} - \delta_{\mathpzc{p},L}$, where $0<\delta_{\mathpzc{p},L} \ll 1 $ is irrational;
\item $\mathpzc{q}$ is elliptic of  monodromy angle $\frac{p+q}{q} - \delta_{\mathpzc{q},L}$, where $0<\delta_{\mathpzc{q},L} \ll 1 $ is irrational.
\end{itemize}
\end{lemma}

In Table \ref{tab:CZ} we compute the Conley-Zehnder indices in the orbibundle trivialization of the first several Reeb orbits for $p=3, q=4$. Note that in Lemma \ref{lem:orbtrivCZp} we obtain the monodromy angles by analyzing nonlinearities in how the Conley-Zehnder indices depend on $k$ (up to some large $k$ depending on $\varepsilon$). For example, when $k$ increases by one, the Conley-Zehnder index $CZ_{orb}(\zp^{pk})$ increases by $p+q$, meaning
\[
2\lfloor (pk+p)\theta\rfloor+1=2\lfloor pk\theta\rfloor+1+2(p+q) \Rightarrow \lfloor pk\theta+p\theta\rfloor-\lfloor pk\theta\rfloor=p+q,
\]
letting $\theta$ denote the monodromy angle of $\zp$ in the orbibundle trivialization. Because $\lfloor x+y\rfloor\geq\lfloor x\rfloor+\lfloor y\rfloor$, we now have
\[
\lfloor p\theta\rfloor\geq p+q \Rightarrow p\theta\geq p+q-1\Rightarrow\theta\geq\frac{p+q-1}{p}.
\]
(Note that we already know $\theta<(p+q)/p$ directly from the value of $CZ_{orb}(\zp)$.) Comparing $CZ_{orb}(\zp^{pk+n})$ to $CZ_{orb}(\zp^{pk})$ for $n$ large, as well as the making the analogous comparisons for $\zb, \zh$, and $\zq$, provides us with the full slate of monodromy angles, whence we may compute the entries in Table \ref{tab:CZ} directly from (\ref{CZ:elliptic}).

\begin{table}[h!]
\centering
\begin{tabular}{ || c | c | c | c | c | c | c | c  | c  | c  | c  | c  | c | c | c  | c | c | c | c || } 
\hline orbit  & $\mathpzc{q}$ & $\mathpzc{p}$  & $\mathpzc{q}^2$ & $\mathpzc{p}^2$ & $\mathpzc{q}^3$ & $\mathpzc{p}^3$ & $\mathpzc{q}^4$ & $\mathpzc{h}$ & $\mathpzc{b}$ & $\mathpzc{q}^5$ & $\mathpzc{p}^4$ & $\mathpzc{q}^6$ & $\mathpzc{p}^5$ & $\mathpzc{q}^7$ & $\mathpzc{p}^6$ & $\mathpzc{h}^2$ & $\mathpzc{q}^8$ & $\mathpzc{b}^2$ \\ 
\hline  $CZ_{orb}$  &  3 & 5  & 7 & 9 & 11 & 13 & 13 & 14 & 15 & 17 & 19 & 21 & 23  & 25 & 27 & 28 & 29 & 29\\
\hline  
\end{tabular}
\caption{Conley-Zehnder indices for the $T(3,4)$ open book decomposition}
\label{tab:CZ}
\end{table}

\subsubsection{Page trivialization aka push-off linking zero trivialization}\label{ss:pagetau}
It remains to describe the Conley-Zehnder indices and  monodromy angles of the bindings with respect to the the page trivialization $\tau_\Sigma$, described in \S \ref{ss:trivs}.  Since the page trivialization is the push-off linking zero trivialization, the below monodromy angles are the rotation numbers used to compute the knot filtration on embedded contact homology.  
\begin{lemma}\label{lem:pagetrivCZp} For the $T(p,q)$ binding $\mathpzc{b}$, we have $CZ_\Sigma(\mathpzc{b}^B)=2pqB+1$.  Thus with respect to the page trivialization,  $\zb$ is elliptic with $\op{rot}(\mathpzc{b})=pq+\delta_{\mathpzc{b},L}$, where $\delta_{\mathpzc{b},L}$ is an irrational number such that $0 < \delta_{\mathpzc{b},L} \ll 1.$ (We are implicitly assuming  $\mathpzc{b}^B$ is nondegenerate with $\A(\mathpzc{b}^B)<L$.)
\end{lemma}

\begin{proof} Using (\ref{CZ:changetriv}) and Lemmas \ref{lem:orbtrivCZp} and \ref{lem:taudiffp}  we have
\[
CZ_\Sigma(\zb^B)=2B(\tau_{orb}(\zb)-\tau_\Sigma(\zb))+CZ_{orb}(b^B)=2B(pq-p-q)+2(p+q)B+1 = 2pqB+1.
\]
\end{proof}

\section{Components of the ECH index}\label{s:generalities}

In this section we provide the remaining calculations necessary to prove our embedded contact homology index result, Theorem \ref{thm:pqI}. Our computations of the relative Chern and relative self-intersection numbers differ only slightly from those in \cite[\S3-4]{kech}, therefore we provide only the detail necessary to understand the slight extensions of the proofs required in the present context.  In particular, the main difference is the presence of the Reeb orbit $\zh$.

\subsection{Definition and key properties of the ECH index}\label{ss:ECHI}

We define the ECH index and state its key properties. 
The ECH index has three components: 
\begin{itemize}
\itemsep-.25em
\item the relative first Chern number $c_\tau$, which is sensitive to $\xi$;
\item the relative intersection pairing $Q_\tau$, which is sensitive to the topology;
\item the Conley-Zehnder terms, which are sensitive to the Reeb dynamics.
\end{itemize}

Given Reeb currents $\alpha$ and $\beta$, we use $H_2(Y,\alpha,\beta)$ to denote the 2-chains $Z$ with $\partial Z=\alpha-\beta$, modulo boundaries of 3-chains. When $H_2(Y)=0$, the set $H_2(Y,\alpha,\beta)$ has only one element, and if $\beta=\emptyset$ let $Z_\alpha$ denote this lone element. If $Z \in H_2(Y,\alpha,\beta)$ and $\tau$ symplectically trivializes $(\xi,d\lambda)$ over the Reeb orbits $\{\alpha_i\}$ and $\{\beta_j\}$ the ECH index is:

\begin{definition}[ECH index]\label{defn:ECHI} 
Let $\alpha=\{(\alpha_i,m_i)\}$ and $\beta=\{(\beta_j,n_j)\}$ be Reeb currents with $\sum_i m_i [\alpha_i]=\sum_j n_j [\beta_j]= \Gamma\in H_1(Y).$   Given $Z \in H_2(Y,\alpha,\beta)$, the \emph{ECH index} is
\[
I(\alpha,\beta,Z) =  c_\tau(Z) + Q_\tau(Z) +  CZ^I_\tau(\alpha) -CZ^I_\tau(\beta),
\]
where $CZ^I(\gamma): =  \sum_i \sum_{k=1}^{m_i}CZ_\tau(\gamma_i^k)$. 
In some cases we simplify the notation: $I(Z)$ when $\alpha$ and $\beta$ are clear from context, and $I(\alpha)$ when $\beta=\emptyset$ and $Z$ is clear from context. The \emph{relative first Chern number} $c_\tau$ is defined in \S\ref{ss:chern}, the \emph{relative intersection pairing} $Q_\tau$ is defined in \S\ref{ss:Q}, and the \emph{Conley-Zehnder index} $CZ_\tau$ was defined in \S \ref{ss:CZ}.
\end{definition}

The ECH index has the following general properties, cf. \cite[\S 3.3]{Hindex}. 
 
\begin{theorem}[{\cite[Prop.~1.6]{Hindex}}]\label{thm:Iproperties}
The ECH index has the following basic properties: \hfill
\begin{enumerate}[{\em (i)}]
\itemsep-.25em
\item {\em(Well Defined)} The ECH index $I(Z)$ is independent of the choice of trivialization.
\item\label{property:indexamb} {\em (Index Ambiguity Formula) }If $Z' \in H_2(Y,\alpha,\beta)$ is another relative homology class, and $Z-Z'$ is as defined using the affine structure of $H_2(Y,\alpha,\beta)$ over $H_2(Y)$, then
\[
I(Z)-I(Z')= \langle Z-Z', c_1(\xi) + 2 \mbox{\em PD}(\Gamma) \rangle.
\]
\item {\em (Additivity) } If $\delta$ is another orbit set in the homology class $\Gamma$, and if $W \in H_2(Y,\beta,\delta)$, then $Z+W \in  H_2(Y,\alpha,\delta)$ is defined by gluing representatives along $\beta$ and 
\[
I(Z+W)=I(Z)+I(W).
\]
\item {\em (Index Parity) }If $\alpha$ and $\beta$ are generators of the ECH chain complex (in particular all hyperbolic orbits have multiplicity 1), then $(-1)^{I(Z)} = \varepsilon(\alpha)\varepsilon(\beta),$ where $\varepsilon(\alpha)$ denotes $-1$ to the number of positive hyperbolic orbits in $\alpha$.
\end{enumerate}
\end{theorem}

\subsection{Surfaces, trivializations, and changes of trivialization}\label{ss:trivs} 

We now describe the trivializations we will need to compute the contributions to the ECH index from the isotropy orbits $\zp$ and $\zq$ as well as the positive hyperbolic orbit $\zh$, and compute their differences.

For the Reeb orbit $\zb$, we may use three trivializations: the constant trivialization $\tau_0$, the page trivialization $\tau_\Sigma$, and the orbibundle trivialization $\tau_{orb}$. Both $\tau_0$ and $\tau_{orb}$ may be used for $\zh$, because it, like $\zb$, is a regular fiber. But $\tau_0$ cannot be used for $\zp$ and $\zq$ as they project to orbifold points. We geometrically ``extend" $\tau_0$ to construct trivializations over $\zh$, $\zp^p$, and $\zq^q$. {See Remarks \ref{rmk:0eh}-\ref{rmk:Tpq0Sig}.}   

\begin{remark}\label{rmk:not} For simplicity, in our computations we suppress $\tau$ in the notation. For example, we use the notation
\[
c_0:=c_{\tau_0}, \quad Q_0:=Q_{\tau_0}, \quad CZ_0:=CZ_{\tau_0}.
\]
\end{remark}

Our trivializations are defined using surface representatives of the elements of the individual $H_2(S^3,\alpha,\beta)$ sets. We now define these surfaces, using $[S]$ to denote the equivalence class in $H_2(Y,\alpha,\beta)$ of a surface $S$ in $Y$ with $\partial S=\alpha-\beta$.

\begin{definition}[{Surfaces for $T(p,q)$, $p\neq 2$}]\label{def:sfcsp}  
In this definition, the notation $\cup_{\mathpzc{b}}$ means we are taking the union along the common boundary component ${\mathpzc{b}}$:

\begin{itemize}
\itemsep-.25em
\item The surface $\Sigma$, which has one boundary component and genus $(p-1)(q-1)/2$, is the page of the open book decomposition discussed in \S\ref{ss:OBDs}.  Thus $\partial\Sigma=\mathpzc{b}={T(p,q)}$ and $[\Sigma]=Z_{\mathpzc{b}}$.

\item  $S_{\mathpzc{h}}$ is the preimage under $\fp$ of a short ray connecting $\fp(\mathpzc{h})$ to a regular fiber $\mathpzc{b}$. Thus $\partial S_{\mathpzc{h}}=\mathpzc{h}-\mathpzc{b}$ and $[S_{\mathpzc{h}}\cup_{\mathpzc{b}}\Sigma]
=Z_{\mathpzc{h}}$.

\item The surface $S_{\mathpzc{q}}$ is the preimage under $\fp$ of a short ray connecting $\fp(\mathpzc{q})$ to a regular fiber $\mathpzc{b}$. Thus $\partial S_{\mathpzc{q}}=\mathpzc{q}^q-\mathpzc{b}$ and $[S_{\mathpzc{q}}\cup_{\mathpzc{b}}\Sigma]
=Z_{\mathpzc{q}^q}
$. 

\item The surface $S_{\mathpzc{p}}$ is the preimage under $\fp$ of a short ray connecting $\fp({\mathpzc{p}})$ to a regular fiber ${\mathpzc{b}}$. Thus $\partial S_{\mathpzc{p}}={{\mathpzc{p}}^p}-{\mathpzc{b}}$ and $[S_{\mathpzc{p}}\cup_{\mathpzc{b}}\Sigma]
=Z_{\mathpzc{p}^p}
$.
\end{itemize}

\end{definition}

We require another piece of notation. Let $\gamma\in\partial S$. As in \cite{weiler}, a trivialization $\tau$ over an orbit $\gamma$ ``has linking number zero with respect to $S$" or ``is the $S$ trivialization" if the pushoff of $\gamma$ into $S$ is constant in $\tau$, or equivalently, has intersection number zero with $S$. See also Remarks \ref{rem:linkingtriv} and \ref{rem:c1sl}. We use the notation $\tau_S$.

\begin{remark}[Trivializations for $T(p,q)$, $p\neq 2$]\label{rmk:trivs}
The trivializations we will use are:
\begin{itemize}
\itemsep-.25em
\item The page trivialization $\tau_\Sigma$ over $\mathpzc{b}$,

\item The \emph{constant trivialization} $\tau_0$ over $\mathpzc{b}$ and $\mathpzc{h}$;
it has three related surface trivializations (see Remark \ref{rmk:0eh}), which are
\begin{itemize}
\itemsep-.25em
\item[*] $\tau_{\mathpzc{q}}:=\tau_{S_{\mathpzc{q}}}$ over $\mathpzc{b}$ and ${\mathpzc{q}}^q$,
\item[*] $\tau_{\mathpzc{p}}:=\tau_{S_{\mathpzc{p}}}$ over $\mathpzc{b}$ and ${{\mathpzc{p}}^p}$,
\item[*] $\tau_{\mathpzc{h}}:=\tau_{S_{\mathpzc{h}}}$ over $\mathpzc{b}$ and ${\mathpzc{h}}$.
\end{itemize}

\item The \emph{orbibundle trivialization} $\tau_{orb}$ over all four orbits $\zb, \zh, \zp$, and $\zq$; this is explained in \S\ref{ss:orbitau}.
\end{itemize}
\end{remark}

See \S\ref{ss:consttau} for the computation of the Conley-Zehnder index in $\tau_0$. We explain the relationship between $\tau_0$ and $\tau_\zq, \tau_\zp$, and $\tau_\zh$. 

\begin{remark}\label{rmk:0eh} The defining condition of the constant trivialization, i.e. that the linearized Reeb flow be the identity, also works for fiber orbits whose neighbors are all regular fibers. These can be exceptional, but we must take a cover of the orbit by a multiple of the underlying orbifold point's isotropy. We may also define these as the surface trivializations for a surface that is a union of nearby fibers, such as $S_\zp$ or $S_\zq$. One boundary of this surface will be the minimal cover discussed above. E.g., the boundary component of $S_\zp$ is a $p$-fold cover of $\zp$; for $S_\zq$, it is a $q$-fold cover. See also the mesh surface in \cite[Fig.~3]{weiler}. 
\end{remark}

In order to provide computational evidence for Remark \ref{rmk:0eh}, we introduce our convention for the difference between two trivializations. It will also be used throughout the rest of \S\ref{s:generalities} and was previously used in \S\ref{ss:pagetau}. 
Given a nondegenerate Reeb orbit $\gamma:\R/T\Z\to Y$,  denote the set of homotopy classes of symplectic trivializations of the 2-plane bundle $\gamma^*\xi$ over $S^1=\R/T\Z$ by $\mathcal{T}(\gamma)$.   After fixing trivializations $\tau_i^+ \in \mathcal{T}(\alpha_i)$ for each $i$ and $\tau_j^- \in \mathcal{T}(\beta_j)$, let $\tau \in \mathcal{T}(\alpha,\beta)$ denote this set of trivialization choices.  The trivialization $\tau$ determines a trivialization of $\xi|_C$ over the ends of $C$ up to homotopy.  As spelled out in \cite[Rem.~2.4]{preech}, we use the sign convention that if $\tau_1, \ \tau_2: \gamma^*\xi \to S^1 \times \R^2$ are two trivializations then
\begin{equation}\label{eqn:trivchange}
\tau_1 - \tau_2 = \mbox{deg} \left(\tau_2 \circ \tau_1^{-1} : S^1 \to \mbox{Sp}(2,\R) \cong S^1\right).
\end{equation}

Formulas allowing us to understand how the components $c_\tau, Q_\tau$, and $CZ_\tau$ of the ECH index change according to the trivialization difference appear in (\ref{CZ:changetriv}) and \cite[Lem.~3.4~(iv), Prop.~3.13]{kech}, which we use freely throughout the remainder of this section.

\begin{remark}[$T(p,q)$ constant and surface trivializations]\label{rmk:Tpq0Sig}
Continuing the discussion in Remark \ref{rmk:0eh}, we can think of the trivializations $\tau_\mathpzc{p}$ and $\tau_\mathpzc{q}$ as extensions of the constant trivialization over $\mathpzc{p}^p$ and $\mathpzc{q}^q$. Combining \cite[Prop.~3.13]{kech}, Lemma \ref{lem:orbtrivCZp}, and Lemma \ref{lem:qpdiff} shows that
\begin{align*}
CZ_{\mathpzc{q}}({\mathpzc{q}}^{qk})&=CZ_{orb}({\mathpzc{q}}^{qk})+2k(\tau_{orb}({\mathpzc{q}}^q)-\tau_{\mathpzc{q}}({\mathpzc{q}}^q))=2({p+q})k-1+2k({-p-q})=-1; \\
CZ_{\mathpzc{p}}({\mathpzc{p}}^{pk})&=CZ_{orb}({\mathpzc{p}}^{pk})+2k(\tau_{orb}({\mathpzc{p}}^p)-\tau_{\mathpzc{p}}({\mathpzc{p}}^p))=2({p+q})k-1+2k({-p-q})=-1.
\end{align*}
These are the values taken if they were regular fibers, analogous to \cite[Lem.~3.9]{preech}.
\end{remark}

We now collect the various change in trivialization formulas. We will use $\tau(\gamma)$ to indicate we are restricting the trivialization to an orbit $\gamma$. The proofs are almost identical to those in \cite[\S3.4]{kech}, so we only provide the proofs of results which have no analogue there or note where we must make a change. Some proofs require the relative first Chern number computations carried out in \S\ref{ss:chern} as a black box.

For the binding, by repeating the proof of \cite[Lem.~3.14]{kech}, we obtain the following.
\begin{lemma}\label{lem:taudiffp} For trivializations defined along the binding orbit $\mathpzc{b}$ realizing $T(p,q)$, we have:
\begin{enumerate}[{\em (i)}]
\itemsep-.25em
\item $\tau_\zp(\zb)=\tau_\zq(\zb)=\tau_\zh(\zb)=\tau_0(\zb)$,
\item $\tau_0(\zb)-\tau_\Sigma(\zb)={pq}$,
\item $\tau_0(\zb)-\tau_{orb}(\zb)={p+q}$, 
\item $\tau_{orb}(\zb)-\tau_\Sigma(\zb)={pq-p-q}$ 
\end{enumerate}
\end{lemma}

For the exceptional fibers and $\zh$, we have:
\begin{lemma}\label{lem:qpdiff} 
The change of trivialization formulae  along $\zq^q,\zp^p$, and $\zh$ are:
\begin{enumerate}[{\em (i)}]
\itemsep-.25em
\item $\tau_{\mathpzc{q}}({\mathpzc{q}}^q)-\tau_{orb}({\mathpzc{q}}^q)=p+q$.
\item $\tau_{\mathpzc{p}}({\mathpzc{p}}^p)-\tau_{orb}({\mathpzc{p}}^p)=p+q$.
\item $\tau_\zh(\zh)-\tau_{orb}(\zh)=p+q$.
\end{enumerate}
\end{lemma}
\begin{proof} The proof is identical to that of \cite[Lem.~3.15]{kech}, except we use Lemma \ref{lem:relfirstCherncalcp} in place of \cite[Lem.~4.1]{kech}. For (iii), we could also use a different argument. We know that $\tau_\zh(\zh)=\tau_0(\zh)$ by the argument used for \cite[Lem.~3.14~(i)]{kech}, and by (\ref{CZ:changetriv}), Corollary \ref{cor:CZ0bh}, and Lemma \ref{lem:orbtrivCZp},
\[
2(\tau_0(\zh)-\tau_{orb}(\zh))=CZ_{orb}(\zh)-CZ_0(\zh)=2(p+q)-0.
\]
\end{proof}

\subsection{Computing the relative first Chern numbers}\label{ss:chern}

Given $\xi$ and a $J$-holomorphic curve $C$, the restriction $\xi|_C$ forms a complex line bundle on a complex surface. It has a \emph{relative first Chern number} with respect to a trivialization $\tau \in \mathcal{T}(\alpha,\beta)$, denoted
\[
c_\tau(C)  = c_1(\xi|_C,\tau),
\]
and defined as follows: it is the algebraic count of zeros of a generic section $\psi$ of $\xi|_{[\pi_YC]}$, where $\pi_Y:\R\times Y\to Y$ denotes projection onto $Y$, which on each end is nonvanishing and constant under $\tau$.

The relative first Chern number depends only on $C$ up to boundaries of 3-chains, and can be defined for $Z\in H_2(Y,\alpha,\beta)$ as follows. Choose a representative of $Z$ by a smooth map $f: S \to Y$, where $S$ is a compact oriented surface with boundary.  Also choose a section $\psi$ of $f^*\xi$ over $S$ for which $\psi$ is transverse to the zero section of $f^*\xi$ (a bundle on $S$) and which is nonvanishing over each boundary component of $S$. Moreover, require $\psi$ to have winding number zero along $\partial S$ in the trivialization $\tau$.  Define 
\[
c_\tau(Z) : = \# \psi^{-1}(0),
\]
using `\#' to denote the signed count.

A virtual repeat of the proof of \cite[Lem.~4.1]{kech} yields the following formulae for the relative first Chern number $c_\tau(Z)$ where $Z=[S]$ for the surfaces $S$ of Definition \ref{def:sfcsp}.

\begin{lemma}\label{lem:relfirstCherncalcp} In the $T(p,q)$ setting, we have: \begin{enumerate}[{\em (i)}]
\itemsep-.25em
\item $c_{orb}([\Sigma])=0$,
\item $c_0([\Sigma])=p+q$,
\item\label{ctauS} $c_\Sigma([\Sigma])= {p+q-pq}$,
\item $c_{\mathpzc{q}}(Z_{{\mathpzc{q}}^q})=c_{\mathpzc{p}}(Z_{{\mathpzc{p}}^p})=c_\zh(Z_\zh)=p+q$,
\item $c_{orb}(Z_{\mathpzc{q}})=c_{orb}(Z_{\mathpzc{p}})=c_{orb}(Z_\zh)=0$.
\end{enumerate}
\end{lemma}

\subsection{Computing the relative intersection pairings}\label{ss:Q}

The final component of the ECH index is the relative intersection pairing $Q_\tau$. To define this, we first need to specify the type of surfaces $S\subset[-1,1]\times Y$ we are using to represent $Z\in H_2(Y,\alpha,\beta)$. 

\begin{definition}[{\cite[Def.~2.11]{Hrevisit}}] An \emph{admissible representative} of $Z\in H_2(Y,\alpha,\beta)$ is a smooth map $f: S \to [-1,1] \times Y$, where we require $S$ be an \emph{admissible surface}, i.e. it is an oriented compact surface satisfying:
\begin{itemize}
\itemsep-.25em
\item $f|_{\partial S}$ consists of positively oriented  covers of $\{ 1\} \times \alpha_i$ with total multiplicity $m_i$ along with negatively oriented covers of $\{ -1\} \times \beta_j$ with total multiplicity $n_j$;
\item $f$ is transverse to $\{-1,1\} \times Y$;
\item $\pi_Y: [-1,1] \times Y \to Y$ satisfies $[\pi_Y(f(S))]=Z$;
\item $f|_{\op{int}(S)}$ is an embedding.
\end{itemize}
\end{definition}

Admissible surfaces $S$ determine braids $S\cap(\{1-\varepsilon\}\times Y)$ for $\varepsilon>0$ sufficiently small (and the analogous intersection at $-1+\varepsilon$), which wind around the $\alpha_i$ and $\beta_j$, and which are denoted by $\zeta_i^+$ and $\zeta_j^-$. These braids are well defined up to isotopy in a small enough tubular neighborhood of the orbits. The braid $\zeta_i^+$ has $m_i$ strands. The data of these braids is the way in which $\tau$ affects $Q_\tau$.

Let $S'$ be an admissible representative of $Z' \in H_2(Y,\alpha',\beta')$ whose interior does not intersect the interior of $S$ near their boundaries. Let $\zeta_i^{+ '}$ and $\zeta_j^{- '}$ denote the braids of $S'$. Define the \emph{linking number of $S$ and $S'$} as
\[
\ell_\tau(S,S') := \sum_i \ell_\tau(\zeta_i^+,\zeta_i^{+ '}) - \sum_j \ell_\tau(\zeta_j^-,\zeta_j^{- '}).
\]
  
Now let $S$ and $S'$ be admissible representatives of $Z \in H_2(Y,\alpha,\beta)$ and $Z'\in H_2(Y,\alpha',\beta')$. Assume their interiors $\dot{S}$ and $\dot{S}'$ intersect transversely and away from the boundary.  The \emph{relative intersection pairing} is
\begin{equation}\label{eqn:Qdef}
Q_\tau(Z,Z'):= \# \left( \dot{S} \cap \dot{S}'\right) - \ell_\tau(S,S').
\end{equation}
It is an integer and depends only on $\alpha, \beta, Z, Z'$ and $\tau$. To simplify notation, we write $Q_\tau(Z):=Q_\tau(Z,Z)$ when $Z'=Z$, and call it the \emph{relative self-intersection number} of $Z$ or of $\partial Z$ if $H_2(Y,\partial Z,\emptyset)$ is a single element.

When $S$ satisfies the conditions described in the next definition, we may simplify (\ref{eqn:Qdef}).

\begin{definition} Let $S$ be an admissible representative of $Z$ for which $\pi_Y\circ f$ is an embedding near $\partial S$ and the $m_i$ (respectively, $n_j$) nonvanishing intersections of these embedded collars of $\partial S$ with $\xi$ lie in distinct rays emanating from the origin and do not rotate (with respect to $\tau$) as one goes around $\alpha_i$ (respectively, $\beta_j$). Such an $S$ is a \emph{$\tau$-representative} of $Z$ and
\[
Q_\tau(Z,Z'):= \# \left( \dot{S} \cap \dot{S}'\right).
\]
\end{definition}

We now compute the necessary relative intersection pairings. The proofs are similar to those in \cite[\S4.3]{kech}, and we largely omit them, noting the analogous result from \cite{kech} in the statement. The formatting is similar for the ease of comparison.

\begin{lemma}[c.f. {\cite[Lem.~4.14]{kech}}] The following surfaces are $\tau$-representatives:
\begin{enumerate}[{\em (i)}]
\itemsep-.25em
\item The surface $\Sigma$ is a $\tau_\Sigma$-representative of $Z_b$.
\item The surface $S_\zp$ is a $\tau_\zp$-representative of the class in $H_2(S^3,\zp^p,b)$.
\item The surface $S_\zq$ is a $\tau_\zq$-representative of the class in $H_2(S^3,\zq^q,b)$.
\item The surface $S_\zh$ is a $\tau_\zh$-representative of the class in $H_2(S^3,\zh,b)$.
\end{enumerate}
\end{lemma}

\begin{lemma}[c.f. {\cite[Lem.~4.15]{kech}}]\label{lem:pzero} The following relative intersection pairings are zero.
\begin{enumerate}[{\em (i)}]
\itemsep-.25em
\item $Q_\zp([S_\zp])=Q_\zq([S_\zq])=0$,
\item $Q_\tau([S_\zp],[S_\zq])=0$, where $\tau(\zp)=\tau_\zp(\zp), \tau(\zq)=\tau_\zq(\zq)$, and $\tau(\zb)=\tau_0(\zb)$.
\end{enumerate}
\end{lemma}

The only difference between the computations in the proof of the following lemma and those in the proof of \cite[Lem.~4.16]{kech} is the use of the change-of-trivialization formulas in Lemma \ref{lem:taudiffp} rather than those in \cite[Lem.~3.14]{kech}.
\begin{lemma}[c.f. {\cite[Lem.~4.16]{kech}}]\label{lem:pQb} The following formulas hold.
\begin{enumerate}[{\em (i)}]
\itemsep-.25em
\item $Q_\Sigma(Z_b)=0$,
\item $Q_{orb}(Z_b)=pq-p-q$,
\item $Q_0(Z_b)=pq$.
\end{enumerate}
\end{lemma}

Similarly, the following lemma differs from its counterpart only in the use of Lemma \ref{lem:qpdiff} instead of \cite[Lem.~3.15]{kech}.
\begin{lemma}[c.f. {\cite[Lem.~4.17]{kech}}]\label{lem:QpQq} The following formulas hold.
\begin{enumerate}[{\em (i)}]
\itemsep-.25em
\item $Q_{orb}(Z_\mathpzc{p})=-1$,
\item $Q_{orb}(Z_\mathpzc{q})=-1$.
\end{enumerate}
\end{lemma}

To prove the following lemma, replace \cite[Cor.~2.28]{kech} with Corollary \ref{cor:linkingnumber} in the proof of \cite[Lem.~4.18]{kech}.
\begin{lemma}[c.f. {\cite[Lem.~4.18]{kech}}]\label{lem:pqcrossQ} The following formulas hold.
\begin{enumerate}[{\em (i)}]
\itemsep-.25em
\item $Q_{orb}(Z_\mathpzc{p},Z_\mathpzc{q})=1$,
\item $Q_{orb}(Z_\mathpzc{p},Z_\mathpzc{b})=Q_{orb}(Z_\mathpzc{p},Z_\mathpzc{h})=q$,
\item $Q_{orb}(Z_\mathpzc{q},Z_b)=Q_{orb}(Z_\mathpzc{q},Z_\mathpzc{h})=p$, 
\item $Q_{orb}(Z_\mathpzc{h},Z_\mathpzc{b})=pq$.
\end{enumerate}
\end{lemma}

Finally, we compute the relative self-intersection pairing of $Z_\zh$, which has no analogue in the case $p=2$.

\begin{lemma}\label{lem:Qh} We have $Q_{orb}(Z_\zh)=pq-p-q$.
\end{lemma}
\begin{proof} Recall as in Definition \ref{def:sfcsp} that $Z_\zh=S_\zh\cup_\zb\Sigma$. The relative intersection pairing is linear with respect to this concatenation of surfaces,\footnote{Note that this is not addition in $H_2(Y,\alpha,\beta)$ using the affine structure over $H_2(Y)$. We refer to it as ``concatenation addition" in \cite[Rmk.~3.7~(iii)]{kech} and it also appears in Theorem~\ref{thm:Iproperties}~(iii).} therefore
\[
Q_{orb}(Z_\zh)=Q_{orb}([S_\zh])+Q_{orb}(Z_\zb)=Q_{orb}([S_\zh])+pq-p-q,
\]
using Lemma \ref{lem:pQb}(ii) for the second equality. We then use the change-of-trivialization formula in \cite[Lem.~3.4~(iv)]{kech} to compute
\[
Q_{orb}([S_\zh])=Q_\zh([S_\zh])+\tau_{orb}(\zh)-\tau_\zh(\zh)-(\tau_{orb}(\zb)-\tau_\zh(\zb)).
\]
This is zero by Lemma \ref{lem:qpdiff}(iii) and Lemma \ref{lem:taudiffp}(i, iii).
\end{proof}

\section{The ECH chain complex}\label{s:ECHI}
 {We now turn to our computation of the ECH chain complex of $S^3$ with the contact form $\lambda_{p,q,\varepsilon}$ up to action level $L(\varepsilon)$.} Since the contact form $\lambda_{p,q}$ is degenerate, there is no chain complex ``$ECC_*(S^3,\lambda_{p,q},J)$," and therefore no filtration with respect to $T(p,q)$. Instead, in \S\ref{s:spectral} we will compute the knot filtration with respect to $T(p,q)$ on equivalence classes of homology classes in the complexes
\[
ECH^{L(\varepsilon)}_*(S^3,\lambda_{p,q,\varepsilon}).
\]

We give a summary our arguments in \S \ref{ss:pnot2} and review how we establish the behavior of the differential.  This requires going beyond the methods we developed in \cite[\S 4-6]{preech}.  We set up the needed intersection theory machinery in \S \ref{ss:currents}, which we use to rule out noncylindrical contributions to the differential in  \S\ref{sss:orbipreech}.  Combinatorial arguments are carried out in \S \ref{ss:combinatorics} and we complete the computation of the ECH chain complex in \S \ref{sss:finalcomp}.  Finally, in \S\ref{ss:toric}, we explain the comparison between our complexes and the combinatorial complex associated to a convex toric perturbation.

Before we continue, we review some notation and define the degree of a pair of homologous Reeb currents. We obtain $\lambda_{p,q,\varepsilon}$ by perturbing $\lambda_{p,q}$ using the the Morse function $\mathpzc{H}_{p,q}$ of Proposition \ref{prop:morsep} on the base orbifold $\CP^1_{p,q}$.  The Reeb vector field $R_{p,q,\varepsilon}$ admits two orbits realizing the singular fibers of the Seifert fibration, called ${\mathpzc{p}}$ and ${\mathpzc{q}}$, which project to critical points of index 0 which are of respective isotropy $\Z/p$ and $\Z/q$.  All regular fibers are positive $T(p,q)$ knots; the orbit $ \mathpzc{b}$ projects to an index two critical point and the  orbit $ \mathpzc{h}$ projects to an index one critical point.  As described in Lemma \ref{lem:orbitseh}, up to large action the only Reeb orbits of $R_{p,q,\varepsilon}$ are iterates of  ${\mathpzc{q}}$ (elliptic), ${\mathpzc{p}}$ (elliptic), $\mathpzc{h}$ (positive hyperbolic), $\mathpzc{b}$ (elliptic). We prove that
\[
\lim_{\varepsilon\to0}ECH^{L(\ve)}_*(S^3,\lambda_{p,q,\varepsilon})= ECH_*(S^3,\xi_{std})=\begin{cases}\Z/2&\text{ if }* \in 2\Z_{\geq 0}
\\0&\text{else},
\end{cases}
\]
and furthermore, we compute the chain complexes $ECH^{L(\ve)}_*(S^3,\lambda_{p,q,\ve})$ in Theorem \ref{thm:ECC}.

\begin{remark}\label{rmk:nochainbij}
In analogy with \cite{preech},  we have that ECH recovers the exterior algebra of the orbifold Morse homology of the base, and for the perturbation used, we also see this correspondence at the level of the chain complex. \end{remark}

\begin{notation}
When we discuss a specific Reeb current in multiplicative notation $\gamma_1^{m_1}\cdots\gamma_n^{m_n}$,  all $\gamma_i$ which appear will have $m_i>0$. However, if the Reeb current is not specified, it will be understood that $m_i=0$ is possible, which indicates the orbit set in the usual notation with that pair removed.  For example, the generators of $ECC_*^{L(\varepsilon)}(S^3,\lambda_{p,q,\varepsilon}, J)$ are of the form $\mathpzc{b}^B\mathpzc{h}^H\mathpzc{p}^P\mathpzc{q}^Q$, where $B,P,Q \in \Z_{\geq 0}$ and $H=0,1$. 
\end{notation}

Before proceeding, we make a final definition.
\begin{definition}\label{def:degreep}
Given a pair of homologous Reeb currents $\alpha$ and $\beta$ expressed in terms of embedded orbits realizing fibers of the prequantization orbibundle or Seifert fiber space, we define their (relative) \emph{degree} to be the relative algebraic multiplicity of the associated fiber sets.  For $\alpha = \mathpzc{b}^B\mathpzc{h}^H\mathpzc{p}^P\mathpzc{q}^Q$ and $\beta = \mathpzc{b}^{B'}\mathpzc{h}^{H'}\mathpzc{p}^{P'}\mathpzc{q}^{Q'}$, we have
\[
d(\alpha, \beta) =  \frac{B+H+\frac{1}{p}P+\frac{1}{q}Q-B'-H'-\frac{1}{p}P'-\frac{1}{q}Q'}{| e|}, \ \ \ \ |e| = \frac{1}{pq}.
\]
\end{definition}

When the differential does not vanish for index reasons, the degree $d$ of a pair $(\alpha,\beta)$ corresponds to the \textit{degree} of any curves counted in $\langle\partial\alpha,\beta\rangle$, similar to arguments in \cite[\S 4]{preech}; this is carried out in \S \ref{sss:orbipreech}. We also expect this to be true for more general Seifert fiber spaces and prequantization orbibundles. In \S \ref{ss:index-degree} we establish the relationship between the ECH index of a generator and its degree, which governs the behavior of the ECH spectral invariants computed in \S \ref{s:spectral}.

\subsection{Summary of the chain complex}\label{ss:pnot2}

\begin{table}[h!]
\begin{subtable}[c]{0.45\textwidth}
\centering
\begin{tabular}{||c|c|c||}
\hline
degree&generator&index
\\\hline 0& $\emptyset$ &0
\\\hline 3&$\zq$&2
\\\hline 4&$\zp$&4
\\\hline 6&$\zq^2$&6
\\\hline 7&$\zp\zq$&8
\\\hline 8&$\zp^2$&10
\\\hline 9&$\zq^3$&12
\\\hline 10&$\zp\zq^2$&14
\\\hline 11&$\zp^2\zq$&16
\\\hline 12&$\zq^4$&18
\\\hline 12&$\zp^3$&18
\\\hline 12&$\zh$&19
\\\hline 12&$\zb$&20
\\\hline 13&$\zp\zq^3$&22
\\\hline 14&$\zp^2\zq^2$&24
\\\hline 15&$\zq^5$&26
\\\hline 15&$\zp^3\zq$&26
\\\hline 15&$\zh\zq$&27
\\\hline 15&$\zb\zq$&28
\\\hline 16&$\zp\zq^4$&30
\\\hline 16&$\zp^4$&30
\\\hline 16&$\zh\zp$&31
\\\hline 16&$\zb\zp$&32
\\\hline 17&$\zp^2\zq^3$&34
\\\hline 18&$\zq^6$&36
\\\hline 18&$\zp^3\zq^2$&36
\\\hline 18&$\zh\zq^2$&37
\\\hline 18&$\zb\zq^2$&38
\\\hline 19&$\zp\zq^5$&40
\\\hline 19&$\zp^4\zq$&40
\\\hline 19&$\zh\zp\zq$&41
\\\hline 19&$\zb\zp\zq$&42
\\\hline 20&$\zp^5$&44
\\\hline 20&$\zp^2\zq^4$&44
\\\hline 20&$\zh\zp^2$&45

\\\hline
\end{tabular}
\subcaption{generators for $T(3,4)$}
\end{subtable}
\begin{subtable}[c]{0.45\textwidth}
\centering
\begin{tabular}{||c|c|c||}
\hline
degree&generator&index
\\\hline 0& $\emptyset$ &0
\\\hline 3&$\zq$&2
\\\hline 5&$\zp$&4
\\\hline 6&$\zq^2$&6
\\\hline 8&$\zp\zq$&8
\\\hline 9&$\zq^3$&10
\\\hline 10&$\zp^2$&12
\\\hline 11&$\zp\zq^2$&14
\\\hline 12&$\zp^4$&16
\\\hline 13&$\zp^2\zq$&18
\\\hline 14&$\zp\zq^3$&20
\\\hline 15&$\zq^5$&22
\\\hline 15&$\zp^3$&22
\\\hline 15&$\zh$&23
\\\hline 15&$\zb$&24
\\\hline 16&$\zp^2\zq^2$&26
\\\hline 17&$\zp\zq^4$&28
\\\hline 18&$\zq^6$&30
\\\hline 18&$\zp^3\zq$&30
\\\hline 18&$\zh\zq$&31
\\\hline 18&$\zb\zq$&32
\\\hline 19&$\zp^2\zq^3$&34
\\\hline 20&$\zp\zq^5$&36
\\\hline 20&$\zp^4$&36
\\\hline 20&$\zh\zp$&37
\\\hline 20&$\zb\zp$&38
\\\hline 21&$\zq^7$&40
\\\hline 21&$\zp^3\zq^2$&40
\\\hline 21&$\zh\zq^2$&41
\\\hline 21&$\zb\zq^2$&42
\\\hline 22&$\zp^2\zq^4$&44
\\\hline 23&$\zp\zq^6$&46
\\\hline 23&$\zp^4\zq$&46
\\\hline 23&$\zh\zp\zq$&47
\\\hline 23&$\zb\zp\zq$&48
\\\hline
\end{tabular}
\subcaption{generators for $T(3,5)$}
\end{subtable}

\caption{Since the homology must be the ECH of $S^3$, these tables indicate that the only nonzero differentials are a union of trivial cylinders together with two nontrivial cylinders with boundary $\zh-\zp^p$ and $\zh-\zq^q$. Because $\zh$ and $\zb$ are torus knots on the same torus, there are two cylinders with boundary $\zb-\zh$ which cancel, and therefore $\zb$ is closed. 
}
\label{table:genp}
\end{table}

When $p=2$, we utilized the very special symmetric presentation of the open book return map, enabling the use of an alternate orbifold Morse function wherein the differential vanished, as explained in \cite{kech}. However, in the setting at hand generators may contain the positive hyperbolic orbit $\zh$, thus by Theorem \ref{thm:Iproperties}(iv) (ECH Index Parity property), there are now odd ECH index generators in the complexes $ECC^{L(\varepsilon)}_*(S^3,\lambda_{p,q,\varepsilon},J)$, and we have nonzero differentials. 

The below theorem concerning the computation of the ECH index is proven in \S \ref{ss:combinatorics}.

\begin{theorem}\label{thm:pqI} The ECH index $I$ for $(S^3,\lambda_{p,q,\varepsilon})$ satisfies the following formula. {For any Reeb current $\zb^B\zh^H\zp^P\zq^Q$ with action less than $L(\varepsilon)$}, we have
\[
I(\zb^B\zh^H\zp^P\zq^Q)=-(P-Q)^2+2qP(H+B)+2pQ(H+B)+(pq-p-q)(H^2+B^2)+2HBpq+\mathbf{CZ},
\]
where $\mathbf{CZ}$ is the total Conley-Zehnder term
\[
\mathbf{CZ}=CZ^I_{orb}(\zb^B)+2(p+q)H+CZ^I_{orb}(\zp^P)+CZ^I_{orb}(\zq^Q),
\]
and
\begin{align*}
CZ^I_{orb}(\zb^B)&=(p+q)B^2+(p+q+1)B;
\\CZ^I_{orb}(\zp^P)&=\sum_{i=1}^P2\left\lfloor\left(\frac{p+q}{p}-\delta_{\zp,L}\right)i\right\rfloor+1;
\\CZ^I_{orb}(\zq^Q)&=\sum_{i=1}^Q2\left\lfloor\left(\frac{p+q}{q}-\delta_{\zq,L}\right)i\right\rfloor+1.
\end{align*}
\end{theorem}

\begin{remark}[Parity of the ECH index] As a sanity check, we show that the ECH index as computed in Theorem \ref{thm:pqI} is even if $H=0$ and odd if $H=1$. 
Consider the sum of the terms in the formula for the ECH index which are not always even integers; these are
\begin{equation}\label{eqn:paritypq}
-(P-Q)^2,\quad (pq-p-q)(H^2+B^2),\quad CZ^I_{orb}(\zb^B),\quad P,\quad Q.
\end{equation}
Here $P$ and $Q$ arise from the respective contributions from $CZ^I_{orb}(\zp^P)$ and $CZ^I_{orb}(\zq^Q)$ besides the $2\lfloor i\theta_T\rfloor$ terms, i.e. arising from the sums $\sum_{i=1}^P1$ and $\sum_{i=1}^Q1$. 
When $H=0$, the sum of the five terms in (\ref{eqn:paritypq}) is
\[
-P^2+2PQ-Q^2+pqB^2+(p+q+1)B+P+Q \equiv pqB^2+(p+q+1)B\mod2,
\]
hence if $B$ is even, then this term is necessarily even. If  $B$ is odd, note then $pqB+p+q+1$ cannot be odd unless both $p$ and $q$ are even, which cannot happen because they are relatively prime.  When $H=1$, we need only to show that $pq-p-q$ is odd, which again follows because both $p$ and $q$ cannot be even.
\end{remark}

 We will show {the first equality below} in Proposition \ref{prop:DL}, {a more refined version of Proposition \ref{prop:directlimitcomputesfiberhomology}:}
\begin{equation}\label{eqn:dirlimS3}
\lim_{\varepsilon\to0}ECH^{L(\varepsilon)}_*(S^3,\lambda_{p,q,\varepsilon},J)=ECH_{*}(S^3,\xi_{std})=\begin{cases}
\Z/2&\text{ if }* \in 2\Z_{\geq0}
\\0&\text{ otherwise}.
\end{cases}
\end{equation}
{In particular, this implies that there can be no odd index homologically essential generators.} Furthermore, Corollary \ref{prop:bijectionp} allows us to conclude that there must be some nonzero differential coefficients in $ECC^{L(\varepsilon)}_*(S^3,\lambda_{p,q,\varepsilon},J)$ for grading small enough relative to $L(\varepsilon)$ in the sense of Lemma \ref{lem:orbitseh}, and they must come from moduli spaces $\M_1(\alpha,\beta,J)$ with $H\neq0$ in either $\alpha$ or $\beta$ (but not both).  Later, in the proof of Theorem \ref{thm:kECH} we analyze the effect of the maps in the direct system in (\ref{eqn:dirlimS3}) on the knot filtration.

In order to understand which generators do not appear in homology, namely the cycles which are also boundaries, we relate {the associated} ECH index one moduli spaces to the moduli spaces determining the orbifold Morse differential of $\mathpzc{H}_{p,q}$ in \S \ref{sss:orbipreech}. { These arguments are rather subtle and involved as we first employ a non-generic $S^1$ invariant almost complex structure used to determine the relevant cylinder counts and then explain how the cylinder counts agree for nearby generic $J$.

The following theorem summarizes our results.
\begin{theorem}\label{thm:ECC}
The chain complex $ECC^{L(\ve)}_*(S^3,\lambda_{p,q,\ve})$ is generated by admissible Reeb currents of the form $\zb^B\zh^H\zp^P\zq^Q$ with action at most $L(\ve)$. All differential coefficients are zero besides $\langle \partial \zh\gamma,\zp^p\gamma \rangle = \langle \partial \zh\gamma,\zq^q\gamma,J \rangle=1$, 
for any Reeb current $\gamma$ with $H=0$. In particular:
\begin{itemize}
\itemsep-.25em
\item $\partial(\zb^B\zp^P\zq^Q)=0$ for all $B,P,Q$;
\item If $I(\alpha)=I(\beta)$ then $\alpha$ and $\beta$ are homologous in ECH; 
\item $\alpha$ and $\beta$ have the same degree if one may be obtained from the other by repeatedly replacing one end of a Morse flow line of $\mathpzc{H}_{p,q}$ with the other.
\end{itemize}
\end{theorem}
The proof of Theorem \ref{thm:ECC} (besides the third bullet point) appears as the first paragraph of the proof of Lemma \ref{lem:degNk}. The third bullet point is an immediate consequence of Definition \ref{def:degreep}. A converse to the third bullet point in Theorem \ref{thm:ECC} is proved in Lemma \ref{lem:degNk}, see Remark \ref{rmk:converse}.  The description of the chain complex provided by Proposition \ref{cor:nodiffp} allows us to complete the computation of the spectral invariants in \S\ref{s:spectral}. 

\begin{remark}\label{rem:nongeneric}
{We will {initially} use a non-generic $S^1$-invariant $\lambda_{p,q,\varepsilon}$-compatible almost complex structure $\frak{J} = \frak{p}^*j_{\CP^1_{p,q}}$ to compute {a portion of} the differential for the cylindrical curves in our appropriately action filtered chain complex $(ECC^{L(\ve)}_*(S^3,\lambda_{p,q,\varepsilon},J),\partial)$ in Proposition \ref{prop:pqM}.   {Lemma \ref{lem:pqM} will establish that these cylinder counts agree for any sufficiently nearby $J$, and since the space of $\lambda_{p,q,\varepsilon}$-compatible almost complex structures is contractible, there will also be $J$ which are generic and sufficiently close to $\frak{J} $ .} 
 That this cylindrical portion of the differential is well-defined will follow from Lemma \ref{lem:d0g0}.   Using alternate arguments, we will show that no other curves contribute to the complex; these are carried out in \S \ref{ss:combinatorics} and \S \ref{sss:finalcomp}.  }
\end{remark}

\subsection{Review of index inequalities}\label{ss:currents}
We collect some applications of Hutchings' intersection theory \cite{Hindex, Hrevisit}, which we make use of to help establish the behavior of the ECH differential in \S \ref{sss:orbipreech}.  We begin with the ECH index inequality, which is key to the existence of the theory of embedded contact homology, and follows from the relative adjunction formula and writhe bounds, cf. \cite[\S 3.3-3.4]{lecture}.

\begin{theorem}[{The ECH index inequality \cite[Thm.~1.7]{Hindex}, \cite[Thm.~4.15]{Hrevisit}}]\label{thm:indexineq} \ \\ Suppose that $C \in \mathcal{M}(\alpha,\beta,J)$ is somewhere injective.  Then \[ \mbox{\em ind}(C) \leq I(C) - 2 \delta(C). \] 
Equality holds only if the $C$ is embedded and the partition conditions hold: $\{  q_{i,k}^+ \} = P_{\alpha_i}^+(m_i)$ for each $i$ and $\{  q_{j,k}^- \} = P_{\beta_j}^-(n_j)$ for each $j$. \end{theorem}

Here $\delta(C)$ is a count of the singularities of a somewhere-injective $J$-holomorphic curve $C$ in $\R \times Y$ with positive integer weights as in \cite[\S 7]{mw} that equals zero when $C$ is embedded.\footnote{In this section we use notation consistent with \cite{Hrevisit}, using the quantity $\delta(C)$, which is distinct from the $\delta$s with or without subscripts used to modify monodromy angles appearing throughout this paper, and we never require $\delta(C)>0$ in our proofs, so the exact definition is not needed.} 
The above statement also invokes the \emph{ECH partition conditions}, which are a topological type of data, namely the covering multiplicities of the Reeb orbits at the ends of the nontrivial components associated to the $J$-holomorphic curves underlying the currents which can be obtained indirectly from various ECH and Fredholm index relations. Namely, given a $J$-holomorphic curve $C \in \mathcal{M}(\alpha,\beta,J)$, where $\alpha = \{ (\alpha_i,m_i)\}$ and $\beta = \{ (\beta_j,n_j \}$, for each $i$ let $a_i^+$ denote the number of positive ends of $C$ at $\alpha_i$ and let $\{ q_{i,k}^+\}_{k=1}^{a_i^+}$ denote their multiplicities, $\sum_{k=1}^{a_i^+} q_{i,k}^+=m_i$.   Likewise, for each $j$ let $b_j^-$ denote the number of negative ends of $C$ at $\beta_j$ and let $\{ q_{j,k}^-\}_{k=1}^{b_j^-}$ denote their multiplicities, $\sum_{k=1}^{b_j^-} q_{j,k}^-=n_j$. The definition and computation of the partition conditions is carried out at end of this section, cf. Definition \ref{def:part}.  

The \emph{Fredholm index} of a $J$-holomorphic curve $C \in \M(\alpha,\beta,J)$ is given by
\begin{equation}\label{eqn:Find}
\op{ind}(C) = -\chi(C) + 2c_\tau(C) + \sum_i\sum_{k=1}^{a_i^+} CZ_\tau\left( (\alpha_i)^{q_{i,k}^+}\right) - \sum_j\sum_{k=1}^{b_j^-} CZ_\tau\left( (\beta_j)^{q_{j,k}^-}\right).
\end{equation}

We also have an inequality on the ECH index of the union of two $J$-holomorphic curves.

\begin{theorem}{\em \cite[Thm.~5.1]{Hrevisit}}\label{thm:intI}
If $C$ and $C'$ are $J$-holomorphic curves in $\R \times Y$, then
\[
I(C \cup C') \geq I(C) + I(C') + 2 C \cdot C,'
\]
where $C \cdot C'$ is an ``intersection number" of $C$ and $C'$ defined in \cite[\S 5.1]{Hrevisit}. \end{theorem}

The ECH index inequality permits the following classification of low ECH index currents.

\begin{proposition}{\em \cite[Prop. 3.7]{lecture}}\label{lowiprop}
Suppose that $J$ is a generic $\lambda$-compatible almost complex structure.  Let $\alpha$ and $\beta$ be Reeb currents and let $\mathcal{C} \in \M(\alpha,\beta,J)$ be any $J$-holomorphic current in $\R \times Y$, not necessarily somewhere injective.  Then:
\begin{enumerate}[{\em (i)}]
\item We have $I(\mathcal{C})\geq 0$ with equality if and only if $\mathcal{C}$ is a union of trivial cylinders with multiplicities.
\item If $I(\mathcal{C}) =1$ then $\mathcal{C}= \mathcal{C}_0 \sqcup C_1$,  where $I(\mathcal{C}_0)=0$, and $C_1$ is embedded and has $\ind(C_1)=I(C_1)=1$.
\item If $I(\mathcal{C}) =2$, and if $\alpha$ and $\beta$ are generators of the chain complex $ECC_*(Y,\lambda,\Gamma,J)$, then $\mathcal{C}= \mathcal{C}_0 \sqcup C_2$, where $I(\mathcal{C}_0)=0$, and $C_2$ is embedded and has $\ind(C_2)=I(C_2)=2$.
\end{enumerate}
\end{proposition}

In order to establish the behavior of the differential, we will show that currents between {certain pairs of admissible Reeb orbit sets}, which are counted by the differential, must be given in terms of unions of cylinders (one of which is nontrivial).  Lemma \ref{lem:d0g0}, establishes for any $\lambda_{p,q,\varepsilon}$-compatible almost complex structure that the differential coefficient for key index difference one pairs is well-defined and counts only cylinders.  The proof makes use of the topological index $J_0$, a variant of the ECH index\footnote{Note that $J_0$ is NOT the standard almost complex structure on $\C^n$.} from \cite[\S 6]{Hrevisit}, which bounds the topological complexity of a $J$-holomorphic curve $C$.  The index $J_0$ and its related variants enjoy many of the same basic properties as the ECH index $I$, which we summarize as needed below.

\begin{definition}\label{def:J0}
If $\alpha=\{ (\alpha_i,m_i) \}$ and $\beta = \{ (\beta_j,n_j) \}$ are Reeb currents with $[\alpha]=[\beta] \in H_1(Y),$ and if $Z \in H_2(Y,\alpha,\beta)$, then we define
\[
J_0(\alpha,\beta,Z):=-c_\tau(Z) + Q_\tau(Z)+ CZ_\tau^J(\alpha,\beta),
\] 
where 
\[
CZ_\tau^J(\alpha,\beta) = \sum_i \sum_{k=1}^{m_i-1}CZ_\tau(\alpha_i^k) -\sum_j \sum_{k=1}^{n_j-1}CZ_\tau(\beta_j^k). 
\]
The differences between the definition of $I$ in Definition \ref{defn:ECHI} and $J_0$ are that the sign of the relative Chern class term is switched and the Conley-Zehnder terms are slightly different, as we do not sum to the last respective $m_i$ and $n_j$ iterates of $\alpha$ and $\beta$.
\end{definition}

The following result is the analogue of the ECH index inequality for $J_0$, which incorporates the genus of a $J$-holomorphic curve. 
\begin{proposition}{\em\cite[Prop.~6.9]{Hrevisit}}\label{prop:J0}
Suppose $C \in \M(\alpha,\beta,J)$ is a somewhere injective and irreducible $J$-holomorphic curve.  Then
\begin{equation}\label{eq:J0irr}
J_0(C) \geq 2(g-1 + \delta(C)) + \sum_\gamma \left\{ \begin{array}{ll}2n_\gamma - 1 & \gamma \mbox{ elliptic,} \\ m_\gamma & \gamma \mbox { positive hyperbolic,} \\ (m_\gamma + n_\gamma^{\op{odd}})/2 & \gamma \mbox{ negative hyperbolic.} \\ \end{array}\right.
\end{equation}
Here the sum is over all embedded Reeb orbits $\gamma$ in $Y$ at which $C$ has ends; $n_\gamma$ denotes the number of ends of $C$ at {covers of} $\gamma$, $m_\gamma$ denotes the total multiplicity of the ends of $C$ at $\gamma$; and $ n_\gamma^{\op{odd}}$ denotes the number of ends of $C$ at covers of $\gamma$ with odd multiplicity.
\end{proposition}

\begin{remark}\label{rem:J0sharp}
One can also deduce the original ECH index inequality from the $J_0$ index inequality. Additionally, if $\op{ind}(C) = I(C)$, e.g. $C$ is an irreducible curve in $\R \times Y$ contributing to the ECH differential and does not contain trivial cylinders,  then \eqref{eq:J0irr} is sharp. These facts are explained in \cite[\S6]{Hrevisit}. 
\end{remark}

This provides the useful important corollary that we will make use of to obtain genus bounds on the curves we will count.

\begin{corollary}{\em \cite[Cor.~6.10]{Hrevisit}}\label{cor:J0genus}
If $C \in \M(\alpha,\beta,J)$ is somewhere injective, but not necessarily assumed to be irreducible, then
\[
-\chi(C) \leq J_0(C) - 2\delta(C).
\]
\end{corollary}

We will also make use of the following inequality on the $J_0$ index of the union of two $J$-holomorphic curves.\footnote{There is an erratum to \cite[Prop. 6.14, Thm.~6.6]{Hrevisit}; Proposition \ref{thm:intJ} has been corrected accordingly.  See: \url{https://floerhomology.wordpress.com/2014/07/09/erratum-to-the-ech-index-revisited/}}

\begin{proposition}{\em \cite[Prop.~6.14]{Hrevisit}}\label{thm:intJ}
If $C$ and $C'$ are $J$-holomorphic curves in $\R \times Y$, then
\[
J_0(C \cup C') \geq J_0(C) + J_0(C') + 2 C \cdot C' + E + N,
\]
where 
\begin{itemize}
\itemsep-.25em
\item $C \cdot C'$ is an ``intersection number" of $C$ and $C'$ defined in \cite[\S 5.1]{Hrevisit}; 
\item $E$ denotes the number of elliptic Reeb orbits in $Y$ at which both $C$ and $C'$ have ends;
\item If $C$ and $C'$ have no component in common then $N$ denotes the number of negative hyperbolic orbits $\gamma$ in $Y$ such that the total multiplicity of the ends of $C$ at $\gamma$ and the total multiplicity of the ends of $C'$ at $\gamma$ are both odd;
\item If $C$ and $C'$ have a component in common, set $N=0$.
\end{itemize}
\end{proposition}

We conclude this section with the definition of partition conditions and compute them in some examples, which will be relevant to our computation of the chain complex.

\begin{definition}\label{def:part} \cite[\S 3.9]{lecture} Let $\gamma$ be an embedded Reeb orbit and $m$ a positive integer.  We define two partitions of $m$, the \emph{positive partition} $P^+_\gamma(m)$ and  the \emph{negative partition} $P^-_\gamma(m)$ as follows.\footnote{Previously the papers \cite{Hindex, Hrevisit} used the terminology incoming and outgoing partitions.}  
\begin{itemize}
\itemsep-.25em
\item If $\gamma$ is positive hyperbolic, then $P_\gamma^+(m): = P_\gamma^-(m): = (1,...,1).$
\item If $\gamma$ is negative hyperbolic, then $
P_\gamma^+(m): = P_\gamma^-(m): = \left\{ \begin{array}{ll}
(2,...,2) & m \mbox{ even,} \\
(2,...,2,1) & m \mbox{ odd. } \\
\end{array}
\right .
$
\item If $\gamma$ is elliptic then the partitions are defined in terms of the monodromy angle $\theta \in \R / \Z $.  We write $P_\gamma^\pm(m): = P_\theta^\pm(m),$ with the right hand side defined as follows. 

 Let $\Lambda^+_\theta(m)$ denote the highest concave  polygonal path in the plane that starts at $(0,0)$, ends at  $(m,\lfloor m \theta \rfloor)$, stays below the line $y = \theta x$, and has corners at lattice points.  Then the integers $P^+_\theta(m)$ are the horizontal displacements of the segments of the path $\Lambda^+_\theta(m)$ between the lattice points.

Likewise, let  $\Lambda^-_\theta(m)$ denote the  lowest convex polygonal path in the plane that starts at $(0,0)$, ends at $(m,\lceil m \theta \rceil)$, stays above the line $y = \theta x$, and has corners at lattice points.  Then the integers $P^-_\theta(m)$ are the horizontal displacements of the segments of the path $\Lambda^-_\theta(m)$ between the lattice points.

Both $P_\theta^\pm(m)$ depend only on the class of $\theta$ in $\R / \Z$.  Moreover, $P_\theta^+(m) = P_{-\theta}^-(m)$.
\end{itemize}

\end{definition}

\begin{figure}[h]
 \begin{center}
\includegraphics[width=.5\textwidth]{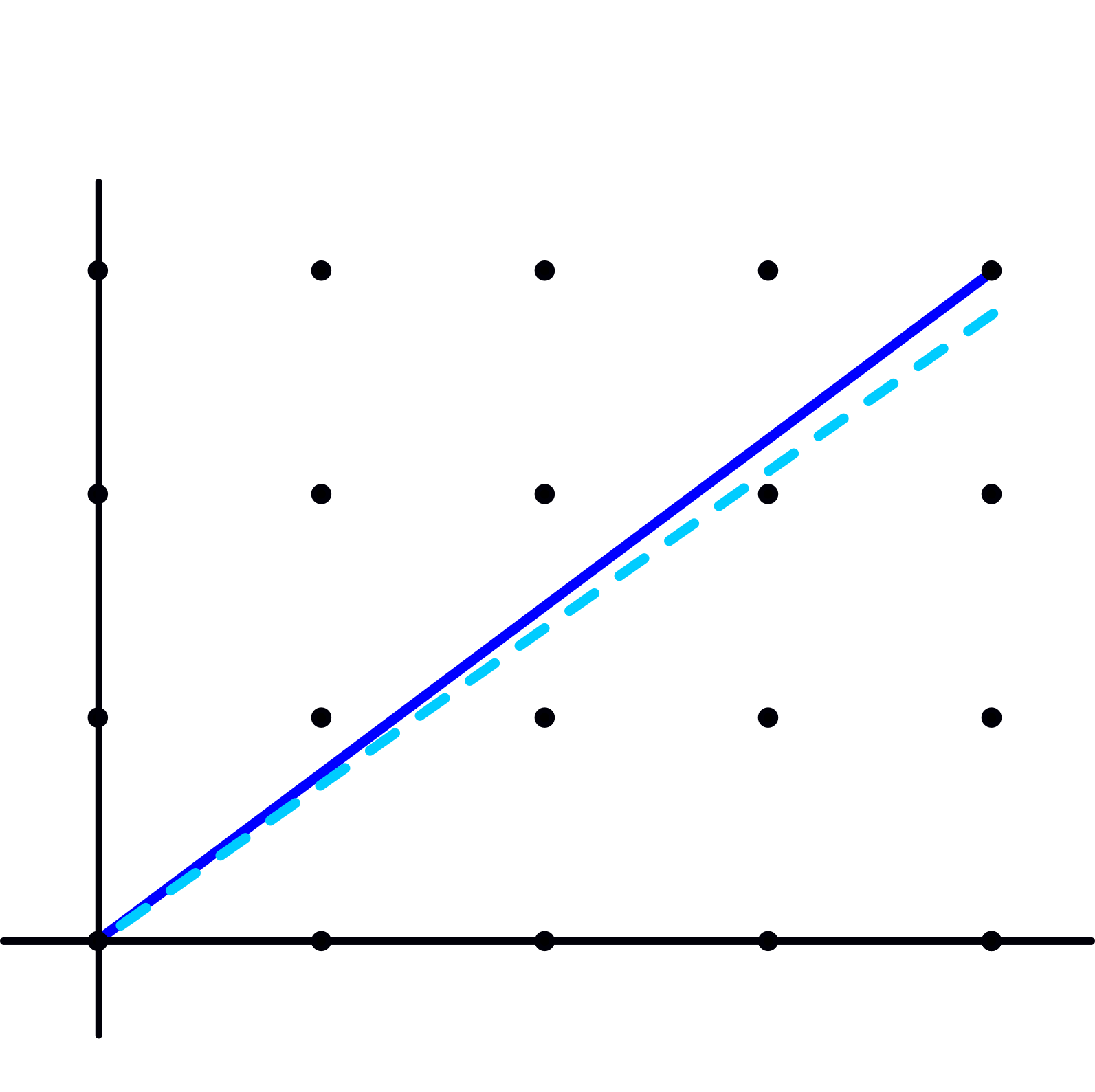}
\end{center}
\caption{The path $\Lambda^-(q)$ for $(p,q)=(3,4)$ is in solid dark blue; it connects the origin to $(q,\lceil pq-q\delta_{\zq,L}\rceil)=(q,pq)$. The line $y=(3/4-\delta_{\zq,L})x$ is in dashed light blue.  Because $p$ and $q$ are coprime, we have that for the embedded Reeb orbit $\zq$, $P^-_{p/q-\delta_{\zq,L}}(q)=(q)$.} 
\label{fig:p-}
\end{figure}

\begin{example}\label{ex:partitions}
If the class of the monodromy angle of an elliptic orbit $\gamma$ satisfies $[\theta] \in (0,1/m)$ then $P_\theta^+(m) = (1,...,1) $ and
$P_\theta^-(m) = (m),$ whereas the partitions are quite complex for other $\theta$ values, see \cite[Fig. 1]{lecture}.
\end{example}

\begin{remark}\label{rem:partb}
{For the binding $\zb$, the class of its monodromy angle $ [\theta ] \in \R/\Z$ is $\delta_{\zb,L} \in (0, 1/m)$, for $\varepsilon$ sufficiently small, thus} $P_{\zb}^+(m) =   (1, ... ,1)$ and  $P_{\zb}^-(m) = (m)$.
\end{remark}

\begin{remark}\label{rmk:p-} As a sanity check, we verify that the partition conditions determined by the ECH and Fredholm index 1 curves agree with the expression of the curve as a lift of gradient flow lines of $\mathcal{H}_{p,q}$ using an $S^1$-invariant $\lambda_{p,q,\varepsilon}$-compatible almost complex structure $\frak{J} = \frak{p}^*j_{\CP^1_{p,q}}$.  We consider the moduli spaces 
\[
\M(\zh\gamma,\zq^q\gamma,\frak{J}), \ \ \  \M(\zh\gamma,\zp^p\gamma,\frak{J}), \ \ \  \M(\zb\gamma,\zh\gamma,\frak{J}),
\]
where $\gamma$ is an admissible Reeb current not containing $\zh$. 
Observe that the curves in the moduli spaces of Proposition \ref{prop:pqM} which are lifts of
\begin{enumerate}[{(i)}]
\itemsep-.25em
\item The gradient flow line of $\mathpzc{H}_{p,q}$ from $\zh$ to $\zq^q$;
\item The gradient flow line of $\mathpzc{H}_{p,q}$ from $\zh$ to $\zp^p$;
\item Either of the gradient flow lines of $\mathpzc{H}_{p,q}$ from $\zb$ to $\zh$;
\end{enumerate}
satisfy the following expected respective partition conditions: 
\begin{enumerate}[{(i)}]
\itemsep-.25em
\item The  partition of $(\zh,1)$ is $(1)$ and $(\zq,q)$ is $(q)$. 
\item The partition of $(\zh,1)$ is $(1)$ and $(\zp,p)$ is $(p)$. 
\item The partition of $(\zb,1)$ is $(1)$ and $(\zh,1)$ is $(1)$. 
\end{enumerate}
Partitions of one are always one, but it remains to check that the negative partition of $q$ of an elliptic Reeb orbit with rotation number $(p+q)/q-\delta_{\zq,L}$ is $q$ and the analogous fact for $p$ (whose proof we omit).  Because partitions only depend on the rotation number mod $\Z$, we may assume that $\zq$ has rotation number $p/q-\delta_{\zq,L}$.  The path $\Lambda^-(q)$ goes from the origin to $(q,\lceil pq-q\delta_{\zq,L}\rceil)=(q,pq)$; because $p$ and $q$ are coprime, we must have ${P^-_{\zq}(q)= }P^-_{p/q-\delta_{\zq,L}}(q)=q$. {See Figure \ref{fig:p-} for an illustration.}
\end{remark}

\subsection{Topological complexity and cylinder counting}\label{sss:orbipreech}
In this section we establish the topological complexity of certain ECH and Fredholm index one, degree zero curves, which we use to show that the action filtered differential counts cylinders which are the union of fibers over orbifold Morse flow lines in the base $\CP_{p,q}^1$, along with trivial cylinders.   These contributions to the action filtered ECH differential occur for the following pairs of ECH index difference 1 admissible Reeb currents, namely
\begin{equation}\label{eq:pairs5}
(\zh\gamma,\zq^q\gamma), \ \ \  (\zh\gamma,\zp^p\gamma), \ \ \  (\zb\gamma,\zh\gamma),
\end{equation}
where $ \gamma$ is an admissible Reeb current generating $ECC_*^{L}(S^3,\lambda_{p,q,\varepsilon(L)})$ not admitting $\zh$. 

As previously observed, automatic transversality \cite{jo2, wendl}, \cite[Prop.~4.7]{preech}, allows us to  count cylinders with a generic fiberwise $S^1$-invariant almost complex structure  $\frak{J}:= \mathfrak{p}^*j_{\CP_{p,q}^1}$,   the $S^1$-invariant lift of $j_{\CP_{p,q}^1}$, but we cannot use $\frak{J}$ for higher genus curves because it cannot be independently perturbed at the intersection points of $\pi_{S^3} C$ with a given $S^1$-orbit by an $S^1$-invariant perturbation.  { We will show that there are no noncylindrical curves of interest to count via Lemma \ref{lem:d0g0} and \S \ref{sss:finalcomp}.  } 

We first clarify how the degree of a pair of admissible homologous Reeb currents expressed in terms of embedded orbits realizing fibers of the prequantization orbibundle, as in Definition \ref{def:degreep}, corresponds to the degree of completed projection of a pseudoholomorphic curve in the symplectization of the prequantization orbibundle, similar to \cite[\S 4]{preech}.    

\begin{definition}[Degree of a completed projection]\label{def:degree}
Composing the $J$-holomorphic curve $C: \ds \to \R \times S^3$ with the projection $\fp: S^3 \to \cpq,$ yields a map $\fp \circ \pi_{S^3} C: \ds \to \cpq,$ which has a well-defined non-negative \emph{degree} because $\fp \circ \pi_{S^3} C$ extends to a map of closed surfaces. We define the \textit{degree} of $C$, denoted $\op{deg}(C)$, to be the degree of this map.  
\end{definition}

The degree of the completed projection map $\fp \circ \pi_{S^3} C$ is related to the number of positive and negative ends via the $d\lambda_\varepsilon$-energy and Stokes' Theorem as follows.   Recall from Lemma \ref{lem:efromL} that the action of a Reeb orbit $\gamma_x^{k|\Gamma_x|}$ of $R_{p,q,\varepsilon}$ over a critical point $x$ of $\mathpzc{H}_{p,q}$ is proportional to the length of the fiber, namely
\begin{equation}\label{e:lemM}
\mathcal{A}\left(\gamma_x^{k|\Gamma_x|}\right) = \int_{\gamma_x^{k|\Gamma_x|}} \lambda_{p,q,\varepsilon} = k (1+\varepsilon \fp^*\mathpzc{H}_{p,q}(x)).
\end{equation}
because $\fp^*\mathpzc{H}_{p,q}$ is constant on critical points $x$ of $\mathpzc{H}_{p,q}$.  Tracing this proportionality through the repeated use of Stokes' theorem, as in the proof of \cite[Prop.~4.4]{preech}, and using the fact that $\omega_{p,q}[\CP^1_{p,q} ]= |e|$, yields the following relationship between the difference in weighted multiplicity of the fiber orbits and $|e|$ times degree of a curve.

\begin{proposition}\label{prop-degree} The relation between the degree $\op{deg}(C)$ of a curve $C\in\mathcal{M}(\alpha,\beta,J)$ and the total multiplicity of the admissible fiber Reeb currents generating $ECC_*^{L(\varepsilon)}(S^3,\lambda_{p,q,\varepsilon},J)$ at the positive and negative ends is given by
\[
|e|\op{deg}(C) = B+H+\frac{1}{p}P+\frac{1}{q}Q-B'-H'-\frac{1}{p}P'-\frac{1}{q}Q',
\]
where $\alpha=\zb^B\zh^H\zp^P\zq^Q$ and $\beta=\zb^{B'}\zh^{H'}\zp^{P'}\zq^{Q'}$, with $|e|=\frac{1}{pq}$.
\end{proposition}

As a consequence of Proposition \ref{prop-degree}, the degree of any two curves in $\mathcal{M}(\alpha,\beta,J)$ must be equal, and agree with the definition of degree for a pair of Reeb currents in Definition \ref{def:degreep}.  Next, we need to verify that certain curves, which contribute nontrivially to the ECH differential, are degree zero cylinders.

\begin{remark}
For prequantization bundles, in \cite[Lem.~4.6]{preech}, we proved that a curve $C$ which contributes nontrivially to the differential is degree zero if and only if it is a cylinder.  While we expect the result to be true in greater generality, it is sufficient for the purposes of this paper to restrict to certain orbit sets. See Remark \ref{rmk:comp} and \S\ref{sss:finalcomp}.
\end{remark}

\begin{lemma}\label{lem:d0g0} {Let $J$ be any $\lambda_{p,q,\varepsilon}$-compatible almost complex structure} and $\cur \in\mathcal{M}_1(\alpha,\beta,J)$ be any $J$-holomorphic current in $\R \times Y$ with $I(\cur)=1$.  Further suppose that the pair $(\alpha,\beta)$ is of one of the forms
\begin{equation}\label{eq:1pairs}
 (\zb\gamma,\zh\gamma),  (\zh\gamma,\zp^p\gamma), (\zh\gamma,\zq^q\gamma),
\end{equation}
where $ \gamma$ is an admissible nondegenerate fiber Reeb current associated to $\lambda_{p,q,\varepsilon}$ of the form $\zb^B\zp^P\zq^Q$ such that the total actions of $\alpha$ and $\beta$ are less than $L(\varepsilon)$.  Then any $J$-holomorphic curve $C \in\mathcal{M}_1(\alpha,\beta,J)$ must consist of a {union of trivial cylinders and one nontrivial cylinder of ECH and Fredholm index one in the moduli space with $\gamma = \emptyset$.}
\end{lemma}

\begin{proof} 
In Lemma \ref{lem:ECHd}, we establish that each of the pairs in \eqref{eq:1pairs} all satisfy $I(\alpha,\beta)=1$.  By inspection, we also have that  $\op{deg}(C) =0$.  Let $\gamma = \zb^B\zp^P\zq^Q$ where $B,P,Q \geq 0$; we will also separately treat the cases where  $B=0$, $P=0$, and $Q=0$.  \\

\textbf{Step 1.} First, we compute $J_0$ of the pairs in \eqref{eq:1pairs}. Using the fact that $I(\alpha,\beta)=1$, Definition \ref{def:J0}, and Lemma \ref{lem:relfirstCherncalcp}, we have that
\[
Q_{orb}(\alpha,\beta) + CZ_{orb}^J(\alpha,\beta) = 1- \sum_i CZ_{orb}(\alpha_i^{m_i}) + \sum_j CZ_{orb}(\beta_j^{n_j}). \]
Thus
\[
J_0(\alpha,\beta, [C]) = 1 - \sum_i CZ_{orb}(\alpha_i^{m_i}) + \sum_j CZ_{orb}(\beta_j^{m_j}).
\]
Unwinding further, and using Lemma \ref{lem:orbtrivCZp}, we obtain
\begin{align}
J_0 (\zb\gamma,\zh\gamma, [C]) = & \ 1 - CZ_{orb}(\zb) + CZ_{orb}(\zh)  \ \ = \ 0,  \ \ \ \mbox{if $\gamma = \zp^P\zq^Q$;} \label{eq:pairsnogamma1} \\
J_0 (\zh \gamma, \zp^p \gamma, [C]) = & \ 1 - CZ_{orb}(\zh) + CZ_{orb}(\zp^p) \ = \ 0, \ \ \ \mbox{if $\gamma = \zb^B\zq^Q$;}  \label{eq:pairsnogamma2} \\
J_0 (\zh\gamma, \zq^q\gamma, [C]) = & \ 1 - CZ_{orb}(\zh) + CZ_{orb}(\zq^q) \ = \ 0, \ \ \ \mbox{if $\gamma = \zb^B\zp^P$}. \label{eq:pairsnogamma3}
\end{align}
For $\gamma =\zb^B \zp^P\zq^Q$ with $B \geq 1$, we have
\begin{equation}\label{eq:J1}
\begin{split}
J_0 (\zb\gamma,\zh\gamma, [C]) = & \ 1 - CZ_{orb}(\zb^{B+1}) + CZ_{orb}(\zh) + CZ_{orb}(\zb^B), \\
=& \ 1- 2(B+1)(p+q) -1 + 2(p+q) + 2B(p+q) + 1, \\
=& \ 1. 
\end{split}
\end{equation}
Additionally, for $\gamma =\zb^B \zp^P\zq^Q$ with $P \geq 1$
\begin{equation}\label{eq:J2}
\begin{split}
J_0 (\zh\gamma,\zp^p\gamma, [C]) = & \ 1 - CZ_{orb}(\zh) - CZ_{orb}(\zp^P) + CZ_{orb}(\zp^{P+p}), \\
 \leq & \ 1,
\end{split}
\end{equation}
because by Lemma \ref{lem:orbtrivCZp} and for sufficiently small positive $\delta$, {namely $\delta < 1/p$}, 
\[
\begin{split}
CZ_{orb}(\zp^{P+p}) & = \ 2 \left \lfloor \left( \frac{p+q}{p} - \delta \right)(P+p) \right \rfloor + 1 \\
& = \ 2\left \lfloor \left( \frac{p+q}{p} - \delta \right)P + (p+q -p\delta) \right \rfloor + 1 \\
& \leq \ CZ_{orb}(\zp^{P}) + 2(p+q),
\end{split}
\]
since $\lfloor a + b \rfloor \leq \lfloor a  \rfloor +  \lfloor b \rfloor + 1.$ An identical argument shows for $\gamma =\zb^B \zp^P\zq^Q$, for $Q \geq 1$
\begin{equation}\label{eq:J3}
J_0 (\zh\gamma,\zq^q\gamma, [C]) \leq 1.
\end{equation}

\textbf{Step 2.} Next, we use the index inequalities detailed in \S \ref{ss:currents} to deduce topological information based on the $J_0$ bounds above.   By invoking Proposition \ref{prop:J0} in conjunction with \eqref{eq:pairsnogamma1}-\eqref{eq:pairsnogamma3}, we obtain that if $C \in \M(\zb,\zh, J), \ \M(\zh, \zp^p, J),$ or $\M(\zh, \zq^q, J)$ then $g(C)=0$, $m_\zh=1$, and $n_\gamma=1$ for $\gamma=\zb,\zp,$ or $\zq$. 
Thus $C$ is a cylinder.  This conclusion alternatively follows from the discussion about the associated partition conditions in Remark \ref{rmk:p-}.

By Corollary \ref{cor:J0genus}, if $C \in \M(\zb \gamma,\zh\gamma, J), \ \M(\zh\gamma, \zp^p\gamma, J),$ or $\M(\zh\gamma, \zq^q\gamma, J)$ is somewhere injective, but not necessarily assumed to be irreducible, then
\begin{equation}\label{eq:J0-bound}
-\chi(C) \leq J_0(C) - 2\delta(C).
\end{equation}
By the proof of Proposition \ref{lowiprop} (cf. \cite[\S 3.4]{lecture}) and Remark \ref{rem:nongeneric} we can assume that $\cur$ is a union of $k$ (not necessarily distinct) somewhere injective irreducible curves $C_i$, 
as well as a union of trivial cylinders $\mathcal{T}$:
\[
\cur = \left(\bigcup_{i=1}^k C_i \right) \sqcup \mathcal{T}.
\]
Since we are not assuming $J$ is generic, we do not want to invoke the full classification result of Proposition \ref{lowiprop}(ii) to immediately conclude that there is only one $C_i$ with $\op{ind}(C_i) = 1$.  This would follow if we demonstrated directly without a genericity assumption on $J$ that the Fredholm index of every somewhere injective curve which is not a trivial cylinder is at least one, since the Fredholm index is additive with respect to taking unions of curves.  Instead we mimic the proof of Proposition \ref{lowiprop}(ii), employ the above computations of $J_0$, and a few more results from \S \ref{ss:currents} on the ECH index and $J_0$ with respect to unions of curves to deduce that $\cur = C_1 \sqcup \mathcal{T}$, where $C_1$ is a nontrivial Fredholm and ECH index one cylinder.  
By \eqref{eq:J0-bound} and the preceding computations of $J_0$,
\begin{equation}
\begin{split}
\sum_{i=1}^k - \chi(C_i) \leq & \ \ J_0( C_1 \cup ... \cup C_k), \\
\sum_{i=1}^k ( 2g(C_i) + \# \mbox{punctures}(C_i) - 2) \leq & \ \  0 \mbox{ or } \  1. \label{eq:*}
\end{split}
\end{equation}
\medskip

\textbf{Step 3.} We now observe that no $C_i$ can have a single puncture. 
Suppose $C_i$ has a single puncture.  Any holomorphic plane of ECH index one must, by Theorem \ref{thm:Iproperties}(ii, iv), appear in the moduli space $\M(h,\emptyset,J)$. But there are no such planes of ECH index one by an elementary computation using Theorem \ref{thm:pqI},
\[
I(\zh)=pq-p-q+2(p+q)=pq+p+q>1.
\]
If $C_i$ has a positive puncture at an elliptic Reeb orbit $\zb, \zp,$ or $\zq$ then $I(C_i) > 1$ by Theorem \ref{thm:pqI} or Theorem \ref{thm:Iproperties}(iv).  If we write $C = C' \cup C_i $ where $C' = \left( C_1 \cup ... \cup \widehat{C_i} \cup ... \cup C_k \right) \sqcup \mathcal{T}$  then  by Theorem \ref{thm:intI}, 
\[
1 = I(C) = I(C' \cup C_i) \geq I(C') + I(C_i)  + 2 C' \cdot C_i.
\]
This produces a contradiction because $I(C') \geq 0$, $ I(C_i) > 1$ and $C' \cdot C_i \geq 0$ because $C'$ and $C_i$ cannot have an irreducible component in common. 

\medskip

\textbf{Step 4.}  {Finally, we observe that it is not possible for any $C_i$ to be a thrice punctured genus zero curve.  (At most one irreducible component can have 3 punctures; 
the rest must be cylinders by the preceding argument and \eqref{eq:*}.) Suppose $C_i$ is an irreducible component with three punctures, {which contributes to the ECH differential}. Since $I(C_i)=1$, then by Index Parity, Theorem \ref{thm:Iproperties}(iv) $C_i$ must have one end on $h$, thus $C_i$ is somewhere injective (and in fact embedded by  Remark \ref{rem:nongeneric}.)

 By \eqref{eq:J0-bound} we must have $J_0(C_i) =1$. (This rules out the configurations in \eqref{eq:pairsnogamma1}-\eqref{eq:pairsnogamma3}.)} If we write $C = C' \cup C_i$  where $C' = \left( C_1 \cup ... \cup \widehat{C_i} \cup ... \cup C_k \right) \sqcup \mathcal{T}$ then by Theorem \ref{thm:intJ}, 
\[
J_0(C' \cup C_i) \geq J_0(C') + J_0(C_i) + 2 C' \cdot C_i + E,
\]
where $E$ is the number of elliptic Reeb orbits in $Y$ at which both $C'$ and $C_i$ have ends.  By intersection positivity and the fact that $C'$ and $C_i$ cannot have an irreducible component in common we have that $C' \cdot C_i \geq 0$, with equality if and only if $C'$ and $C_i$ are disjoint. 
If $E \geq 1$ we get an immediate contradiction to the existence of the genus zero thrice punctured curve $C_i$  because $J_0(C'\cup C_i)\leq1$ by (\ref{eq:J1})-(\ref{eq:J3}).  

Thus it remains to consider the possibilities for $C_i$ when $E=0$ and the curves $C'$ and $C_i$ are disjoint.  We previously established that $C_i$ must have an end on $\zh$.  By Proposition \ref{prop:J0} and Remark \ref{rem:J0sharp}, $C_i$ must have its remaining two ends on two different elliptic orbits.  Thus the options for $C'$ are quite limited. In particular, $C'$ can only have ends on the remaining elliptic orbit. 
The ECH index of unions, Theorem~\ref{thm:intI}, tells us $I(C') = 0$, thus $C'$ must be a trivial cylinder over this remaining elliptic orbit. 
 We now explain why these conditions on the asymptotics for $C_i$ and $C'$ cannot be satisfied if $\cur = C_i \sqcup C'$ belongs to $\M(\alpha,\beta,J)$ when $(\alpha,\beta)$ are one of the special pairs of Reeb currents in \eqref{eq:1pairs}.

By the partition conditions described in Remark \ref{rem:partb}, if there is a positive end at $\zb^B$ then $B=1$.  
If $B=1$, and $C_i$ has a positive end on $\zb$ then $J_0(C_i \cup C') = 0$ by \eqref{eq:pairsnogamma1} since $C'$ is trivial cylinder over a cover of either $\zp$ or $\zq$.  However, this is a contradiction to \eqref{eq:*}, thus $C_i$ cannot have a positive end on $\zb$.  The previous discussion on the asymptotics of $C'$ and $C_i$ when $E=0$ in conjunction with the special pair condition on the Reeb currents in \eqref{eq:1pairs} implies that $C_i$ cannot have a negative end on any cover of $\zb$.  Thus the ends of $C_i$ must be asymptotic to $\zh$, $\zp^P$, and $\zq^Q$, where $P\geq 1$, $Q \geq 1$ and $C'$ must be a trivial cylinder over $\zb^B$. But this is not achievable for a special pair of Reeb currents in \eqref{eq:1pairs} because we would need one of $Q=0$ or $P=0$.  

\medskip

Thus for every $i$, $C_i$ has exactly two punctures and $g(C_i) = 0$.
\end{proof}

We have the following computation of three cylindrical differential coefficients using a non-generic $S^1$-invariant $\lambda_{p,q,\varepsilon}$-compatible almost complex structure $\frak{J} = \frak{p}^*j_{\CP^1_{p,q}}$.   

\begin{proposition}\label{prop:pqM}
 If $\alpha$ is an admissible Reeb current for $(S^3,\lambda_{p,q,\varepsilon(L)})$ not admitting $\zh$, then for an $S^1$-invariant  $\lambda_{p,q,\varepsilon}$-compatible almost complex structure $\frak{J} = \frak{p}^*j_{\CP^1_{p,q}}$,
\begin{enumerate}[{\em(i)}]
\itemsep-.25em
\item $\#_2\M(\zh\alpha,\zq^q\alpha,\frak{J})=1$,
\item $\#_2\M(\zh\alpha,\zp^q\alpha,\frak{J})=1$, 
\item $\#_2\M(\zb\alpha,\zh\alpha,\frak{J})=0$,
\end{enumerate}
where $\#_2$ denotes the mod 2 count.
\end{proposition}

\begin{proof} The fact that the pairs appearing below have index difference one and equal degrees is proven in Lemma \ref{lem:ECHd}. Proposition \ref{prop:morsep} and Lemma \ref{lem:orbitseh} provide us with a Morse function $\mathpzc{H}_{p,q}$ on $\CP^1_{p,q}$ with
\begin{itemize}
\itemsep-.25em
\item Two index 0 critical points, denoted $x_p:=\fp(\zp)$ and $x_q:=\fp(\zq)$, with respective isotropy groups $\Z/p\Z$ and $\Z/q\Z$;
\item One index 1 critical point, denoted $x_h:=\fp(\zh)$, with no isotropy;
\item One index 2 critical point, denoted $x_b:=\fp(\zb)$, also with no isotropy.
\end{itemize}

The cyclic group $\Z/pq\Z$ acts on $Y$, the $T(p,q)$ fibration of $S^3$ as detailed in Proposition \ref{prop:SeifertInvts}, when it is regarded as a subgroup of $S^1$.  The quotient $Y'=Y/(\Z/pq\Z)$ is a genuine $S^1$ bundle over the same base $\CP^1_{p,q}$; let $\pi':Y \to Y'$ and $\frak{p}': Y' \to \CP^1_{p,q}$ denote the associated projection maps. Thus we may think of $Y'$ as a prequantization bundle, albeit with a finite symmetry. This allows us to invoke the setup of \cite[\S 5]{jo2} to count Fredholm index one cylinders in the relevant orbifold setting; see also the discussion about cylinder counts over orbifold Morse trajectories in \cite[\S 2.4]{leo}.   Note that by Lemma \ref{lem:d0g0}, each of the moduli spaces in (i)-(iii) consist solely of cylinders, and moreover, the curve counts in (i)-(iii) are given by the cylinders which are not trivial, namely those respectively in $\M(\zh,\zq^q,\frak{J})$, $\M(\zh,\zp^p,\frak{J}),$ and $\M(\zb, \zh,\frak{J}).$ 

To more precisely understand this correspondence, we observe the following. 
Let $z$ be a negative gradient flow line of  $\mathpzc{H}_{p,q}$  on $ \CP^1_{p,q}$ in $\M(x_b,x_h), \ \M(x_h,x_p),$ or $\M(x_h,x_q)$. Let $u_z'$ be the corresponding $\frak{J}':= (\frak{p}')^*j_{\CP^1_{p,q}}$-holomorphic cylinder in $\R \times Y'$ over $z$.  This correspondence is guaranteed by  \cite[Thm.~5.1, 5.5]{jo2}.  That the contact forms and contact structure on $Y$ induce an $S^1$-invariant transverse contact structure $\xi'$ on $Y'$ is guaranteed by \cite[Prop.~3.1]{lm}.  The latter reference also details the necessary local behavior around singular fibers.  

  Note that $\pi_* \frak{J} = \frak{J'}$ and $(\op{id} \times \pi) : (\R \times Y, \frak{J}) \to  (\R \times Y', \frak{J}')$ is holomorphic.  Thus any such holomorphic cylinder $u_z'$ lifts to a holomorphic map $u_z: S \to \R \times Y$, where $S$ is a suitable cover of $\R \times S^1$ satisfying $(\op{id} \times \pi) \circ u_z = u_z'$.  In fact, $S$ must be an unbranched $pq$-fold cover of $\R \times S^1$, since by construction, $\pi \circ \zb = (\zb')^{pq}$,   $\pi \circ \zh = (\zh')^{pq}$, $\pi \circ \zp^p = (\zp')^{pq}$, $\pi \circ \zq^q = (\zq')^{pq}$.  Here $\zb'$, $\zh'$, $\zp'$, and $\zq'$ are the corresponding embedded fiber Reeb orbits in $Y'$, which respectively project to $x_b$, $x_h$, $x_p$, and $x_q$.   Thus $u_z$ is in $\M(\zb,\zh), \ \M(\zh, \zp^p),$ or $\M(\zh,\zq^q)$, depending on the original choice of $z$.

In light of this correspondence between $S^1$-invariant cylinders and orbifold flow lines, the desired counts equal the mod 2 counts of negative gradient flow lines of $\mathpzc{H}_{p,q}$ from (i) $x_h$ to $x_q$, (ii) $x_h$ to $x_p$, and (iii) $x_b$ to $x_h$.  
The isotropy of $x_h$ and $x_b$ are both trivial, thus the orbifold Morse differentials as defined in \cite[Lem.~3.1, (3.1)]{morse-orbifold}, \cite[\S 1.4]{leo} are given by an unweighted count of negative gradient flow lines with ends as indicated in the list above.  This means that the isotropy of any of the negative gradient index one paths $z$ is 1. Thus all we need to do is compute this orbifold Morse differential.

Since orbifold Morse homology computes homology and we are counting mod 2, the generators $x_p$ and $x_q$ must be homologous; this proves (i) and (ii). There is only one index 2 generator, thus it must be closed, proving (iii). A visual is provided in Figure \ref{fig:morsepq}.

\end{proof}
 
 Next we show that these cylinder counts agree using any sufficiently nearby  $\lambda_{p,q,\varepsilon}$-compatible almost complex structures. In particular, this includes the generic ones since this space is contractible.
 
\begin{lemma}\label{lem:pqM}
 If $\alpha$ is an admissible Reeb current for $(S^3,\lambda_{p,q,\varepsilon(L)})$ not admitting $\zh$, then for any $\lambda_{p,q,\varepsilon}$-compatible almost complex structure $J$ which is sufficiently close to the $S^1$-invariant  $\lambda_{p,q,\varepsilon}$-compatible almost complex structure $\frak{J} = \frak{p}^*j_{\CP^1_{p,q}}$,
\begin{enumerate}[{\em(i)}]
\itemsep-.25em
\item $\#_2\M(\zh\alpha,\zq^q\alpha,{J})=1$,
\item $\#_2\M(\zh\alpha,\zp^q\alpha,{J})=1$, 
\item $\#_2\M(\zb\alpha,\zh\alpha,{J})=0$,
\end{enumerate}
where $\#_2$ denotes the mod 2 count.
\end{lemma}

\begin{proof}
The proof of this lemma is carried out in two parts.  In Step 1, we prove existence of the relevant cylinders for nearby generic  $\lambda_{p,q,\varepsilon}$-compatible almost complex structures using parametric regularity and automatic transversality.  In Step 2, we prove that the cylinder counts agree, which relies on Lemma \ref{lem:d0g0}, index computations, and action and energy reasons which we elucidate below.  

\medskip

\textbf{Step 1.} The outline of the argument is as follows. First we construct a parametric moduli space corresponding to a path of varying $\lambda_{p,q,\varepsilon}$-compatible almost complex structures $\{J_s\}_{s\in[0,1]}$, with the parameter value $s=0$ corresponding to $\mathfrak{J}$, e.g. $J_0 = \frak{J}$.\footnote{Please excuse our slight abuse of notation.}  Next, we observe that this moduli space is smooth at $(0,C_0)$ and has dimension $\ind(C_0)+1$.   Here $C_0$ is any curve in $\M(\zh,\zq^q,\frak{J})$, $\M(\zh,\zp^q,\frak{J})$, or $\M(\zb,\zh,\frak{J})$.  As a result, $(0,C_0)$ is not isolated with respect to the parameter value, and thus sufficiently small perturbations of the almost complex structure admit regular cylinders close to $C_0$.

Let $\{J_s\}_{s\in[0,1]}$ be a smooth path of almost complex structures on $\R\times S^3$ such that $J_s$ is $\lambda_{p,q,\varepsilon}$-compatible for each $s$, with $s=0$ fixed to be $\mathfrak{J}$. 
Consider the parametric moduli spaces of cylinders with one positive and one negative asymptotic ends on the respective Reeb orbit pair  $(\gamma_+,\gamma_-)$ which is of one of the forms
\begin{equation}\label{eq:lempairs}
 (\zb,\zh),  (\zh,\zp^p), (\zh,\zq^q).
\end{equation}
We define 
\[
\mathcal{M}(\gamma_+,\gamma_-,\{J_s\}):= \left\{ (s,C_s):\  C_s\text{ is a $J_s$-cylinder with asymptotics as in \eqref{eq:lempairs}}\right\}.
\] 
For each $s\in [0,1]$ we consider also the ``sliced" moduli space
\[
\mathcal{M}_s:= \mathcal{M}(\gamma_+,\gamma_-,J_s).
\]

One defines the parametric moduli space $\mathcal{M}(\gamma_+,\gamma_-,\{J_s\})$ and the associated notion of parametric regularity, in terms of an explicit description of the associated linearized operator $D\bar\partial_{\{J_s\}_{s\in [0,1]}}$ as explained in \cite[Remark 8.7]{wendl-sft}.  In particular, the linearized operator  $D\bar\partial_{\{J_s\}}(j,C,s)$ is the sum of $D\bar\partial_{J_s}(j,C)$ with an additional term defined on $T_sP$, where $P=[0,1]$ is the parameter space. 
When these operators are surjective, the implicit function theorem guarantees that the parametric moduli space $\mathcal{M}(\gamma_+,\gamma_-,\{J_s\})$ and the slice $\mathcal{M}_s$ are smooth manifolds near $(s,C_s)$ and $C_s$ respectively, of dimensions given by the Fredholm indices.

That the parametric moduli space $\mathcal{M}(\gamma_+,\gamma_-,\{J_s\})$ is a smooth manifold in an open neighborhood of the point $(0,C_0)$ follows from an instance of the fact that ``honest regularity implies parametric regularity". More concretely, by automatic transversality  cf. \cite{wendl}, \cite[Lem.~4.1]{dc},  the linearized operator $D\bar\partial_{\frak{J}}(j,C_0)$ is surjective at $C_0$.  To see this, first observe that the asymptotics for $C_0$ ensure that $C_0$ is somewhere injective and that
\[
\ind(C_0)=I(C_0) =1,
\]
hence $C_0$ is embedded by the ECH index inequality, Theorem \ref{thm:indexineq}.  The Fredholm index computation follows from the Conley-Zehnder index computation in Lemma \ref{lem:orbtrivCZp} and the ECH index computation is carried out Lemma \ref{lem:ECHd}.  The automatic transversality inequality for immersed curves, c.f. \cite[Lem. 4.1]{dc},
\[
2g(C_0) - 2 + h_+(C_0) < \ind (C_0)
\]
holds because $g(C_0)=0$ and $h_+(C_0) =1$, the latter being the number of positive hyperbolic orbits of $C_0$.  

Next, recall that the parametric operator $D\bar\partial_{\{J_s\}}(j,C,s)$ is the direct sum of $D\bar\partial_{J_s}(j,C)$ with an additional term and that the co-domain is the same for the parametric and non-parametric operators.  Thus the parametric operator must be surjective at $(0,C_0)$ as well.  

We conclude that  $\mathcal{M}(\gamma_+,\gamma_-,\{J_s\})$ is a smooth manifold near $(0,C_0)$ of dimension $$
    \dim(\ker(D\bar\partial_{\{J_s\}}(j,C_0,0))) = \dim (\ker(D\bar\partial_{\frak{J}}(j,C_0)))+1 = \ind(C_0)+1.
    $$
    In particular, there exists $s$ arbitrarily close to 0 such that the  slice $\mathcal{M}_s$ is not empty and there exists $C_s\in \mathcal{M}_s$ that is $C^\infty$ close to $C_0$. Automatic transversality again applies in an identical argument as for $C_0$ to establish that $C_s\in \mathcal{M}$ is regular for these $s$ close enough to 0. (In particular all such $J_s$-cylinders $C_s$ will be regular, without the need for a genericity assumption).

\medskip

\textbf{Step 2.} We conclude by showing that the the curve counts of the slices $\M_s$ agree with the $s=0$ slice $\mathcal{M}(\gamma_+,\gamma_-,\frak{J})$, which is in some sense a sort of uniqueness result.  Here we utilize an Arzel\`a-Ascoli argument in conjunction with action and index calculations.  

Let $s_k \to 0$ and $C_{s_k}$ be a $J_{s_k}$-holomorphic sequence of cylinders with one positive asymptotic at $\gamma_+$ and one negative asymptotic at $\gamma_-$ of one of the forms in \eqref{eq:lempairs}.  By Stokes' Theorem we have a uniform bound on the energies $E_{s_k}(C_{s_k})$, since the asymptotics are on fixed Reeb orbits for all curves in this sequence.  

Next, we invoke a compactness result, so as to argue that $C_{s_k}$ has a subsequence convergent to a finite energy $\frak{J}$-holomorphic building $C_\infty$ consisting only of cylinders.  

\medskip
First, we note that by Stokes' Theorem, for any contact form $\lambda$, the use of a $\lambda$-compatible $J$ implies that the $d\lambda$-energy of a $J$-holomorphic curve $C$ must be nonnegative and it is strictly positive unless $C$ is a trivial cylinder, cf. \cite[Ex.~1.23]{wendl-sft}.  

For our contact form $\lambda_{p,q,\varepsilon}$ we have that strict positivity of $\int_{\dot{\Sigma}}C^*d\lambda_{p,q,\varepsilon}$ implies that the Fredholm index of any somewhere injective cylinder (with one positive and one negative end) must be positive because we can enumerate the orbits, their actions, and their Conley-Zehnder indices.  (Here $\dot{\Sigma}$ is the domain of the curve $C$.) Automatic transversality applies for these cylinders (as well as their low index iterates), cf.~\cite[\S 4.2]{dc} and the argument in Step 1.  Any unbranched cover of a cylinder with positive Fredholm index must also have positive Fredholm index by \cite[Lem.~2.5(a)]{dc}.  So the only Fredholm index zero cylinders must be trivial cylinders and there are no negative Fredholm index cylinders.

\medskip
{Thus by the preceding paragraph,} it remains to check that there are not any genus zero curves with nonpositive Fredholm index that could appear in the $\frak{J}$-holomorphic building $C_\infty$, which are not trivial cylinders. 
\medskip
Note that the Conley-Zehnder index computation in Lemma \ref{lem:orbtrivCZp} tells us that with respect to the ``global orbibundle trivialization," which has vanishing first Chern number for all curves, the Conley-Zehnder index is at least three for  for all Reeb orbits, e.g. the contact forms $\lambda_{p,q,\varepsilon}$ are all dynamically convex.  Since we are not assuming that our almost complex structures are generic, we do not a priori know that any somewhere injective genus zero curve that can appear in $C_\infty$, which is not a trivial cylinder, must have positive Fredholm index, so the index calculations of \cite[\S 2]{dc} cannot be invoked off the shelf.  

At our disposal are the action computation in Lemma \ref{lem:efromL}, recalled most recently in \eqref{e:lemM}, that the asymptotics of {each of the cylinders in the sequence $C_{s_k}$}  are of one of the special forms in \eqref{eq:lempairs}, and that $d\lambda_{p,q,\varepsilon}$-energy of any nontrivial level must be nonnegative.  This is sufficient to guarantee that there are not any genus zero curves with nonpositive Fredholm index that could appear in the $\frak{J}$-holomorphic building $C_\infty$, which are not trivial cylinders, as follows. 

In particular, if there are any non-cylindrical genus zero levels, at least one must have exactly one positive end and there will be insufficient action for there to be multiple negative ends if it is to appear in  $C_\infty$.  This follows from the fact that the largest action of a positive end  is that of  $\zb$ or $\zh$ and 
\[
\begin{split}
\mathcal{A}(\zb) = & (1+\varepsilon \fp^*\mathpzc{H}_{p,q}(x_{\text{max}}))\\
\mathcal{A}(\zh) = & (1+\varepsilon \fp^*\mathpzc{H}_{p,q}(x_{\text{saddle}}))\\
\mathcal{A}(\zp^p) = & (1+\varepsilon \fp^*\mathpzc{H}_{p,q}(x_{\text{min for } \zp})) \\
\mathcal{A}(\zq^q) = & (1+\varepsilon \fp^*\mathpzc{H}_{p,q}(x_{\text{min for } \zq})).
\end{split}
\]
Recall as in the construction of the Morse function $\mathpzc{H}_{p,q}$ in Proposition \ref{prop:morsep} that  $|\mathpzc{H}_{p,q}|$ is $C^2$ close to one.  {(To better illuminate this action argument, take the case where the asymptotics are $(\gamma_+,\gamma_-) = (\zh, \zp^p)$.  There is sufficient action for a curve with one positive end on $\zh$ and $p$ negative ends on $\zp$ to exist.  However, this curve could not appear as a level in $C_\infty$ because there must be a negative end on $\zp^p$.)}

\medskip

Because the Fredholm index of the building $C_\infty$ must be one, the building can have at most one nontrivial level, $C_\infty$, which is Fredholm regular by the automatic transversality argument explained in Step 1, and must be one of the respective $\frak{J}$-holomorphic cylinders arising as the lift of the respective index difference one orbifold Morse trajectory as constructed in Proposition \ref{prop:pqM}.  Thus $C_{s_k} \to C_\infty$, hence the curve counts must agree for $s_k$ sufficiently close to 0.


\medskip

Finally, the rest of the result follows from Proposition \ref{prop:pqM}.  In particular, Lemma \ref{lem:d0g0} tells us that the curve counts in (i)-(iii) are given by the cylinders which are not trivial, namely those respectively in $\M(\zh,\zq^q,{J})$, $\M(\zh,\zp^p,{J}),$ and $\M(\zb, \zh,{J})$.  Step 1 and 2 allow us to know that the curve counts agree with those of $\M(\zh,\zq^q,\frak{J})$, $\M(\zh,\zp^p,\frak{J}),$ and $\M(\zb, \zh,\frak{J}),$ which were computed in Proposition \ref{prop:pqM}.\end{proof}

\begin{remark}\label{rmk:comp} Crucially, Proposition \ref{prop:pqM} and Lemma \ref{lem:pqM} do not rule out any other differentials. It is not immediately obvious from the index formula or the ordinary homology of the total space -- because it is just $S^3$ -- that these are indeed the only differentials or index one moduli spaces. To establish that these are the only differentials and index one moduli spaces requires the roundabout method of appealing to the known computation of the ECH of $(S^3,\xi_{std})$, which we complete in the proof of Proposition \ref{cor:nodiffp} in \S\ref{sss:finalcomp}.
\end{remark}

\subsection{Combinatorial ECH index calculations}\label{ss:combinatorics}

We now prove Theorem \ref{thm:pqI} and collect several combinatorial facts about the ECH index, which will be used to compute spectral invariants in \S\ref{s:spectral}.

\begin{proof}[Proof of Theorem \ref{thm:pqI}]
 Recall that the goal of Theorem \ref{thm:pqI} is to show that for any $L$ and Reeb current $\zb^B\zh^H\zp^P\zq^Q$, we have
\[
I(\zb^B\zh^H\zp^P\zq^Q)=-(P-Q)^2+2qP(H+B)+2pQ(H+B)+(pq-p-q)(H^2+B^2)+2HBpq+\mathbf{CZ},
\]
where $\mathbf{CZ}$ is the Conley-Zehnder term
\[
\mathbf{CZ}=CZ^I_{orb}(\zb^B)+2(p+q)H+CZ^I_{orb}(\zp^P)+CZ^I_{orb}(\zq^Q),
\]
and
\begin{align*}
CZ^I_{orb}(\zb^B)&=\sum_{i=1}^B\left(2(p+q)i+1\right)=(p+q)B^2+(p+q+1)B;
\\CZ^I_{orb}(\zp^P)&=\sum_{i=1}^P2\left\lfloor\left(\frac{p+q}{p}-\delta_{\zp,L}\right)i\right\rfloor+1;
\\CZ^I_{orb}(\zq^Q)&=\sum_{i=1}^Q2\left\lfloor\left(\frac{p+q}{q}-\delta_{\zq,L}\right)i\right\rfloor+1.
\end{align*}
Lemma \ref{lem:relfirstCherncalcp} computes the contributions from $c_\tau$ and the results in \S\ref{ss:Q} compute the contributions from the relative intersection pairing, see Lemmas \ref{lem:pzero}, \ref{lem:pQb}, \ref{lem:QpQq}, \ref{lem:pqcrossQ}, and \ref{lem:Qh}. The necessary monodromy angles from which the Conley-Zehnder contributions to the ECH index follow are computed in Lemma \ref{lem:orbtrivCZp}. 
\end{proof}

We show that the generators in Proposition \ref{prop:pqM} appear in the correct degrees and indices.

\begin{lemma}\label{lem:ECHd} Let $\alpha$ be any Reeb current with $H=0$.
\begin{enumerate}[{\em(i)}]
\itemsep-.25em
\item The degrees of $\zq^q\alpha, \zp^p\alpha, \zh\alpha$, and $\zb\alpha$ are all equal.
\item We have $I(\zh\alpha)=I(\zq^q\alpha)+1=I(\zp^p\alpha)+1.$
\item Furthermore, $
I(\zb\alpha)=I(\zh\alpha)+1.$
\end{enumerate}
\end{lemma}
\begin{proof} The first conclusion is immediate from the definition of degree, Definition \ref{def:degreep}. 

For the second and third conclusions, let $\alpha=\zb^B\zp^P\zq^Q$. We have
\begin{align}
I(\zq^q\alpha)&=I(\alpha)+2(P-Q)q-q^2+2pqB+\sum_{i=Q+1}^{Q+q}CZ_{orb}(\zq^i);\nonumber
\\I(\zp^p\alpha)&=I(\alpha)-2(P-Q)p-p^2+2pqB+\sum_{i=P+1}^{P+p}CZ_{orb}(\zp^i);\nonumber
\\I(\zh\alpha)&=I(\alpha)+2qP+2pQ+(pq-p-q)+2Bpq+2(p+q).\label{eqn:halph}
\end{align}
We simplify to establish (ii):
\begin{align*}
&I(\zh\alpha)=I(\zq^q\alpha)+1
\\\iff&2qP+2pQ+pq+2Bpq+p+q=2qP-2qQ-q^2+2pqB+\sum_{i=Q+1}^{Q+q}CZ_{orb}(\zq^i)+1
\\\iff&(p+q)(2Q+q)+p-1=2\sum_{i=Q+1}^{Q+q}\left\lfloor\left(1+\frac{p}{q}-\delta_{\zq,L}\right)i\right\rfloor
\\\iff&(p-1)(q+1)=2\sum_{j=1}^{q}\left\lfloor\left(\frac{p}{q}-\delta_{\zq,L}\right)j\right\rfloor.
\end{align*}
The sum on the right hand side is the number of points in the first quadrant under the line of slope $p/q-\delta_{\zq,L}$, exclusive of those on the $x$-axis (there will be no points on the line itself as we assume $\delta_{\zq,L}$ is irrational). Each summand (e.g. for a fixed $j$) is represented by the points on a single vertical line. If we include the $q+1$ points on the $x$-axis, this number of points is exactly half of the number of points in the $p\times q$ rectangle. Since we're multiplying by two, it suffices to show that the number of points in this rectangle, minus $2(q+1)$, is equal to the left hand side. That is elementary:
\[
\mbox{\# of points in $p\times q$ rectangle}-2(q+1)=(p+1)(q+1)-2(q+1)=\mbox{ right hand side}.
\]
The proof that $I(\zh\alpha)=I(\zp^p\alpha)+1$ is identical, with the roles of $p$ and $q$ swapped and the roles of $P$ and $Q$ swapped.

Finally we prove (iii). We compute
\begin{equation}
I(\zb\alpha)=I(\alpha)+2qP+2pQ+(pq-p-q)(2B+1)+CZ_{orb}(b^{B+1}).\label{eqn:balph}
\end{equation}
Using (\ref{eqn:halph}) and (\ref{eqn:balph}), we may simplify
\[
I(\zb\alpha)=I(\zh\alpha)+1 \iff CZ_{orb}(b^{B+1})=2(p+q)(B+1)+1,
\]
which holds by Lemma \ref{lem:orbtrivCZp}.
\end{proof}

\subsection{Computation of the ECH chain complex}\label{sss:finalcomp}

In this section we complete the computation of the ECH chain complexes $ECC_*^{L(\varepsilon)}(S^3,\lambda_{p,q,\varepsilon},J)$ in terms of generators and differentials.  The coefficients
\begin{equation}\label{eqn:diffp}
\langle\partial(\zh\alpha),\zq^q\alpha\rangle=\langle\partial(\zh\alpha),\zp^p\alpha\rangle=1 \mbox{ and }\langle\partial(\zb\alpha),\zh\alpha\rangle=0
\end{equation}
were computed in Proposition \ref{prop:pqM}, which relied on the now established Lemma \ref{lem:ECHd}.  {These coefficient counts hold for any $J$ in a continuous one parameter family, which is sufficiently close to the $S^1$-invariant  $\lambda_{p,q,\varepsilon}$-compatible almost complex structure $\frak{J} = \frak{p}^*j_{\CP^1_{p,q}}$, as shown in Lemma \ref{lem:pqM}. Because the space of $\lambda_{p,q,\varepsilon}$-compatible almost complex structure is contractible, this includes some nearby $J$ which are themselves generic. It remains to show that all other differential coefficients are zero: }

\begin{proposition}\label{cor:nodiffp} Let $J$ be a generic $\lambda_{p,q,\varepsilon}$-compatible almost complex structure. So long as $*$ is small enough relative to $L(\varepsilon)$ as in Lemma \ref{lem:orbitseh}, all moduli spaces determining differentials in $ECC^{L(\varepsilon)}_*(S^3,\lambda_{p,q,\varepsilon},J)$ have a mod 2 count of zero besides those in Proposition \ref{prop:pqM} (i) and (ii).
\end{proposition}

To prove Proposition \ref{cor:nodiffp}, we utilize action filtered Morse-Bott direct limit arguments in conjunction with the (well-known) computation (and invariance) of the ECH of the standard tight 3-sphere.\footnote{Means of recovering $ECH_*(S^3, \xi_{std})$ include the irrational ellipsoid in \cite[\S3.7]{lecture}, the prequantization bundle over $S^2$ with Euler class -1 in \cite[\S7.2.2]{preech}, or the exciting index calculations for nonstandard families of contact forms in \cite[\S 5.3]{kech}.}  In particular,
\begin{equation}\label{eq:2}
ECH_*(S^3,\xi_{std}) =\begin{cases}\Z/2&\text{ if }* \in 2\Z_{\geq 0}
\\0&\text{else}.
\end{cases}
\end{equation}
We start by relating the direct limit of the homologies with respect to $\lambda_{p,q,\varepsilon}$ to the embedded contact homology of $(S^3,\xi_{std})$ in \eqref{eq:2} in Proposition \ref{prop:DL}.  The proof is a virtual repeat of our arguments in \cite[Prop.~3.2, Thm.~7.1]{preech}, and hence omitted.

\begin{proposition}\label{prop:DL} For $\varepsilon>\varepsilon'$, there is an exact symplectic cobordism from $(S^3,\lambda_{p,q,\varepsilon})$ to $(S^3,\lambda_{p,q,\varepsilon'})$, and the compositions of the inclusion induced maps $\iota^{L(\varepsilon),L(\varepsilon')}$ and the cobordism maps $\Phi^L$ defined in \cite[Thm.~2.17]{preech} are the canonical bijection on generators.  For generic $\lambda_{p,q,\varepsilon}$-compatible $J$,
\[
\lim_{\varepsilon\to0}ECH^{L(\varepsilon)}_*(S^3,\lambda_{p,q,\varepsilon},J)=ECH_*(S^3,\xi_{std}).
\]
\end{proposition}

We use Proposition \ref{prop:DL} to understand the ECH of $(S^3,\lambda_{p,q,\varepsilon})$ in more detail.\footnote{While these two results are similar to \cite[Prop.~5.7]{kech}  and its \cite[Cor~5.8]{kech}, the logic is reversed.} In particular, we obtain the following as an immediate consequence of the fact that the maps in the directed system in Proposition \ref{prop:DL} are the canonical bijections.  (That the action filtered embedded contact homology does not depend on the choice of $J$ is established in \cite{cc2}.)

\begin{corollary}\label{prop:bijectionp}
For any generic $\lambda_{p,q,\varepsilon}$-compatible  almost complex structure $J$, we have
\[
ECH_*^{L(\varepsilon)}(S^3,\lambda_{p,q,\varepsilon},J)=\begin{cases}
\Z/2&\text{ if }* \in 2\Z_{\geq 0}\\0&\text{ otherwise},
\end{cases}
\]
so long as $*$ is small enough relative to $L(\varepsilon)$ as in Lemma \ref{lem:orbitseh}.
\end{corollary}

We can now complete the proof of Proposition \ref{cor:nodiffp}. 

\begin{proof}[Proof of Proposition \ref{cor:nodiffp}]
The logic of the proof is as follows: we will identify each group $ECH^{L(\varepsilon)}_{2k}(S^3,\lambda_{p,q,\varepsilon},J)$ with $\Z/2$ generated by a single Reeb current. Our identification will provide a bijection between $2\Z_{\geq0}$ and the set of action filtered ECH generators with $H=0$ up to the equivalence $\zp^p\alpha\sim \zq^q\alpha$. However, this equivalence is precisely that provided by the differentials (\ref{eqn:diffp}) which we have previously computed in Proposition \ref{prop:pqM} and Lemma \ref{lem:pqM}. Therefore there can be no other differentials lest $ECH^{L(\varepsilon)}_{2k}(S^3,\lambda_{p,q,\varepsilon},J)$ not describe all of the ECH of $S^3$ as it ought to based on Proposition \ref{prop:DL}. 

The fact that the map $I$ from the set of ECH generators with $H=0$ to $2\Z_{\geq0}$ is well-defined up to $\zp^p\sim \zq^q$ follows from Lemma \ref{lem:ECHd}(ii). Further, it is surjective by Corollary \ref{prop:bijectionp}. It remains to show that it is injective, that is, if $\alpha\not\sim\beta$ by replacing $\zp^p$s with $\zq^q$s or vice versa, then $I(\alpha)\neq I(\beta)$.

We do this by defining an inverse $\mathbf{J}$ of $I$ and showing that $\mathbf{J}\circ I=id$. Let $T_k$ be the triangle in the first quadrant below the lowest line of slope $-p/q$ which contains at least $k+1$ points of the lattice $\Z^2$ (including those points on the axes and the line), and let $L_k$ be this line. Label the lattice points on $L_k$ as $(Q_i,P_i)$, starting with $(Q_0,P_0)$ being the closest to the $x$-axis and with $P_{i+1}/Q_{i+1}>P_i/Q_i$. Let $k'+1$ be the number of lattice points in $T_k$ excluding those on $L_k$. Define $\mathbf{J}$ by
\[
\mathbf{J}:2(k'+i+1)\mapsto\zb^i\zp^{P_0}\zq^{Q_0-iq}.
\]
It is defined for $i\geq0$ and at most the number of lattice points on $L_k$ minus one.

To show $\mathbf{J}\circ I=id$, let $\alpha=\zb^B\zp^P\zq^Q$; we may assume $P<p$ by replacing $\zp^p$s with $\zq^q$s. \\

\textbf{Claim:} the index $I(\alpha)=2k$, where $L_k$ passes through the point $(Q,pB+P)$ and there are precisely $B$ points on $L_k$ with smaller $y$-coordinate (i.e., $(Q,pB+P)=(Q_i,P_i)$ where $k=k'+i+1$). After showing the claim, we are done because $\mathbf{J}(2k)=\alpha$. \\

\textbf{Proof of claim:} By Lemma \ref{lem:ECHd}(ii, iii), it suffices to show that the index of $\zp^P\zq^Q$ with $P<p$ equals $2(k'+1)$. Divide $T_k\setminus L_k\cup\{(Q,P)\}$ into three regions: the triangle $T_Q$ strictly above the horizontal line $y=Q$, the triangle $T_P$ strictly to the right of the vertical line $x=P$, and the rectangle $R$ consisting of points with $0\leq x\leq P$ and $0\leq y\leq Q$. We start by counting the points in $T_Q, T_P$, and $R$.
\begin{itemize}
\item By splitting the count into a sum over each column $x=Q-1,\dots,x=0$, the triangle $T_Q$ contains $\mathcal{L}_Q:=\sum_{i=1}^Q\left\lfloor\left(\frac{p}{q}-\delta_{\zq,L}\right)i\right\rfloor$ lattice points.
\item By splitting the count into a sum over each row $y=P-1,\dots,y=0$, the triangle $T_P$ contains $\mathcal{L}_P:=\sum_{i=1}^P\left\lfloor\left(\frac{q}{p}-\delta_{\zp,L}\right)i\right\rfloor$ lattice points.
\item The rectangle $R$ contains $\mathcal{L}_R:=(P+1)(Q+1)$ lattice points.
\end{itemize}

We simplify the formula for $I(\zp^P\zq^Q)$ from Theorem \ref{thm:pqI}:
\begin{align*}
I(\zp^P\zq^Q)&=-(P-Q)^2+CZ^I_{orb}(\zp^P)+CZ^I_{orb}(\zq^Q)
\\&=2(P+1)(Q+1)-2+2\sum_{i=1}^P\left\lfloor\left(\frac{q}{p}-\delta_{\zp,L}\right)i\right\rfloor+2\sum_{i=1}^Q\left\lfloor\left(\frac{p}{q}-\delta_{\zq,L}\right)i\right\rfloor
\\&=2(\mathcal{L}_R+\mathcal{L}_P+\mathcal{L}_Q-1)
\\&=2(k'+1),
\end{align*}
because $\mathcal{L}_R+\mathcal{L}_P+\mathcal{L}_Q=k'+2$ as it counts the point $(Q,P)$, which is not in $T_k\setminus L_k$.

\end{proof}

\subsection{Comparison with a convex toric perturbation}\label{ss:toric}

Proposition \ref{prop:pqM} and the proof of Proposition \ref{cor:nodiffp} suggests that it is possible to compute spectral invariants for $\lambda_{p,q}$ using the following combinatorial model of the ECH chain complex, inspired by \cite{T3, intoconcave}, so long as the index is low enough with respect to the action as in Lemma \ref{lem:orbitseh}:

\begin{figure}[h]
\begin{center}
\begin{overpic}[width=\textwidth, unit=1.75mm]{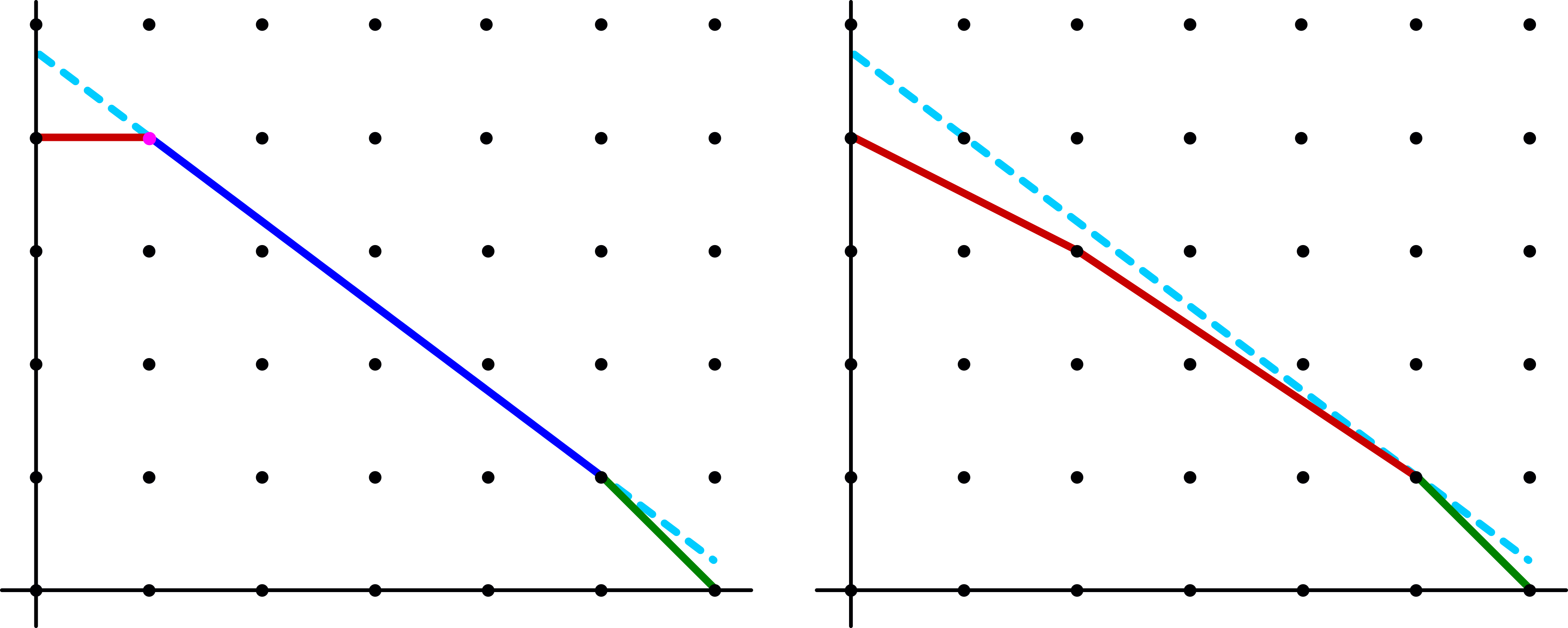}
\put(25,19){`$h$' label}
\end{overpic}
\end{center}
\caption{
{On the left in red, blue, and green we have drawn the curve $\Lambda$ representing the generator $\zh\zp\zq$ when $(p,q)=(3,4)$. On the right in red and green is the curve $\Lambda_\zq$ obtained by rounding the pink corner at $(1,4)$, which represents $\zp\zq^5$ and appears in $\partial\Lambda$. (The reader may compute their ECH indices by counting lattice points, and compare the result with Table \ref{table:genp} (a).) In both figures the dashed light blue line has slope $-3/4$.}
}
\label{fig:rounding}
\end{figure}

\begin{description}
\itemsep-.25em
\item[Generators:] Piecewise-linear curves $\Lambda:[0,a]\to\R^2_{\geq0}$ whose nonsmooth points are at $\Z^2$ lattice points and which can be divided into three pieces:

\begin{itemize}
\itemsep-.25em
\item the upper boundary of the convex hull of the points between the $y$-axis (inclusive), the vertical line $x=Q$ (inclusive), and the line through the point $(Q,P)$ of slope $-p/q$ (exclusive of all points except $(Q,P)$ itself),
\item the line segment of slope $-p/q$ from $(Q,P)$ to $(Q,P)+m(q,-p)$, labeled either `$e$' or `$h$', and
\item the top boundary of the convex hull of the points between the $x$-axis (inclusive), the vertical line $x=Q+mq$ (inclusive), and the line through the point $(Q,P)+m(q,-p)$ of slope $-p/q$ (exclusive of all points except $(Q+mq,P-mp)$ itself).
\end{itemize}
Such a curve $\Lambda$ corresponds to the generator
\[
\zb^m\zp^{P-mp}\zq^Q
\]
if the line segment of slope $-p/q$ is labeled `$e$,' otherwise it corresponds to the generator
\[
\zb^{m-1}\zh\zp^{P-mp}\zq^Q
\]
if the line segment is labeled `$h$.' Notice that every generator may be so described.
\emph{\textbf{Warning:} These `$e$' and `$h$' labels do not correspond to the $\lambda_{2,p,\varepsilon}$ generators from \cite{kech}.}

{Compare this definition of the differential to the partition conditions for negative ends described in Remark \ref{rmk:p-}. Up to the action of $AGL(2,\Z)$ (affine linear transformations of the $\Z^2$ lattice), the first and third pieces of a path $\Lambda$ correspond to a path $\Lambda^+$ determining the ``positive partitions" $P^+_\theta(m):=P^-_{-\theta}(m)$ with $\theta=\frac{p}{q}-\delta_{\zq,L}, \ \ \frac{q}{p}-\delta_{\zp,L}$ and $m=q, \ p$, respectively.}\footnote{{We make a point to note that this is confusing, because curves with positive ends on covers of $\zq$ or $\zp$ never appear in the differential. One would expect that the partitions determined by the curves counted by the ECH differential, namely, the negative partitions, would make an appearance.}}
\item[]
\item[ECH index:] The ECH index of $\Lambda$ equals $2(\mathcal{L}(\Lambda)-1)-h(\Lambda)$, where $\mathcal{L}(\Lambda)$ is the number of lattice points in the region bounded between $\Lambda$ and the axes, inclusive, and $h(\Lambda)=0$ if the line segment of slope $-p/q$ is labeled `$e$,' otherwise $h(\Lambda)=1$. {This can be proved using the argument in the proof of the Claim in the proof of Proposition \ref{cor:nodiffp} in \S\ref{sss:finalcomp}.}
\item[]
\item[Differential:] If there is no segment of slope $-p/q$ or it is not labeled `$h$,' then $\partial\Lambda=0$. Otherwise, the differential of $\Lambda$ is the sum $\Lambda_\zq+\Lambda_\zp$, where
\begin{itemize}
\itemsep-.25em
\item $\Lambda_\zq$ agrees with $\Lambda$ when $x\geq Q+q$, but when $0\leq x<Q+q$, the role of $(Q,P)$ is played by $(Q+q,P-p)$,
\item $\Lambda_\zp$ agrees with $\Lambda$ when $x\leq P-(m-1)p$, but when $x>P-(m-1)p$, the role of $(Q,P)+m(q,-p)$ is played by $(Q,P)+(m-1)(q,-p)$, and
\item if either $\Lambda_\zq$ or $\Lambda_\zp$ still contains a segment of slope $-p/q$ then it is labeled `$e$.'
\end{itemize}
That is, $\Lambda_\zq$ is the convex hull of the set of points between $\Lambda$ and the axes, inclusive, with $(Q,P)$ removed, and $\Lambda_\zp$ is similar but with $(Q+mq,P-mp)$ removed. {An example appears in Figure \ref{fig:rounding}, where $\Lambda_\zq$ has no segment of slope $-p/q$}.

Note that it is immediate from the definition of $\partial$ that $\partial^2=0$.
\item[]
\item[Action:] By Lemma \ref{lem:efromL}, the action of a generator $\zb^B\zh^H\zp^P\zq^Q$ is $B+H+P/p+Q/q$, which equals the value that the function $x/q+y/p$ takes at any point along the middle, slope $-p/q$, section of the corresponding $\Lambda$ (even if this section consists only of a single point). {The level sets of this function are the lines $L_k$ described in the proof of Proposition \ref{cor:nodiffp} in \S\ref{sss:finalcomp}. See Figure \ref{fig:ck}.}
\item[]
\item[Knot filtration:] Similarly, the knot filtration of a generator $\zb^B\zh^H\zp^P\zq^Q$ equals the value of the function $px+qy$ on the part of $\Lambda$ with slope $-p/q$.\footnote{This is actually only true in the limit $\varepsilon\to0$ as the perturbation recovers the true degenerate contact form. When $\varepsilon>0$, we need to add $\delta B$. Note that fact that the action and knot filtration values are equal up to constant multiplication is only an artifact of the fact that $(S^3,\lambda_{p,q})$ is strictly contactomorphic to an ellipsoid; when generalizing to other toric domains, this will not be the case, as the knot filtration will remain a function of the form $px+qy$ for some $(x,y)$ on $\Lambda$ while the action function becomes much more complex; see \cite[Def.1.13]{beyond}.}
\end{description}

\begin{figure}
     \centering
     \begin{subfigure}[b]{0.32\textwidth}
         \centering
         \includegraphics[width=\textwidth]{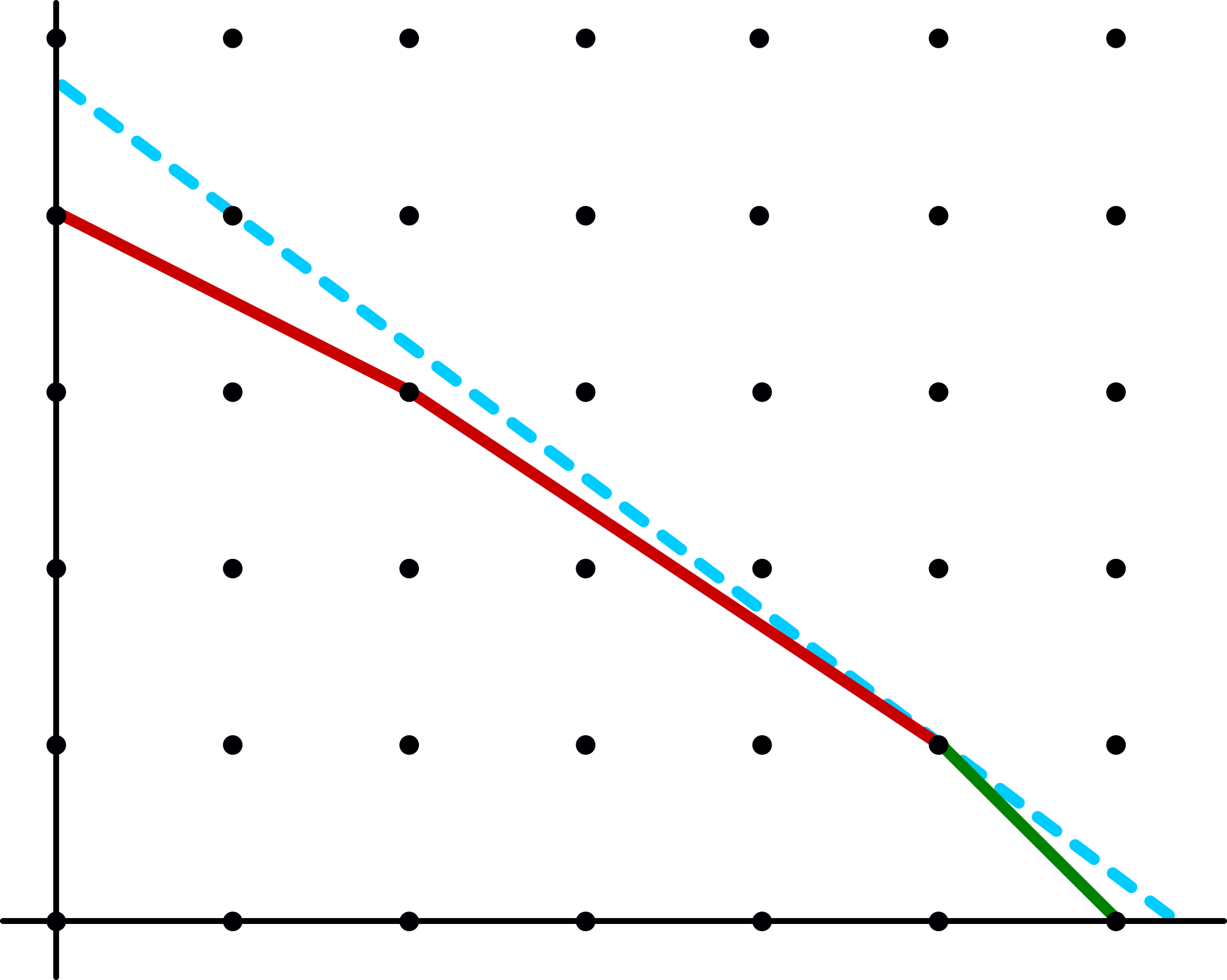}
         \caption{$\zp\zq^5$}
         \label{fig:q5}
     \end{subfigure}
     \hfill
     \begin{subfigure}[b]{0.32\textwidth}
         \centering
         \includegraphics[width=\textwidth]{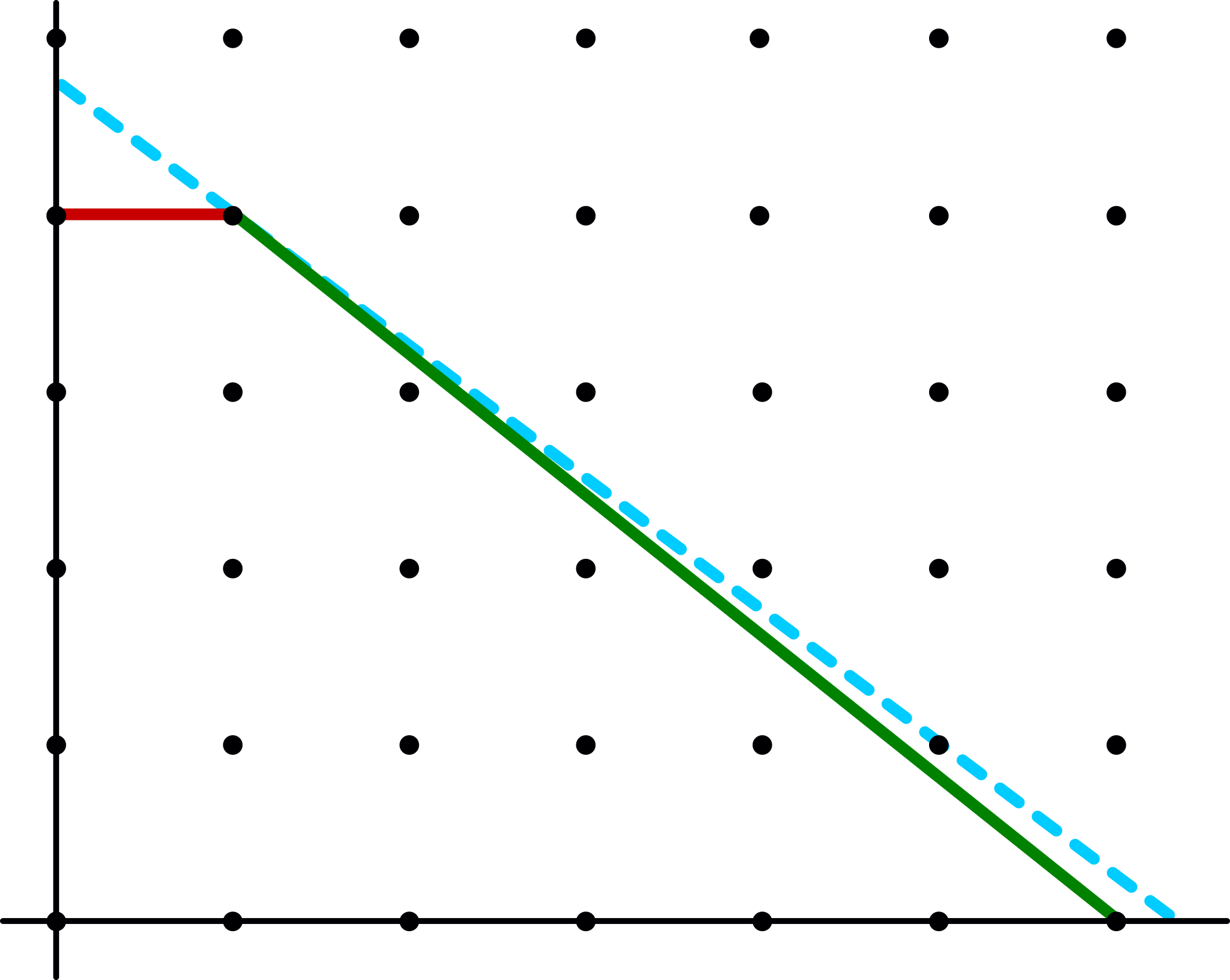}
         \caption{$\zp^4\zq$}
         \label{fig:p4}
     \end{subfigure}
     \hfill
     \begin{subfigure}[b]{0.32\textwidth}
         \centering
         \includegraphics[width=\textwidth]{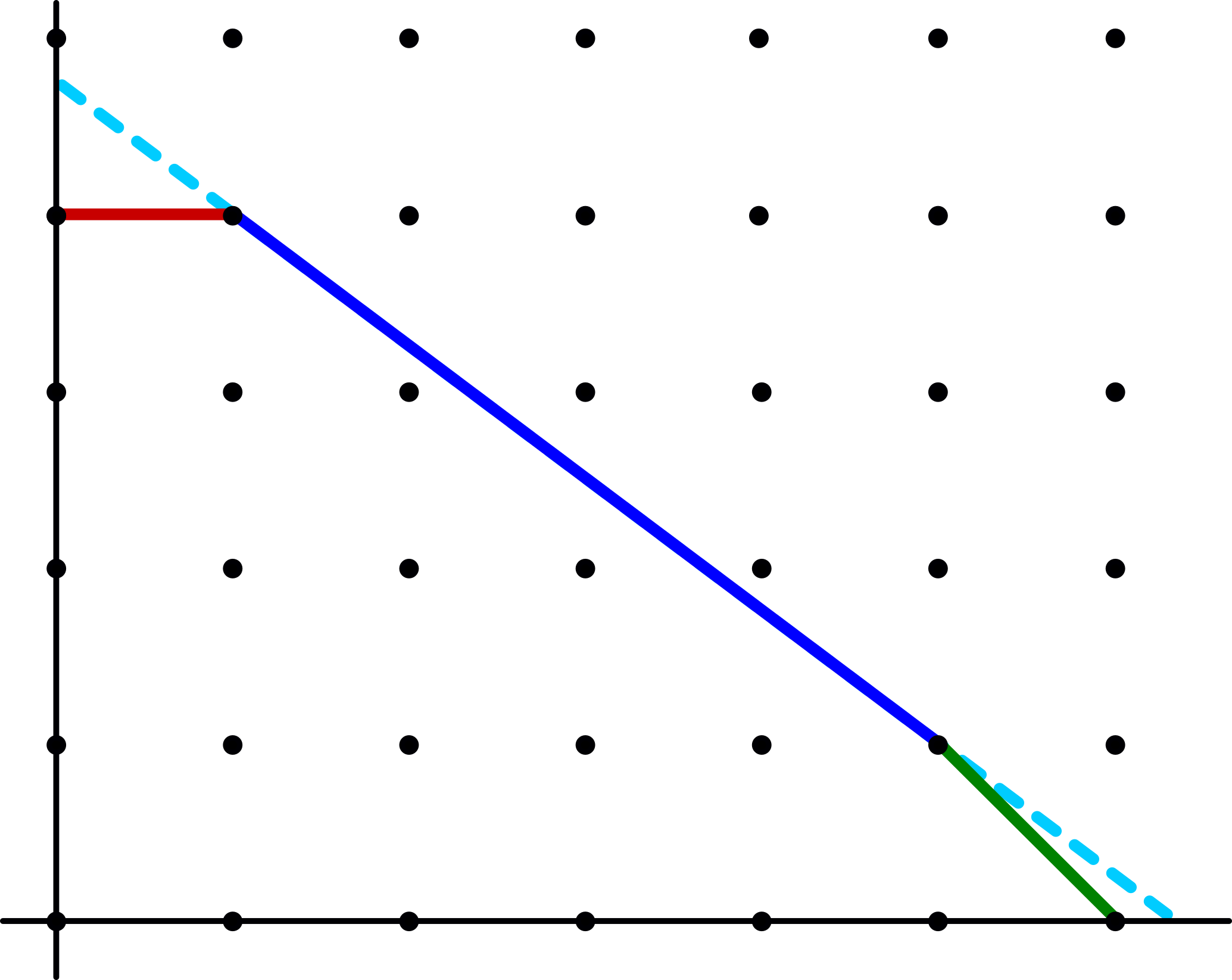}
         \caption{$\zh\zp\zq$ and $\zb\zp\zq$}
         \label{fig:bpq}
     \end{subfigure}
        \caption{{This figure shows the (dotted light blue) line determining the action and knot filtration value of the generators $\zp\zq^5, \zp^4\zq, \zh\zp\zq$, and $\zb\zp\zq$ corresponding to the curves illustrated in (a)-(c), respectively. (Note that the curve in (c) represents both $\zb\zp\zq$ and $\zh\zp\zq$ as its middle blue segment is unlabeled.)}}
        \label{fig:ck}
\end{figure}

The above description is enough to compute the knot filtration on ECH. Each generator without any `$h$' labels is closed. All such generators are homologous if they share the same index, because given two such generators $\Lambda_1$ and $\Lambda_2$ they will only differ by the number of consecutive points along the line of slope $-p/q$ they contain. There is a linear combination of generators with `$h$' labels whose differential equals $\Lambda_1-\Lambda_2$: start with $\Lambda_2$ and include the next point along the line of slope $-p/q$ in the direction of the points contained in $\Lambda_1$, then repeat. Moreover, homologous generators abut the same line of slope $-p/q$, so they have the same knot filtration value. As described in \cite{letternw}, it is also possible to compute the knot filtration using a larger combinatorial chain complex, which is a convex analogue to the complex described in \cite{intoconcave}, and that it is possible to make $\delta<0$ using the concave chain complex partially described (without its differential) in \cite{intoconcave}.

The relative simplicity of the above description makes it quite tempting. The main strength lies in the ease of computing the ECH index, since this would rely on well-studied results. However, the proof that the differential is as claimed has not been completed, and would require a detailed understanding and adaptation of Taubes' pseudoholomorphic (punctured) beasts in $\R \times (S^1 \times S^2)$  \cite{letter, beasts, beasts2, beasts3}. 

The problem is the following. In \cite{T3, choi, yao1, yao}, a ``rounding corners" differential similar to the one described above does agree with the ECH differential on manifolds obtained from $T^2\times[0,1]$ by collapsing circles in $T^2\times\{0\}$ and $T^2\times\{1\}$. (This procedure is called a \textit{contact cut}, see \cite{lerman}, and when those circles represent the standard basis of $H_1(T^2;\Z)$, we obtain $S^3$.) However, there is no proof as yet in the literature that the combinatorial rounding corners differential does agree with the ECH differential when the Reeb currents involved incorporate the elliptic orbits now appearing as the images of the two collapsed tori. (For us, these are iterates of $\zp$ and $\zq$; in general, if the manifold is $S^3$, it can be realized as embedded in $\C^2$, and these orbits are the Hopf link consisting of the intersections of $S^3$ with the complex axes.) 
Our chain complex described above \emph{only} has such ``virtual" combinatorial differentials. One can think of our present work as providing a proof of this agreement between differentials in the case of integral ellipsoids.

While the main advantage of our method when $p\neq2$ is in computing the differential, there are other advantages as well. Our prequantization orbibundle perspective will generalize to all contact forms adapted to periodic, positive fractional Dehn twist coefficient open books. If $\lambda$ is toric, then $Y$ must be diffeomorphic to $S^3, S^1\times S^2$, or a lens space (see \cite{choi}), which is a much more topologically restricted class of manifolds.

The case of $p=2$ could also have been handled using toric methods, but our chain complex in \cite{kech} is even simpler than the $p\neq2$ case as it has no differential, which greatly eases computations.

\section{ECH spectral invariants and Reeb dynamics}\label{s:spectral}

{This final section contains our computations of both the ECH action spectrum and knot filtration. We compute the action spectrum in \S\ref{ss:ECHspectrum} for the Morse-Bott contact form $(S^3,\lambda_{p,q})$, and we compute the knot filtration in \S\ref{ss:kECH} for the quadruple $(S^3,\xi_{std},T(p,q),pq+\delta)$, where $\delta$ is either zero or a positive irrational number that is sufficiently small when compared to the ECH index. Finally, in \S\ref{ss:dyn}, we use our computation to obtain the quantitative existence result for Reeb flows that we relied upon in \S\ref{s:calabi} to prove our results on surface dynamics. Our computations involve simplifying both the action and knot filtration to degree, which we relate to the ECH index in \S\ref{ss:index-degree}. The basic definitions and facts about both filtrations are reviewed in their respective sections.}

\subsection{Relating the ECH index to the degree of a Reeb current}\label{ss:index-degree} 
In this subsection we compute the function $I=2k\mapsto d$. Recall from Definition \ref{def:degreep} and Remark \ref{def:degree} that for a Reeb current $\alpha=\zb^B\zh^H\zp^P\zq^Q$,
\[
d(\alpha)=pq(B+H)+qP+pQ
\]
is the formula for its degree. The relationship between the ECH index and degree is given in terms of the sequence $N(a,b)$ for $a,b\in\R^+$, which is defined to be the sequence of all nonnegative integer linear combinations of $a$ and $b$, with multiplicities, in nondecreasing order. Let $N_k(a,b)$ denote the $k^\text{th}$ element of $N(a,b)$, including multiples, which, if $a,b\in\Z_{>0}$, will begin to occur at their least common multiple. We also always start with $N_0(a,b)=0$.

\begin{lemma}\label{lem:degNk}The degree of any Reeb current representing the generator of the group \\ $ECH_{2k}^{L(\varepsilon)}(S^3,\lambda_{p,q,\varepsilon})$ is $N_k(p,q)$.
\end{lemma}
\begin{proof}

Each group $ECH_{2k}^{L(\varepsilon)}(S^3,\lambda_{p,q,\varepsilon})$ is generated by a single element, namely, the homology class of any ECH chain complex generator with ECH index $2k$. This follows from the computation of the differential in \S\ref{s:ECHI}:  all even index chain complex generators are closed by Proposition \ref{prop:pqM}, Lemma \ref{lem:pqM}, and Proposition \ref{cor:nodiffp}, and those with the same index are all homologous by Corollary \ref{prop:bijectionp}. Note that these ECH chain complex generators are precisely the Reeb currents with $H=0$. Thus we may assume that $ECH_{2k}^{L(\varepsilon)}(S^3,\lambda_{p,q,\varepsilon})$ is generated as a group by the homology class of a single cycle, which may be represented by a single admissible Reeb current $\alpha$ with $H=0$.

From the discussion above, we know that $\alpha$ is homologous to any Reeb current of the form $\zb^B\zp^{P'}\zq^{Q'}$ where $P'=P-xp, Q'=Q+xq$ for some integer $x$. Thus we may assume $0\leq P<p$ and set $(m,n)=(Q,pB+P)$. This map is surjective onto $\Z^2_{\geq0}$ because $\zb^{\lfloor n/p\rfloor}\zp^{n-p\lfloor n/p\rfloor}\zq^m\mapsto(m,n)$. It is injective when restricted to generators with $P<p$. Since assuming $P<p$ is the same as choosing one representative of each homology class by Proposition \ref{prop:pqM} and Corollary \ref{cor:nodiffp}, we now have a bijection from $ECH_{2k}^{L(\varepsilon)}(S^3,\lambda_{p,q,\varepsilon})$ and $\Z^2_{\geq0}$.

To show that composing with $(m,n)\mapsto pm+qn$ is monotonically increasing with respect to index we note that in the proof of Corollary \ref{cor:nodiffp} in \S\ref{sss:finalcomp}, we can express the degree as the value of the function $px+qy$ on the line through $(Q,pB+P)$ of slope $-p/q$. To increase the index $2k$, we must move this line up/to the right (as explained in the proof), which evidently also increases the value of $px+qy$.
\end{proof}

\begin{remark}\label{rmk:converse}
 The arguments in the first paragraph of the proof of Lemma \ref{lem:degNk} also serve to prove the first two bullets in Theorem \ref{thm:ECC}. The proof of Lemma \ref{lem:degNk} proves a converse to the third bullet (which is immediate from the formula for degree) because if $d(\alpha)=d(\beta)$ then their corresponding lattice points are on the same line in the proof of Proposition \ref{cor:nodiffp}, which exactly means that $\alpha$ and $\beta$ are related by repeatedly replacing ends of $\mathpzc{H}_{p,q}$ flow lines with each other. Heuristically, the third bullet of Theorem \ref{thm:ECC} tells us that cylinders above Morse flow lines are degree zero, while Lemma \ref{lem:degNk} tells us they are the only degree zero curves. We have also proved they are the only possible contributions to the differential: c.f. \cite[\S4.1]{preech}, where this relationship was proved more directly.
\end{remark}

\subsection{The action filtration and the ECH action spectrum}\label{ss:ECHspectrum} 
Our main goal in this section is to state the ECH Weyl Law, Theorem \ref{thm:weyl}, relating the ECH action spectrum to the contact volume. It is used in \S\ref{ss:dyn}. To state the Weyl Law, we first need to first define the ECH spectrum of a contact manifold, which itself requires understanding the action filtration on the ECH chain complex. With all that in place, we are also able to compute the ECH spectrum of $(S^3,\lambda_{p,q})$ in Proposition \ref{prop:ck}.

Although most of this paper is concerned with the case $Y=S^3$ in which $H_1(Y)=0$, our definitions in this section pertain to the general case. As a first instance of this generality, when $\Gamma\in H_1(Y)$ is not specified, we define
\begin{equation}\label{e:decomp}
ECH_*(Y,\lambda) := \bigoplus_{\Gamma \in H_1(Y)} ECH_*(Y,\lambda,\Gamma).
\end{equation}

First we explain the action filtration. By Stokes' Theorem and the fact that $J$-holomorphic curves have nonnegative area in $\R\times Y$ for the symplectization contact form, the ECH differential decreases symplectic action.  Thus for each $L\in \R$, the subgroup $ECC^L(Y,\lambda,J)$ generated by the Reeb currents $\alpha$ for which $\A(\alpha) < L$ is a subcomplex, and its homology is the \emph{action filtered embedded contact homology}\footnote{We include the word ``action" here, which is usually omitted in the literature, to distinguish action filtered ECH from knot filtered ECH.} 
$ECH_*^L(Y,\lambda)$. It is independent of $J$ by \cite[Thm.~1.3]{cc2}, and depends heavily on $\lambda$.

There is a map
\begin{equation}\label{eq:incl}
\iota: ECH^L(Y,\lambda) \longrightarrow  \ ECH(Y,\lambda)
\end{equation}
induced by the inclusion of chain complexes. It is independent of $J$ by \cite[Thm.~1.3]{cc2}.

\begin{remark}[$U$ map]\label{rmk:U}
Assume $Y$ is connected. There is a chain map
\[
ECC_*(Y,\lambda,\Gamma,J)\to ECC_{*-2}(Y,\lambda,\Gamma,J)
\]
counting ECH index two $J$-holomorphic currents in $\R\times Y$ passing through a generic point; see \cite[\S2.5]{shs}. This map induces a map called ``$U$" on homology, and under the isomorphisms \cite{taubesechswf}-\cite{taubesechswf5} of Taubes as $\Z/2$-modules with Kronheimer-Mrowka's Seiberg-Witten Floer cohomology, it agrees with an analogous $U$ map as defined in \cite{KMbook}.

When $Y=S^3$, the $U$ map is an isomorphism: see \cite[\S4.1]{lecture}.
\end{remark}

We are now ready to define the ECH action spectrum.

\begin{definition}\cite[\S 4]{qech}
Let $(Y, \lambda)$ be a closed connected\footnote{One can define the ECH spectrum for disconnected contact 3-manifolds, cf.~\cite[\S 1.5]{lecture}.} contact 3-manifold. Embedded contact homology contains a canonical class, called the \emph{contact invariant} $c(\xi) \in ECH(Y,\xi,0)$, which is the homology class of the cycle given by the empty set of Reeb orbits. Assume that $c(\xi) \neq 0$.  There is a sequence
\[
0=c_0(Y,\lambda) < c_1(Y,\lambda) \leq c_2(Y,\lambda) \leq ... \leq \infty
\]
of real numbers, called the \emph{ECH action spectrum}. For nondegenerate $\lambda$, the number $c_k(Y,\lambda)$ is the infimum of actions $L$ for which there is a class $\eta\in ECH^L(Y,\lambda,0)$ for which $U^k\eta=c(\xi)$. Note that $c_k(Y,\lambda)<\infty$ if and only if the contact class is in the image of $U^k$.

For degenerate $\lambda$, choose a sequence $f_n:Y\to\R_{>0}$ of functions for which $\lim_{n\to\infty}f_n=1$ in the $C^0$-topology and with $f_n\lambda$ nondegenerate. Then
\[
c_k(Y,\lambda):=\lim_{n\to\infty}c_k(Y,f_n\lambda).
\]
In both cases, the $c_k$ are valued in the action spectrum of $\lambda$.

\end{definition}

Note that $c_k(Y,\lambda)$ is the symplectic action of some ECH generator which has ECH index $2k$. Thus the $k^\text{th}$ ECH spectral number can be thought of as the symplectic action of ``homologically essential" ECH generators, specially selected by the $U$ map and ECH index. The quantities $c_k$ satisfy a number of nice properties, such as monotonicty under symplectic embedding and scaling with the contact form. In particular, when $Y$ is the boundary of a Liouville domain, these properties are enough to make $c_1$ a symplectic capacity. The ECH spectrum is especially useful for obstructing symplectic embeddings of such Liouville domains. In many interesting cases, the resulting symplectic obstructions are sharp.

A remarkable and very useful fact about the ECH spectrum is that it asymptotically recovers the contact volume, which we elucidated for the (irrational) ellipsoid in \cite[\S 6.2]{kech}.

\begin{theorem}[ECH Weyl Law {\cite[Thm.~1.2]{weyl}}]\label{thm:weyl}
Let $(Y,\lambda)$ be a closed contact 3-manifold with nonvanishing contact invariant. If $c_k(Y,\lambda) < \infty$ for all $k$, then
\[
\lim_{k \to \infty} \frac{c_k(Y,\lambda)^2}{2k} = \vol(Y, \lambda).
\]
\end{theorem}
The ECH Weyl Law as we have stated it is a corollary of \cite[Thm.~1.3]{weyl}, which holds for all contact three-manifolds.  Seiberg-Witten Floer theory features in the proof in the general case; when $Y=S^3$, as in \cite[Prop.~8.6(b)]{qech}, its use can be largely circumvented, though it bears mention that the existence and properties of ECH cobordism maps require Seiberg-Witten Floer theory.

\begin{remark}
Frequently, the ECH Weyl law is stated in terms of the ECH capacities for Liouville domains $(X,\omega)$, \[\lim_{k \to \infty} \frac{c_k(X,\omega)^2}{4k} = \vol(X, \omega).\] Note that $
 \vol(X,\omega) = \frac{1}{2}\vol(Y,\lambda),$ because by Stokes' Theorem $\int_X \omega \wedge \omega = \int_Y \lambda \wedge d\lambda.$ Recall that symplectic volume is defined by $\vol(X,\omega):=\frac{1}{2}\int_X \omega \wedge \omega$ and contact volume is defined by $\vol(Y,\lambda):= \int_Y \lambda \wedge d\lambda$.
\end{remark}


We now compute the ECH spectrum of the degenerate contact form $\lambda_{p,q}$. (It may be worth noting that by Chern-Weil theory, $\vol(S^3, \lambda_{p,q}) = |e| = \frac{1}{pq}$.) 

\begin{proposition}\label{prop:ck} We have $c_k(S^3,\lambda_{p,q}) = N_k(1/q,1/p)$.
\end{proposition}
\begin{proof}
Since $\lambda_{p,q}$ is degenerate, we will use the perturbations $f_n\lambda_{p,q}=\lambda_{p,q,1/n}$, by setting
\[
f_n=1+\frac{1}{n}\fp^*\mathpzc{H}_{p,q}.
\]
For large $n$, we have $c_k(Y,\lambda_{p,q,1/n})<L(1,n)$, making the capacities eventually constant in $n$. Thus we must only compute $c_k(Y,\lambda_{p,q,1/n})$ when $k$ and $L(1/n)$ satisfy Lemma \ref{lem:orbitseh}, allowing us to consider only admissible Reeb currents $\zb^B\zh^H\zp^P\zq^Q$.

By Lemma \ref{lem:efromL}(ii) 
\[
\A(\zb^B)=B, \ \ \ \A(\zp^P)=\frac{P}{p}, \ \ \ \A(\zq^Q)=\frac{Q}{q}.
\]
Therefore
\begin{equation}\label{eqn:Ad}
\A(\zb^B\zp^P\zq^Q)=\frac{d(\zb^B\zp^P\zq^Q)}{pq}.
\end{equation}

The result follows from Lemma \ref{lem:degNk} and the fact that each homology group \\ $ECH^{L(1/n)}_{2k}(S^3,\lambda_{p,q,1/n})$ is generated by the homology class of a Reeb current of the form $\zb^B\zp^P\zq^Q$. Although this ECH generator is not unique, all homologous generators have the same degree and thus by (\ref{eqn:Ad}) the same action; this is explained at the beginning of the proof of Lemma \ref{lem:degNk}.
\end{proof}

\subsection{The knot filtration and the ECH linking spectrum}\label{ss:kECH} 
The aim of this section is to compute knot filtered ECH, a topological spectral invariant, for for the transverse positive torus knots $T(p,q)$ in $(S^3,\xi_{std})$ with rotation numbers $pq+\delta$ for $\delta$ positive, irrational, and small. We complete this computation in Theorem \ref{thm:kECH} using our model contact forms $\lambda_{p,q,\varepsilon}$ of \S\ref{s:topology}-\ref{s:ECHI}. We also introduce the linking spectrum $c_k^{\op{link}}$ and its key properties, which we require in \S\ref{ss:dyn}.

We must first define knot filtered ECH and explain its invariance properties. Unlike action filtered ECH (and hence the ECH spectrum), which depends heavily on $\lambda$, knot filtered ECH depends only on the smooth manifold $Y$,\footnote{The original definition of \cite[Thm.~5.3]{HuMAC} was for $\rot(b)$ irrational and $H_1(Y)=0$. In \cite[Thms.~5.2 \& 5.3]{weiler}, the assumption on $H_1$ was weakened to $H_1(Y)$ torsion.} the contact structure $\xi$, the filtration level $K$, a transverse knot $b$ which can be realized as an elliptic orbit for some contact form $\lambda$ of $\xi$, and its rotation number $\rot(b)$ as an elliptic orbit. This last component is the only vestige of the contact form, and is computed in the canonically defined trivialization in which a pushoff of $b$ has linking number zero with $b$ (see Remarks \ref{rem:linkingtriv} and \ref{rem:c1sl}). In light of this invariance, knot filtered ECH is denoted by $ECH_*^{\fb \leq K}(Y,\xi,b,\rot(b))$. It was originally defined in \cite{HuMAC}.

We also recount \cite[Thm.~1.5]{kech}, which provides the favorable Morse-Bott circumstances in which one can allow for rational rotation numbers; otherwise, $\rot(b)$ must be irrational. We now review how to define knot filtered ECH.

On a contact manifold $(Y,\lambda)$ with $H_1(Y)=0$, let $b$ be any (embedded) Reeb orbit and let $\rot(b)$ denote its rotation number in the canonical trivialization described above. Let $b^m \alpha$ be a Reeb current such that the orbit set $\alpha$ does not contain $b$ and where $m \in \Z_{\geq 0}$. The knot filtration is respected by any holomorphic curve, regardless of whether or not $\alpha$ is an ECH generator. The \emph{knot filtration} on Reeb currents of $(Y,\lambda)$ with respect to $(b, \rot(b))$ is
\begin{equation}\label{eq:filt}
\mathcal{F}_b(b^m\alpha) = m \rot(b) + \ell(\alpha,b).
\end{equation}
The linking number $\ell(\alpha, b)$ of $b$ with the Reeb current $\alpha$ is
\[
{\ell(\alpha,b)=\sum_im_i\ell(\alpha_i,b).}
\]
{If $b$ is a nondegenerate elliptic Reeb orbit then $\rot(b)\in\R\setminus\Q$ and the function $\mathcal{F}_b$ is not integer valued. Yet its values may still be in a discrete set of real numbers: when $\rot(b)>0$ and all other Reeb orbits link positively with $b$, which occurs in our setting, where $b$ is the binding of an open book, or when $\rot(b)<0$ and all other Reeb orbits link negatively with $b$.}

{As proved in \cite[Lem.~5.1]{HuMAC}, the function $\mathcal{F}_b$ is a filtration on the ECH chain complex, as the ECH differential (and more generally, any map counting appropriate $J$-holomorphic curves in $\R\times Y$) does not increase the knot filtration.}

If $K \in \R$, we use $ECH_*^{\mathcal{F}_b \leq K}(Y,\lambda,J)$ to denote the homology of the subcomplex generated by admissible Reeb currents $b^m\alpha$ with $\mathcal{F}_b (b^m\alpha)  \leq K$.  Hutchings proved that $ECH_*^{\mathcal{F}_b\leq K}(Y,\lambda,J)$ is a topological invariant in the following sense.  

\begin{theorem}{\em \cite[Thm.~5.3]{HuMAC}}\label{thm:kech}
Let $(Y, \xi)$ be a closed contact 3-manifold with $H_1(Y)=0$, $b\subset Y$ be a transverse knot and $K\in \R$.  Let $\lambda$ be a contact form with $\Ker(\lambda) = \xi$ such that $b$ is an elliptic Reeb orbit with rotation number $\rot(b)\in\R \setminus \Q$.  Let $J$ be any generic $\lambda$-compatible almost complex structure.  Then $ECH_*^{\mathcal{F}_b\leq K}(Y,\lambda,J)$ is well-defined, and this homology, together with the inclusion induced map
\begin{equation}\label{eq:kechi1}
ECH_*^{\mathcal{F}_b \leq K}(Y,\xi, b, \rot(b)) \to ECH(Y,\xi),
\end{equation}
depends only on $Y, \xi, b, \rot(b),$ and $K$.  
\end{theorem}

We generalized this result to allow for rational rotation numbers in \cite[\S7]{kech} in the following sense, via a double direct limit argument, which built on our Morse-Bott methods in \cite[\S3, 7]{preech}.

\begin{theorem}{\em \cite[Thm.~1.5]{kech}}\label{thm:ok}
Let $(Y, \xi)$ be a closed contact 3-manifold with $H_1(Y)=0$, $b\subset Y$ be a transverse knot and $K\in \R$.  If $\lambda$ is degenerate, 
we define
\[
ECH_*^{\mathcal{F}_b \leq K}(Y,\lambda, b, \rot(b)) :=
\lim_{\varepsilon \to 0}ECH_*^{\afe}(Y, \lambda_{\varepsilon}, b, \rot_{\varepsilon}(b),J_{\varepsilon}),  \]
where $\{(\lambda_\ve, J_\ve)\}$ is a  {knot admissible pair} for $(Y,\lambda, b, \rot(b))$.    Then \\ $ECH_*^{\mathcal{F}_b \leq K}(Y,\lambda, b, \rot(b))$ is well-defined, and this homology, together with the inclusion induced map
\begin{equation}\label{eq:kechi2}
ECH_*^{\mathcal{F}_b \leq K}(Y,\xi, b, \rot(b)) \to ECH(Y,\xi),
\end{equation}
depends only on $Y, \xi, b, \rot(b),$ and $K$.
\end{theorem}
Here 
\[
\left( ECC_*^{\afe}(Y, \lambda_{\varepsilon}, b, \rot_{\varepsilon}(b),J_{\varepsilon}), \partial \right)
\]
is the action filtered subcomplex, which has been further restricted to the knot filtered subcomplex where $\fb \leq K$. It underlies the homology groups $ECH_*^{\afe}(Y, \lambda_{\varepsilon}, b, \rot_{\varepsilon}(b),J_{\varepsilon})$ appearing in the theorem statement.   Note that by the definition of $\ve(L)$,
\[
\lim_{L \to \infty}ECH_*^{\af}(Y, \lambda_{\varepsilon(L)}, b, \rot_{\varepsilon(L)}(b),J_{\varepsilon(L)}) = \lim_{\ve \to \infty}ECH_*^{\afe}(Y, \lambda_{\varepsilon}, b, \rot_{\varepsilon}(b),J_{\varepsilon}).
\]

When $Y$ is $S^3$ and $\xi$ is the standard tight contact structure, one can encode knot filtered embedded contact homology by a sequence of real numbers, similarly to the ECH spectrum defined in \S \ref{ss:ECHspectrum}, as observed by Hutchings in \cite{letternw}.
\begin{definition}\label{def:cklink}
Let $(Y,\xi) = (S^3, \xi_{std})$.  If $k$ is a nonnegative integer, define the \emph{ECH linking spectrum}
\[
c_k^{\op{link}}(b,\rot(b)) \in [-\infty, \infty)
\]
to be the infimum over $K$ such that the degree $2k$ generator of $ECH(S^3,\xi_{std})$ is in the image of the inclusion induced map \eqref{eq:kechi1} if $\rot(b)$ is irrational or in the image of the inclusion induced map   \eqref{eq:kechi2} if $\rot(b)$ is rational.

We have that 
\[
c_k^{\op{link}}(b,\rot(b)) \leq c_{k+1}^{\op{link}}(b,\rot(b)),
\]
because the $U$ map (which induces an isomorphism from $ECH_{2k+2}(S^3,\xi_{std})$ to $ECH_{2k}(S^3,\xi_{std})$) is given in terms of a $J$-holomorphic current, and any $J$-holomorphic current respects the linking filtration \cite[Lem.~5.1]{HuMAC}.  
Additionally, if $\theta < \theta'$ then 
\begin{equation}\label{eq:cklinkth}
c_k^{\op{link}}(b,\theta) \leq  c_k^{\op{link}}(b,\theta'),
\end{equation}
by the cobordism map argument in \cite[\S 7]{kech}, similar to the proof of \cite[Lem.~5.1]{HuMAC}.

\end{definition}

The definition of ``knot admissible" is as follows.    Precise definitions of each condition can be found in the statement of \cite[Lem.~7.11]{kech}, which were collected from \cite{cc2}.  

\begin{definition}\label{def:knotadmissible}
A pair of families $\{(\lambda_\ve, J_\ve)\}$ is said to be a \emph{knot admissible pair} for $(Y,\lambda, b, \rot(b))$, where $\lambda$ is a degenerate contact form admitting the transverse knot $b$ as an embedded Reeb orbit whenever
\begin{itemize}
\itemsep-.25em

\item  $f_{s}: [0,c_0]_s \times Y \to \R_{>0}$  are smooth functions such that $\frac{\partial f}{\partial s}>0$ and  $\lim_{s \to 0}f_{s}=1$ in the $C^0$-topology;
\item  There is a full measure set $\mathcal{S} \subset (0, c_0]$, such that for each $\varepsilon \in \mathcal{S},$ $f_{\ve}\lambda$  is $L(\ve)$-nondegenerate where $L(\ve)$ monotonically increases towards $+\infty$ as $\ve$ decreases towards 0;
\item $\lambda_{\varepsilon}:=f_{\varepsilon}\lambda$ each admit the transverse knot $b$ as an embedded elliptic Reeb orbit (when $\ve \neq 0$) and $\{ \rot_{\varepsilon}(b)\}$ is monotonically decreasing to $\rot(b)$ as $\varepsilon \in [0, c_0]$ decreases to 0;
\item $J_{\varepsilon}$ is an $ECH^{L(\varepsilon)}$ generic $\lambda_{\varepsilon}$-compatible almost complex structure (when $\varepsilon \neq 0$).
\end{itemize}
Sometimes we also suppress the almost complex structure and refer to the sequence of contact forms $\{\lambda_\ve\}$ as a \emph{knot admissible family}, provided it satisfies the above conditions.  By the discussion in \cite[Lem.~7.11, Rem.~7.12]{kech} summarizing results of \cite{cc2}, it follows that for any $\varepsilon' \in (0,\varepsilon)$, the admissible Reeb currents of action less than $L(\varepsilon)$ associated to $\lambda_{\varepsilon}$ and $\lambda_{\varepsilon'}$ are in bijective correspondence.  
\end{definition}

  In particular, our methods from \cite{kech} allows us to compute knot filtered embedded contact homology of $(S^3,\xi_{std},T(p,q),pq)$ via successive approximations using the sequence $\{ \lambda_{p,q,\ve} \}$, which is a knot admissible family by the computations preceding and summarized in Lemma \ref{lem:orbtrivCZp}.  The latter arguments also give a means of constructing a knot admissible family for any fiber of any Seifert fiber space of negative Euler class.
    
We now compute the positive $T(p,q)$ knot filtration.\footnote{Should it be of interest to the reader, we reviewed the computation of unknot filtered ECH in the irrational ellipsoid from \cite[\S 5]{HuMAC} in \cite[\S 6.3]{kech}.   }

\begin{proposition}\label{prop:Fb}
Given $(S^3,\lambda_{p,q,\varepsilon})$ and $\varepsilon(L)$ as in Proposition \ref{prop:morsep}, then for any Reeb current $\alpha$ not including the right handed $T(p,q)$ torus knot $\zb$, 
\[
\mathcal{F}_{\zb}(\mathpzc{b}^B\alpha) = d(\mathpzc{b}^B \alpha) + B\delta_L.
\]
\end{proposition}

\begin{proof}
By Lemma \ref{lem:pagetrivCZp}, we know  $\rot(\zb)=pq+\delta_L$.  We have that $\alpha=\zh^H\zp^P\zq^Q$, thus again using Corollary \ref{cor:linkingnumber} and the definition of degree,
\begin{align*}
\fb(\zb^B\alpha)&=B\rot(\zb)+\ell(\alpha,\zb)
\\&=B(pq+\delta_L)+\ell(\zh^H,\zb)+\ell(\zp^P,\zb)+\ell(\zq^Q,\zb)
\\&=B(pq+\delta_L)+pqH+qP+pQ
\\&=d(\zb^B\zh^H\zp^P\zq^Q)+B\delta_L.
\end{align*}

\end{proof}

\begin{remark}\label{rmk:tknot}
Etnyre proved that two positive transverse torus knots have the same smooth knot type and self-linking number if and only if they are transversely isotopic in \cite{torus}. Therefore, there is a standard $T(p,q)$ in $(S^3,\xi_{std})$, precisely the one which as a transverse knot has maximum self-linking equal to $pq-p-q$. This is precisely the self-linking corresponding to rotation number $pq+\delta_L$ in the trivialization $\tau_\Sigma$, using $c_\Sigma([\Sigma])=p+q-pq$; see \S\ref{ss:chern} and \S\ref{ss:pagetau}. (The connection to $c_\tau$ is that $\op{sl}(\zb):=-c_1(\xi|_\Sigma, \tau)$.)
\end{remark}

{We now compute knot filtered ECH for the $T(p,q)$ (and make a few clarifying remarks before completing the proof).}

\begin{theorem}\label{thm:kECH} Let $\xi_{std}$ be the standard tight contact structure on $S^3$.  
Let ${b}$ be the standard {right-handed} (positive) transverse $T(p,q)$ torus knot for $(p,q)$ relatively prime.  If $\rot({b}) = pq,$ then
\[
ECH_{2k}^{\fb \leq K}(S^3,\xi_{std},T(p,q),pq)=\begin{cases}\Z/2&K\geq{ N_k(p,q)  },
\\0&\text{otherwise,}\end{cases}
\]
and in all other gradings $*$,
\[
ECH_*^{\fb \leq K}(S^3,\xi_{std},T(p,q),pq)=0.
\]
If $\rot({b}) = pq+\delta$, where $\delta$ is a sufficiently small positive irrational number less than $\min\{1/p,1/q\}$, then up to grading $*$ and filtration threshold $K$ inversely proportional to $\delta$,
\[
ECH_{2k}^{\fb \leq K}(S^3,\xi_{std},T(p,q),pq+\delta)=\begin{cases}\Z/2&K\geq{ N_k(p,q) + \delta {(\$N_k(p,q) -1)}},
\\0&\text{otherwise,}\end{cases}
\]
where { $\$N_k(p,q)$ is the number of repeats in $\{ N_j(p,q)\}_{j\leq k}$ with value $N_k(p,q)$,} and in all other gradings $*$, up to the threshold inversely proportional to $\delta$,
\[
ECH_*^{\fb \leq K}(S^3,\xi_{std},T(p,q),pq+\delta)=0.
\]
\end{theorem}

\begin{remark}\label{rem:kechconj}
Theorem \ref{thm:kECH} and the known calculations for the unknot in \cite[Prop.~5.5]{HuMAC}, are consistent with the conjecture that for any transverse knot $b \subset S^3$, 
\[
\lim_{k \to \infty} \frac{c_k^{\op{link}}(b,\rot(b))^2}{2k} = \rot(b).
\]
\end{remark}

\begin{remark}\label{rem:delta}
The formula for the threshold of the grading $2k$ up to which we can compute $c^{\op{link}}_k$ and the size of $\delta$ relies on our ability to find a model contact form for which $\rot(T(p,q))=pq+\delta$. If $\delta$ is fixed, this takes the form of an upper bound on $\delta$ in terms of $1/k, p$, and $q$, meaning it will only be satisfied for $k$ up to a finite value (a precise upper bound would require adding further computations relating $\varepsilon, L, k$, and $\delta$ arising from a specific $\mathpzc{H}_{p,q}$ to Lemma \ref{lem:orbitseh}). Because our applications require us to estimate $c^{\op{link}}_k(T(p,q),pq+\delta)$ for any fixed $\delta$ and arbitrarily large $k$, knowing the formula for this finite maximum value of $k$ in terms of $\delta$ does not help us. However, we say a few words about what that formula might look like now.

\begin{itemize}
\item In order for the linking filtration values to be ordered according to the ECH index, we require $N_k(p,q)+\delta(\$N_k(p,q)-1)\leq N_\ell(p,q)$, where $\ell$ is the first index greater than $k$ for which $N_\ell(p,q)>N_k(p,q)$. As $k$ grows, the number of values in $N(p,q)$ equalling $N_k(p,q)$ grows, and moreover, because $p$ and $q$ are coprime, there will eventually be an index where $N_\ell(p,q)-N_k(p,q)=1$.  Thus at that point, using Lemma \ref{lem:Nk}(ii), because our estimate must hold for all $k$, we require
\[
\delta\left(\sqrt{\frac{2k}{pq}}-1\right)\leq1,
\]
which can be rearranged into an upper bound on $k$ in terms of $\delta$.

\item Hutchings suggests in \cite{letternw} that it may be possible to compute $c^{\op{link}}_k$ using a toric perturbation, assuming the isomorphism to the combinatorial chain complex akin to that described in \S\ref{ss:toric}) so long as
\[
\delta\left\lfloor\frac{N_k(p,q)}{pq}\right\rfloor<1,
\]
which, using our estimates in Lemma \ref{lem:Nk}(i), implies
\[
\sqrt{8kpq+(p+q+1)^2}-(p+q+1)<\frac{2pq}{\delta}-2,
\]
providing an upper bound on $k$ in terms of $\delta$.

\item Our arguments in the proof of Lemma \ref{lem:d0g0} require $\delta<\min\{1/p,1/q\}$, which is why that hypothesis appears in the statement of Theorem \ref{thm:kECH}. However, this produces no restrictions on $k$. Note that our results in \S\ref{ss:dyn} hold for potentially larger $\delta$.
\end{itemize}
\end{remark}

\begin{proof}

{We may use $\lambda_{p,q,\varepsilon}$ to compute ECC up to the index and action determined from $\varepsilon$ by Lemma \ref{lem:orbitseh}, because all contact forms are assumed to be contactomorphic and all share $T(p,q)$ as an elliptic orbit, with monotonically decreasing rotation number. The choice of $\varepsilon$ determines the value of $\delta$ and thus the rotation number, which is why we require $\delta$ to be small, as only small $\delta$ can be realized thusly.}

We may argue as in the proof of Lemma \ref{lem:degNk} that the lower bound on $K$ is precisely the value of the knot filtration on the generator $\zb^B\zp^P\zq^Q$ of index $2k$ with $P<p$, and that this is also the value which the knot filtration takes on all generators (and hence cycles) homologous to it. Again, this value may be computed using Proposition \ref{prop:Fb} and Lemma \ref{lem:degNk}.

By Theorem \ref{thm:ok} and the discussions in \cite[\S 7]{kech}, we can take direct limits to realize $\delta=0$.
\end{proof}

\subsection{Quantitative positive torus knotted Reeb dynamics}\label{ss:dyn}

We now establish the necessary input from contact geometry, from which we derived the mean action bounds for area preserving diffeomorphisms of higher genus surfaces with one boundary component in terms of the Calabi invariant in \S \ref{s:calabi}.   First, we use our computation of torus knot filtered ECH in Theorem \ref{thm:kECH} in conjunction with the ECH Weyl Law in Theorem \ref{thm:weyl} to provide an upper bound on the action of any orbit set and a lower bound on the linking filtration.

\begin{proposition}\label{prop:actionlinking}
Let $\lambda$ be any contact form for $(S^3, \xi_{std})$ which admits a positive $T(p,q)$ torus knot, denoted by ${b}$, as an elliptic Reeb orbit with rotation number $pq+\Delta$, where $\Delta$ is positive and irrational. Then for all $\epsilon > 0$, if $k$ is a sufficiently large positive integer, there exists a Reeb current $\gamma$, which is an ECH generator, not containing the $T(p,q)$ knotted orbit ${b}$, and a nonnegative integer $m$,  such that
\begin{equation}\label{eq:actionbound}
\frac{(\A(\gamma) + m\A({b}))^2}{2k} \leq \vol(\lambda) + \epsilon,
\end{equation}
\begin{equation}\label{eq:fbound}
 \ell(\gamma,{b}) +  m({pq}+\Delta) \geq {N_k(p,q)}.
\end{equation}
\end{proposition}
A few remarks are in order before we proceed with the proof.

\begin{remark}\label{rem:subleadfun}
We will subsequently assume $\A(b)=1$ and rewrite \eqref{eq:actionbound} as
\begin{equation}\label{eq:infAhutchings}
\A(\gamma) + m \leq \sqrt{2k( \vol(\lambda) + \epsilon)}.
\end{equation}
\end{remark}

\begin{remark}\label{rem:delicate}
When the torus knot $b$ is degenerate, we expect that the result is still true, by employing methods akin to \cite{two}.  We do not carry out this out as it is not needed for the proofs of our main results.

In the setting where $\rot(b) = pq+ \Delta$, we can only recover so much of knot filtration on ECH, though we can recover it for arbitrarily large $k$ provided our perturbation $\lambda_{p,q,\varepsilon}$ is sufficiently small, which in turn governs the size of $\Delta$, as described in Remark \ref{rem:delta}.
\end{remark}

\begin{proof}
We prove the theorem by separately handling whether or not the contact form is degenerate.   
\medskip

\noindent \textbf{Step 1:} Assume that $\lambda$ is nondegenerate.  Let $J$ be a generic $\lambda$-compatible almost complex structure on $\R \times S^3$.  
By the ECH Weyl law, Theorem \ref{thm:weyl}, if $k$ is sufficiently large then
\begin{equation}\label{eq:action1}
\frac{c_k(S^3,\lambda)^2}{2k} \leq \vol(S^3, \lambda) + \epsilon.
\end{equation}
Assume that $k$ is sufficiently large in this sense.  From the definition of $c_k(S^3,\lambda)$, there exists a cycle
\[
\alpha = \sum_i \alpha_i \in ECC_{2k}(S^3,\lambda,J),
\]
representing the generator of $ECH_{2k}(S^3, \lambda)$ such that each orbit set $\alpha_i$ in the cycle has symplectic action $\A(\alpha_i) \leq c_k(S^3,\lambda)$.  By \eqref{eq:action1}, for each $i$,
 \begin{equation}\label{eq:action}
\frac{\A(\alpha_i)^2}{2k} \leq \vol(S^3, \lambda) + \epsilon.
\end{equation}
By Theorem \ref{thm:kECH}, the fact that the knot filtration is an invariant when the rotation number is irrational, and because it is increasing in the rotation number, see (\ref{eq:cklinkth}), we must have
\begin{equation}\label{eq:cklinkbound}
c_k^{\op{link}}(T(p,q), pq+ \Delta) \geq c_k^{\op{link}}(T(p,q), pq ) = N_k(p,q). 
\end{equation}
Thus for at least one $i$, we have that
\begin{equation}\label{eq:link}
\fb(\alpha_i) \geq N_k(p,q).
\end{equation}
Otherwise  $\alpha$ would be nullhomologous, a contradiction to the fact that $\alpha$ represents a generator of  $ECH_{2k}(S^3, \lambda)$.  

Pick such an $i$ such that \eqref{eq:action} and \eqref{eq:link} hold.  Write $\alpha_i = {b}^m\gamma$, where $m \in \Z_{\geq 0}$ and $\gamma$ is a Reeb current not including ${b}$.  Then the desired inequalities \eqref{eq:actionbound} and \eqref{eq:fbound} follow from \eqref{eq:action} and \eqref{eq:link}, respectively. 
 \bigskip

\noindent \textbf{Step 2:} 
If $\lambda$ is degenerate, let $\{f_n\}_{n=1,2,...}$ be a sequence of functions $S^3 \to \R$ with $0 < f_n \leq 1$ such that $f_n \to 1$ in $C^\infty$ and the contact form $\lambda_n = f_n\lambda$ is nondegenerate.  We can arrange that each $\lambda_n$ satisfies the hypotheses of the proposition for each $n$, with the same orbit ${b}$, such that $\A({b})$ and $\rot({b})$ do not depend on $n$, similar to arguments in \cite[App.~A]{abw}. 

The ECH Weyl law does not assume nondegeneracy of the contact form, thus we may assume that $k$ is sufficiently large so that 
\begin{equation}\label{eq:actiondeg}
\frac{c_k(S^3,\lambda)^2}{2k} \leq \vol(S^3, \lambda) + \epsilon.
\end{equation}
Assume that $k$ is sufficiently large in this sense. By the monotonicity property of the ECH spectrum, Proposition \ref{prop:ck}(ii), we have for each $n$ that
\begin{equation}\label{eq:fnweyl} 
\frac{c_k(S^3,f_n\lambda)^2}{2k} \leq \vol(S^3,\lambda)+\epsilon.
\end{equation}
The argument in Step 1 applies: for each $n$, we can choose a nonnegative integer $m(n)$, and an orbit set $\gamma(n)$ not including ${b}$, such that
\begin{align}
\frac{(\A_n(\gamma(n)) + m(n)\A({b}))^2}{2k} &\leq \vol(S^3, \lambda) +  \epsilon, \label{eq:action-n} \\
\ell(\gamma(n),{b}) + m(n) (pq + \delta) & \geq N_k(p,q). \label{eq:link-n} 
\end{align}
Here $\A_n(\gamma) : = \int_\gamma \lambda_n$ is the action determined by $\lambda_n$.

The existence of an $n$-independent upper bound on the nonnegative integer $m(n)$ follows from the inequality \eqref{eq:action-n}.  Thus we may pass to a subsequence so that $m(n)$ is constant; denote this constant by $m$.  Likewise,  there is an $n$-independent upper bound on $\A_n(\gamma(n))$ as a result of the inequality \eqref{eq:action-n}.  Since the Reeb vector fields of $\lambda_n$ and $\lambda$ are nonsingular, there are also $n$-independent lower bounds on the action of any Reeb orbit computed with respect to $\lambda_n$.  By appealing to \eqref{eq:action-n} again, we have that the total multiplicity of all the Reeb orbits in $\gamma(n)$ admit an $n$-independent upper bound.

We may then pass to a further subsequence so that $\gamma(n) =\{ \gamma_i(n), m_i\}$, where $m_i$ is an $n$-independent positive integer, and $\gamma_i(n)$ is a embedded Reeb orbit associated to $\lambda_n$ with $\lim_{n \to \infty} \gamma_i(n) = \gamma_i$ a Reeb orbit associated to $\lambda$.  The Reeb orbits $\gamma_i$ need not be embedded or distinct, because the contact form $\lambda$ is degenerate.

Since the Reeb orbit ${b}$ has irrational rotation number and is thus a nondegenerate Reeb orbit associated to $\lambda$, we can conclude that none of the orbits $\gamma_i$ are a cover of ${b}$.  Thus $\gamma = \prod_i \gamma_i^{m_i}$ is a well-defined Reeb current of $\lambda$ that does not include ${b}$, and the sequence of orbit sets $\gamma(n)$ converges ``as a current" to $\gamma$.   In particular, $\A_n(\gamma(n))$ converges to $\A(\gamma)$, and $\ell(\gamma(n),{b})$ equals $\ell(\gamma,{b})$ for $n$ sufficiently large.  Thus the inequalities \eqref {eq:action-n} and \eqref{eq:link-n} imply that the pair $(\gamma, m)$ satisfy the desired conclusions.

\end{proof}

Before we can use Proposition \ref{prop:actionlinking} to establish the mean action bounds in Theorem \ref{thm:toruslinking}, we first need some estimates on the quantities $N_k(p,q)$ and $\$N_k(p,q)$. Part (i) is an improvement of \cite[Lem.~3.2]{HuMAC}, which is only possible because $p,q\in\Z$ rather than in $\R$. 

\begin{lemma}\label{lem:Nk} 
Given positive integers $p$ and $q$, there are infinitely many $k$ for which the following hold simultaneously:
\begin{enumerate}[{\em (i)}]
\item $N_k(p,q)=\dfrac{\sqrt{8pqk+(p+q+1)^2}-(p+q+1)}{2}$,
\item when $(p,q)\neq(1,1)$, $\$N_k(p,q)\geq\sqrt{\dfrac{2k}{pq}}$.
\end{enumerate}
\end{lemma}
\begin{proof} Recall that $N_k(p,q)=L$ where $L$ is the lowest value that the function $px+qy$ takes for $(x,y)\in\Z^2_{\geq0}$ subject to the constraint that the triangle $T_k$ between the axes and $px+qy=L$ contains at least $k+1$ lattice points, including those on the axes and line.

The $k$ in question will be the largest $k$ with $N_k(p,q)=npq$ for some $n\in\Z_{\geq0}$. In particular, their triangles $T_k$ will have vertices $(0,0), (nq,0)$, and $(0,np)$.

To prove (i) we must solve for $n$ in terms of $k$. Because $k$ is the largest index so that $N_k(p,q)=npq$, its triangle $T_k$ contains exactly $k+1$ lattice points. We may solve for the number of lattice points in another way: it is precisely the number of lattice points in an $np\times nq$ triangle, minus the $n+1$ lattice points along the diagonal, divided by two, and finally with the lattice points along the diagonal added back in. That is,
\[
k+1=\frac{(np+1)(nq+1)-(n+1)}{2}+n+1 \iff n=\frac{\sqrt{8kpq+(p+q+1)^2}-(p+q+1)}{2pq},
\]
where we have taken only the larger solution to the quadratic equation because the smaller one is negative. The fact that $N_k(p,q)=npq$ completes the proof of conclusion (i).

We now turn to the proof of (ii). There are precisely $n+1$ lattice points along the diagonal of the triangle $T_k$, meaning there are exactly $n+1$ terms in the sequence $N(p,q)$ equal to $N_k(p,q)$ and with index less than or equal to $k$. We compute
\begin{align*}
\$N_k(p,q) = n+1 &= \frac{\sqrt{8kpq+(p+q+1)^2}-(p+q+1)}{2pq} + 1 \geq \sqrt{\frac{2k}{pq}}
\\&\Longleftarrow \frac{p+q+1}{2pq} \leq 1,
\end{align*}
which holds because at least one of $p$ or $q$ is at least two.
\end{proof}

We now restate and prove Theorem \ref{thm:toruslinking}.
\begin{theorem}\label{thm:toruslinking6} 
Let $\lambda$ be a contact form on $(S^3,\xi_{std})$ with $\vol(\lambda)=V$.  Suppose that the Reeb vector field admits the right handed $T(p,q)$ torus knot, denoted by ${b}$, as an elliptic Reeb orbit with symplectic action 1 and rotation number $pq + \Delta$, where $\Delta$ is a positive irrational number.  If
\[
V < \frac{pq}{(pq+\Delta)^2},
\]
then
\[
\inf \left\{ \frac{\A(\gamma)}{\ell(\gamma,{b})} \ \bigg \vert \ \gamma \in \mathcal{P}(\lambda) \setminus \{ {b}\} \right\} \leq \sqrt{\frac{V}{pq}}.
\] 
\end{theorem}
Here $\ell(\gamma,b)$ denotes the linking number of the embedded Reeb orbit $\gamma$ with the torus knot $b$ and $\mathcal{P}(\lambda)$ denotes the set of embedded Reeb orbits of $\lambda$.  Note that our computation of knot filtered ECH in Theorem \ref{thm:kECH} demonstrates that $\mathcal{P}(\lambda) \setminus \{ {b}\} \neq \emptyset$.

\begin{proof}
\textbf{Step 1:} Let $\mathcal{O}_0(\lambda)$ denote the set of Reeb currents which do not include the torus knot $b$.  It suffices to establish
\begin{equation}\label{eq:desire}
\inf \left\{ \frac{\A(\alpha)}{\ell(\alpha,{b})} \ \bigg \vert \ \alpha \in \mathcal{O}_0(\lambda) \right\} \leq \sqrt{\frac{V}{pq}}.
\end{equation}
This is because if $\alpha = \{ (\alpha_i,m_i)\}$ is an orbit set not including $b$, then by the definition of action and linking for Reeb currents 
\[
\A(\alpha) =  \sum_i m_i \A(\alpha_i), \ \ \ \ \ \ell(\alpha, b) =  \sum_i m_i \ell(\alpha_i,b),
\]
it follows that at least one of the orbits $\alpha_i$ must have 
\[
\frac{\A(\alpha_i)}{\ell(\alpha_i,{b})} \leq \frac{\A(\alpha)}{\ell(\alpha,{b})}.
\]
Proposition \ref{prop:actionlinking} provides us with an admissible Reeb current $\alpha$ for which
\begin{align}
 \A(\alpha) \leq & \ \sqrt{2k ( \vol(\lambda) + \epsilon)} - m, \label{eq:A1} \\
 \ell(\alpha,b) \geq & \ {N_k(p,q)} - m(pq+\Delta). \label{eq:L1}
\end{align}

\noindent \textbf{Step 2:}  Since we are interested in an upper bound on the infimum, we may choose $k$ to be one of the infinitely many for which Lemma \ref{lem:Nk} holds to carry out our computations. In conjunction with \eqref{eq:cklinkbound} we may update the right hand sides of \eqref{eq:A1} and \eqref{eq:L1} to 
\begin{align}
\A(\alpha) \leq & \ \sqrt{2k ( V+ \epsilon)} - m, \label{eqn:A} \\
\ell(\alpha,{b}) \geq & \ {\sqrt{2pqk+\left(\frac{p+q+1}{2}\right)^2} \ -\frac{(p+q+1)}{2}} - m(pq+\Delta). \label{eqn:ell}
\end{align}

\noindent \textbf{Step 3:} They key step in the proof is to establish that 
\begin{equation}\label{eq:ratio}
\frac{\A(\alpha)}{\ell(\alpha,b)} \leq  \frac{\sqrt{2k ( V + \epsilon)}}{{\sqrt{2pqk+\left(\frac{p+q+1}{2}\right)^2} \ -\frac{(p+q+1)}{2}}} .
\end{equation}
This follows if the ratio of the right hand sides of \eqref{eqn:A} over \eqref{eqn:ell} is maximized by setting $m=0$.\footnote{Observe that $\frac{a-mb}{c-md} \leq \frac{a}{c}$ if $m \geq 0$, $ad\leq bc$, and $c-md > 0$.} 

First, we note that  there exists a $k_0$ such that for $k>k_0$, the right hand side of \eqref{eqn:ell} is positive.  By \eqref{eqn:A} and our assumption that
\begin{equation}\label{eq:Vbound}
V < \frac{pq}{(pq + \Delta)^2} <  \frac{1}{pq + \Delta},
\end{equation}
it follows that there is a constant $h<1$ independent of $k$ such that $m < h\sqrt{\frac{2k}{pq+\Delta}}$.

It remains to verify that for sufficiently large values of $k$,
\begin{equation}\label{eq:lhsrhs}
\underbrace{\left(\sqrt{2k (V + \epsilon)}\right)(pq+\Delta)}_\textrm{(LHS)} \leq \underbrace{{\sqrt{2pqk+\left(\tfrac{p+q+1}{2}\right)^2} \ -\tfrac{(p+q+1)}{2}} }_\text{(RHS)}.
\end{equation}
To do so, we divide both sides by $\sqrt{2k}$, and continue to denote them by (LHS) and (RHS). Using our assumption on the volume $V$ as in \eqref{eq:Vbound} and requiring that 
\begin{equation}\label{eq:ep}
\epsilon \leq \frac{pq}{(pq+\Delta)^2} - V
\end{equation}
yields 
\[
\begin{split}
\mbox{(LHS)} \ & \leq  \ (pq + \Delta) \sqrt{V + \epsilon}, \\ 
& \leq \sqrt{pq}.
\end{split}
\]
We have 
\[
\mbox{(RHS)}   \geq  \sqrt{pq + \tfrac{(p+q+1)^2}{8k}} - \tfrac{(p+q+1)}{2\sqrt{2}k},\]
 thus for sufficiently large values of $k$, \eqref{eq:lhsrhs} holds.

 \textbf{Step 4:} Since $k$ can be arbitrarily large, taking a sequence of $k \to \infty$ such that Lemma \ref{lem:Nk} holds, and using our bound \eqref{eq:Vbound} on $V$,  yields
  \[
\inf \left\{ \frac{\A(\alpha)}{\ell(\alpha,{b})} \ \bigg \vert \ \alpha \in \mathcal{O}_0(\lambda) \right\} \leq \sqrt{\frac{V + \epsilon }{pq}}.
 \]
 from \eqref{eq:ratio}. Since $\epsilon>0$ can be arbitrarily small, this proves \eqref{eq:desire}.

\end{proof}

\begin{remark}\label{rmk:Vpqd} The hypothesis $V<pq/(pq+\Delta)^2$ is stronger than the analogue of the hypothesis $V<\rot(B)^{-1}$ of \cite[Prop.~2.2]{HuMAC}, and it is the reason we require the hypothesis $\V(\psi)<pq\cdot\theta_0^2$ in Theorem \ref{thm:Calabi}. It is required to prove the key step (\ref{eq:lhsrhs}) in the proof above. The issue is that the lower bound (\ref{eqn:ell}) underlying (RHS) in (\ref{eq:lhsrhs}) comes from (\ref{eq:cklinkbound}) in which we ignore the contribution to the knot filtration from $\rot(T(p,q))$, which we may do by (\ref{eq:cklinkth}). This is necessary because we cannot compute the knot filtration for arbitrarily large $k$ given a fixed $\Delta$, as discussed in Remark \ref{rem:delta} if we set $\delta=\Delta$. We strengthen our relatively weaker lower bound on $\ell(\alpha,b)$ in (\ref{eqn:ell}) to what we need to prove (\ref{eq:lhsrhs}) using the stronger hypothesis $V<pq/(pq+\Delta)^2$.
\end{remark}

\begin{remark}\label{rmk:deltadenom} If we were lucky enough to know that we could obtain the estimate (\ref{eq:A1}) on the ECH capacity $c_k$ in terms of the contact volume using $k$ satisfying an upper bound, and furthermore if that bound were stronger than those discussed in Remark \ref{rem:delta} with $\Delta=\delta$ (or potential others obtained by a clever choice of $\lambda$ used to compute $c^{\op{link}}_k(T(p,q),pq+\Delta)$), then using Lemma \ref{lem:Nk}(ii) we could improve (\ref{eq:ratio}) to
\[
\frac{\A(\alpha)}{\ell(\alpha,b)}\leq\frac{\sqrt{2k(V+\epsilon)}}{\sqrt{2pqk+\left(\frac{p+q+1}{2}\right)^2}-\frac{p+q+1}{2}+\Delta\left(\sqrt{\frac{2k}{pq}}-1\right)},
\]
which, for $k$ large enough relative to $p$ and $q$ (a lower bound on $k$, but in the hypothetical scenario of this remark, one we are assuming is satisfied even in light of our competing upper bounds), implies
\[
\frac{\A(\alpha)}{\ell(\alpha,b)}\leq\sqrt{\frac{V}{pq+\Delta}}.
\]
In this scenario, we also don't require the stronger hypothesis $V<pq/(pq+\Delta)^2$, but instead can use the weaker hypothesis $V<1/(pq+\Delta)$; the proof of Theorem \ref{thm:Calabi} in \S\ref{ss:calabi} could be used with $\varepsilon(1)\to\V(\psi)-\theta_0$ and we could prove that the infimum of the mean action of $\psi$ is at most $\V(\psi)$. Unfortunately, our $\lambda$ satisfies very few restrictions, thus we know nothing about how quickly in $k$ we can assume a bound of $\epsilon$ on the difference $c_k(S^3,\lambda)-\sqrt{2\vol(\lambda)}$. Results on the subleading asymptotics also do not help, as again, they do not bound $k$.
\end{remark}

\noindent \textsc{Jo Nelson \\   Rice University \\}
{\em email: }\texttt{jo.nelson@rice.edu}\\

\noindent \textsc{Morgan Weiler \\ University of California, Riverside}\\
{\em email: }\texttt{morgan.c.weiler@ucr.edu}\\

\end{document}